\documentclass{amsart}
\usepackage{amssymb}
\usepackage{epsfig}

\newtheorem{thm}{Theorem}[section]
\newtheorem{cor}[thm]{Corollary}
\newtheorem{lem}[thm]{Lemma}
\newtheorem{prop}[thm]{Proposition}
\theoremstyle{definition}
\newtheorem{defn}[thm]{Definition}
\theoremstyle{remark}
\newtheorem{rem}[thm]{Remark}

\numberwithin{equation}{section}
\newcommand{\Z}{\mathbb Z}

\newcommand {\exit}{\textrm {exit}}

\newcommand {\eps}{\varepsilon}
\newcommand {\Fin}{f_{\mid\Lambda}}

\begin{document}

\title[Diffusion along transition chains of invariant tori]
{Diffusion along transition chains of invariant tori and Aubry-Mather sets}
%%%{Existence of orbits  that follow transition  chains of invariant tori and cross Birkhoff zones of instability with Lipschitz boundaries}

\author[Marian Gidea]{Marian Gidea$^\dag$}

\address{School of Mathematics, Institute for Advanced Study, Princeton, NJ 08540, USA, and Department of Mathematics, Northeastern Illinois University, Chicago, IL 60625, USA}
\email{mgidea@neiu.edu}%

\author{Clark Robinson}%
\address{Department of Mathematics, Northwestern University,  Evanston, IL
60208, USA}%
\email{clark@math.northwestern.edu}

\thanks{$^\dag$Research partially supported by NSF grant: DMS
0601016 and   DMS 0635607.} %Index theory, Morse-Conley indices

\thanks{}
%\date{}%
%\dedicatory{}%
%\commby{}%
% ----------------------------------------------------------------
\begin{abstract}
We describe a topological mechanism for the existence of diffusing orbits in a dynamical system satisfying the following assumptions: (i) the phase space contains a normally hyperbolic invariant manifold diffeomorphic to a two-dimensional annulus, (ii) the restriction of the dynamics to the annulus is an area preserving monotone twist map, (iii) the annulus contains sequences of invariant one-dimensional tori that form transition chains, i.e., the unstable manifold of each torus has a topologically transverse intersection with the stable manifold of the next torus in the sequence, (iv) the transition chains of tori are interspersed with gaps created by resonances, (v) within each gap there is prescribed a finite collection of Aubry-Mather sets. Under these assumptions, there exist trajectories that follow the transition chains, cross over the  gaps, and follow the Aubry-Mather sets within each gap, in any specified order. This mechanism is related to the Arnold diffusion problem in Hamiltonian systems. In particular,  we prove the existence of diffusing trajectories in the large gap problem of Hamiltonian systems. The argument is topological and constructive.
\end{abstract}
%-----------------------------------------------------------------
\maketitle
% ----------------------------------------------------------------

\section{Introduction}\label{section:introduction}
In this paper  we study  a topological mechanism of diffusion and chaotic orbits related to the Arnold diffusion problem. We consider a  normally hyperbolic invariant manifold diffeomorphic to a two-dimensional annulus. We assume that the dynamics restricted to the annulus is given by an area preserving monotone twist map.  We assume that inside the annulus there exist invariant primary one-dimensional tori (homotopically nontrivial invariant closed curves)  with the property that the unstable manifold of each  torus has   topologically transverse  intersections with the stable manifolds of all sufficiently close tori.
Sequences of such tori and their heteroclinic connections form   transition chains of tori. The successive transition chains of tori are  interspersed with gaps. We assume that each gap is a region in the annulus bounded by two invariant primary  tori
of one of the following types: (i) a  Birkhoff Zone of Instability, i.e., it contains no invariant primary torus in its interior (ii) it contains only  finitely many invariant primary tori in its interior, such that each torus is either isolated, or it consists of a  hyperbolic periodic orbit together with its stable and unstable manifolds that are assumed to coincide.
Within each gap we prescribe a finite collection of Aubry-Mather sets.
We prove the existence of orbits that shadow the primary tori in each transition chain, cross over the gaps that separate the successive transition chains, and also shadow the specified  Aubry-Mather sets within each gap.

The motivation for this result is the Arnold diffusion problem of Hamiltonian systems. This problem asserts that all sufficiently small perturbations of generic, integrable Hamiltonian systems exhibit orbits  along which the action variable changes  substantially; also, there exist  chaotic orbits that can be coded through symbolic dynamics.
A classification of nearly integrable systems proposed  in \cite{ChierchiaG94} distinguishes between \textit{a priori} stable systems, in which the unperturbed system can be expressed in terms of action-angle variables only, and \textit{a priori} unstable systems, in which the unperturbed system contains both action-angle  and hyperbolic variables. In the \textit{a priori} stable case analytical results have been announced in \cite{Mather03}. In the a priori unstable case there have been several analytical results  in the last several years (see \cite{ChierchiaG94,Xia98,Treschev02b,Treschev04,ChengY2004,DelshamsLS2006,Bernard08,ChengY2009}). Some of the approaches involve variational methods,   geometric methods, or topological methods. See \cite{DelshamsGLS2008} for an overview on the Arnold diffusion problem, applications, and additional references.

It is relevant in applications to detect, combine, and compare different  mechanisms of diffusion displayed by concrete systems. In many models, as well as in numerical experiments, diffusion can only be observed  for perturbations of sizes much larger than those considered by the analytical approaches  \cite{Chirikov79,Lieberman83,Lichtenberg92,Laskar93,GuzzoLF05}. For these types of problems, the geometric and topological approaches  are particularly advantageous as  they yield constructive methods to detect diffusion, quantitative estimates on diffusing orbits, and explicit conditions that can be verified in concrete examples.

In this paper we describe a general method  to establish the existence of diffusing orbits for a large class of dynamical systems. The dynamical systems under consideration are not assumed to be small perturbations of integrable Hamiltonians. Moreover, some   systems that are not Hamiltonian can be considered. Our method requires the existence of certain geometric objects that organize the dynamics, and  employs topological tools to establish the existence of diffusing orbits. The existence of the geometric objects can be verified in concrete systems through analytical methods, or through numerical methods, or through a combination of thereof.

We illustrate our method in the case of a Hamiltonian system consisting of a pendulum and a rotator with a small periodic coupling. We give an analytic argument for the existence of diffusing orbits. In \cite{DelshamsGR10} the topological method is applied to show, with the aid of a computer, the existence of diffusing orbits in the spatial restricted three-body problem, where the two primaries are the Sun and the Earth; this model is not nearly integrable. Similar ideas appear in \cite{CapinskiZ11,GalanteK11a,GalanteK11b}.

The diffusing orbits detected by this approach follow transition chains of invariant tori up to the gaps between these transition chains, and then cross the gaps following the inner dynamics restricted to the annulus.
The orbits that cross the gaps  follow Birkhoff connecting orbits that go from one boundary of the gap to the other, or Mather connecting orbits, that shadow a prescribed sequence of Aubry-Mather set inside the gap, or homoclinic orbits.

Related approaches, that use Birkhoff's ideas to understand generic drift properties of random iterations  of exact-symplectic twist maps, appear in \cite{Moeckel2002,LeCalvez2007}.
These yield  the concept of a polysystem, which is a locally constant skew-product  over a Bernoulli
shift \cite{Marco2008}. Polysystems are a key ingredient in the approach of the instability problem in  \textit{a priori}  unstable systems in  \cite{Bounemoura}. Arnold diffusion-type orbits for  random iterations  of  flows of families of Tonelli Hamiltonians are studied in \cite{Mandorino}, based on  an extension of the pseudograph method from \cite{Bernard08}.

The diffusing orbits that we obtain in this paper are similar to those found through variational methods, as  in \cite{ChengY2004,ChengY2009}; they are however not necessarily  action minimizing. Also, we do not need  the unperturbed Hamiltonian to have positive-definite normal torsion.

This paper  completes the investigations undertaken in \cite{GideaR07,GideaR09} where some non-generic conditions, or some conditions  difficult to verify in concrete systems, were assumed.
The present paper assumes very general conditions and provides a new topological mechanism of diffusion based on Aubry-Mather sets.

\section{Main result}\label{sec:mainresult}

In this section we state the assumptions and the main result  of this paper.  After the statement, the assumptions and the main result  are explained and exemplified.
\begin{itemize}
\item [(A1)]  $M$ is a $n$-dimensional $C^r$-differentiable Riemannian manifold, and $f:M\to M$ is a $C^r$-smooth map, for some $r\geq 2$.
\item[(A2)] There exists a 2-dimensional submanifold $\Lambda$ in $M$, diffeormorphic to an
annulus $\Lambda\simeq \mathbb{T}^1\times[0,1]$.  We assume that $f$ is  $2$-normally hyperbolic to $\Lambda$ in $M$ (see Subsection
\ref{section:scattering} for the definition).
Since $f$ is  $2$-normally hyperbolic to $\Lambda$, $W^s(\Lambda)$, $W^u(\Lambda)$, and
$\Lambda$ are all $C^2$-differentiable.
Denote the dimensions of the stable and unstable manifolds of a point $x \in \Lambda$ by
$\text{dim}(W^s(x)) = n_s$ and $\text{dim}(W^u(x)) = n_u$.
Then, $n = 2 + n_s + n_u$.
\item[(A3)]  On $\Lambda$  there is a system of angle-action coordinates $(\phi,I)$, with
$\phi \in {\mathbb T}^1$ and $I \in [0,1]$.
The restriction $f|_{\Lambda}$ of $f$ to $\Lambda$ is a boundary component preserving,  area preserving, monotone twist map, with respect to the angle-action coordinates $(\phi,I)$.
\item [(A4)] The stable and unstable manifolds of $\Lambda$, $W^s(\Lambda)$ and $W^u(\Lambda)$, have a differentiably transverse intersection along a $2$-dimensional homoclinic channel $\Gamma$. Since the manifolds $W^s(\Lambda)$ and $W^u(\Lambda)$ are $C^2$ and transverse,
$\Gamma$ is $C^2$.   We assume that the scattering map  $S$   associated to $\Gamma$ is well defined, and
 hence is   $C^1$.
See  Subsection \ref{section:scattering}.
    \item [(A5)]
    There exists a bi-infinite sequence of Lipschitz primary invariant tori $\{T_i\}_{i\in\mathbb{Z}}$ in $\Lambda$, and
    a bi-infinite, increasing sequence of integers $\{i_k\}_{k\in\mathbb{Z}}$ with the following properties:
    \begin{itemize}
\item[(i)] Each torus $T_i$ intersects the domain $U^-$ and the range $U^+$ of the scattering map $S$ associated to $\Gamma$.
\item [(ii)] For each $i\in\{{i_{k}+1}, \ldots, {i_{k+1}-1}\}$, the image of $T_i\cap U^-$ under the scattering map $S$  is topologically transverse to  $T_{i+1}$.
\item [(iii)] For each torus $T_i$ with $i\in \{{i_{k}+2}, \ldots, {i_{k+1}-1}\}$, the  restriction of $f$ to $T_i$ is topologically transitive.
\item[(iv)] Each torus $T_i$ with $i\in \{{i_{k}+2}, \ldots, {i_{k+1}-1}\}$, can be
$C^0$-approximated from both sides by other primary invariant tori from
$\Lambda$.
\end{itemize}
We will refer to a finite sequence $\{T_i\}_{i=i_k+1,\ldots, i_{k+1}}$  as above as a transition chain of tori.
\item [(A6)]  The region in $\Lambda$ between $T_{i_k}$ and $T_{i_{k}+1}$ contains no invariant primary torus in its interior.
\item[(A7)] Inside each region between $T_{i_k}$ and $T_{i_{k}+1}$   there is prescribed  a finite collection of Aubry-Mather sets $\{\Sigma _{\omega^k_1}, \Sigma _{\omega^k_2},\ldots, \Sigma _{\omega^k_{s_k}}\}$, where $s_k\geq 1$, and $\omega^k_s$ denotes the rotation number of $\Sigma_{\omega^k_s}$. Each  Aubry-Mather set $\Sigma_{\omega^k_s}$ is assumed to lie on some essential circle  $C_{\omega^k_s}$, with the circles  $C_{\omega^k_s}$ mutually disjoint for all $s\in \{1,\ldots,s_k\}$, and $C_{\omega^k_s}$ below $C_{\omega^k_{s'}}$ for all $\omega^k_s<\omega^k_{s'}$. The vertical ordering of these circles is relative to the $I$-coordinate on the annulus.
\end{itemize}

Instead of (A6)  we can consider the following condition:
\begin{itemize}
\item [(A6$'$)]  The region $\Lambda_k$ in $\Lambda$ between $T_{i_k}$ and $T_{i_{k}+1}$  contains finitely many invariant primary  tori $\{\Upsilon_{h^k_1}, \ldots , \Upsilon_{h^k_{l_k}}\}$, where $l_k\geq 1$, satisfying the following properties:
    \begin{itemize}
    \item [(i)] Each $\Upsilon_{h^k_j}$ falls in one of the following two cases:
    \begin{itemize}
    \item [(a)] $\Upsilon_{h^k_j}$ is an isolated invariant primary torus, i.e., an invariant torus that has a neighborhood in the annulus that does not contain any other invariant  primary torus inside.
    \item [(b)] There exists a hyperbolic periodic orbit in $\Lambda$ such that its stable and unstable manifolds coincide. Each invariant manifold has two branches, an upper branch and a lower branch (where the vertical ordering  is relative to the $I$-coordinate of the annulus). Then $\Upsilon_{h^k_j}$ is an invariant primary torus consisting of the hyperbolic periodic orbit together with the upper branches of the invariant manifolds, or consisting of the hyperbolic periodic orbit together with the lower branches of the invariant manifolds.
    \end{itemize}
    \item [(ii)]  The invariant primary  tori $\{\Upsilon_{h^k_1}, \ldots , \Upsilon_{h^k_{l_k}}\}$ are vertically ordered, in the sense that $\Upsilon_{h^k_j}$ is below $\Upsilon_{h^k_{j+1}}$, for all $j=1,\ldots, l_k-1$. The vertical ordering of these tori is relative to the $I$-coordinate on the annulus.
    \item [(iii)] For each $\Upsilon_{h^k_j}$,  $j=1,\ldots, l_k$,
       the inverse image $S^{-1}(\Upsilon_{h^k_j}\cap U^+)$ forms with $\Upsilon_{h^k_j}$  a topological disk  $D_{h^k_j} \subseteq U^-$ below $\Upsilon_{h^k_j}$, such that $S(D_{h^k_j})\subseteq U^+$ is a topological disk  above $\Upsilon_{h^k_j}$, which is bounded by $\Upsilon_{h^k_j}$ and  $S(\Upsilon_{h^k_j}\cap U^-)$. See Fig. \ref{fig:a6prime}.
     \end{itemize}
\end{itemize}

Now we state the main result of the paper.
\begin{thm}\label{thm:main1}
Let   $f:M\to M$  be a $C^r$-differentiable map, and let  $(T_i)_{i\in\mathbb{Z}}$ be a sequence of invariant primary
tori in $\Lambda$, satisfying the properties
(A1) -- (A6), or (A1)-(A5) and (A6\,$'$),  from above. Then for each  sequence
$(\epsilon_i)_{i\in\mathbb{Z}}$ of positive real numbers,   there exist a point $z\in M$
and a bi-infinite increasing sequence   of integers
$(N_i)_{i\in\mathbb{Z}}$  such that
\begin{eqnarray}\label{eqn:main11}  d(f^{N_i}(z), T_{i})<\epsilon _i, \textrm { for all }
i\in\mathbb{Z}.\end{eqnarray}

In addition, if  condition (A7) is assumed, and  some  positive integers $\{n^k_s\}_{s=1,\ldots, s_k}$, $k\in\mathbb{Z}$ are given, then there exist $z\in M$ and $(N_i)_{i\in\mathbb{Z}}$ as in \eqref{eqn:main11}, and positive integers $\{m^k_s\}_{s=1,\ldots, s_k}$, $k\in\mathbb{Z}$,  such that, for each $k$ and each $s\in\{1,\ldots, s_k\}$, we have
\begin{equation}\label{eqn:main22}\pi_\phi(f^j(w^k_s))<\pi_\phi(f^j(z))<\pi_\phi(f^j(\bar w^k_s)),\end{equation}
for some $w^k_s,\bar w^k_s\in\Sigma_{\omega^k_s}$ and for all $j$ with \[N_{i_k}+\sum_{t=0}^{s-1} n^k_t+\sum_{t=0}^{s-1}m^k_t\leq j\leq N_{i_k}+\sum_{t=0}^{s} n^k_t+\sum_{t=0}^{s-1}m^k_t.\]
\end{thm}

\begin{figure} \centering
\includegraphics*[width=0.7\textwidth, clip, keepaspectratio]
{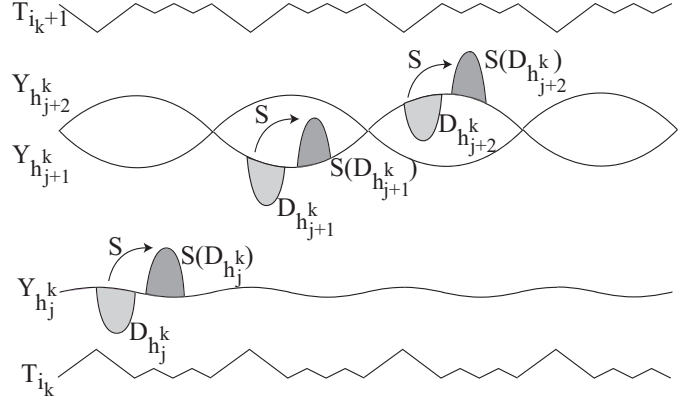}
    \caption[]{An illustration  of the condition (A6$'$). In the  figure, inside the region between $T_{i_k}$ and $T_{i_{k}+1}$, three invariant primary tori are shown, $\Upsilon_{h^k_j}$, $\Upsilon_{h^k_{j+1}}$, $\Upsilon_{h^k_{j+2}}$, where  $\Upsilon_{h^k_j}$  is as in (A6$'$-i-a), and $\Upsilon_{h^k_{j+1}}$, $\Upsilon_{h^k_{j+2}}$ are as in (A6$'$-i-b). The vertical ordering of the tori stated condition (A6$'$-ii) is depicted.  For each of the tori $\Upsilon_{h^k_j}$, $\Upsilon_{h^k_{j+1}}$, $\Upsilon_{h^k_{j+2}}$, the condition (A6$'$-iii) is also illustrated.}
    \label{fig:a6prime}
\end{figure}

Theorem \ref{thm:main1} asserts that if the conditions (A1)-(A6), or (A1)-(A5) and (A6$'$) are satisfied, then there exists an orbit that shadows all tori in the transition chains in the prescribed order, and also crosses over the large gaps that  separate the successive transition chains. In particular,  there exists an orbit that travels arbitrarily far with respect to the action variable, and there also exists  an orbit that executes chaotic excursions.
Additionally, if some  Aubry-Mather sets are prescribed inside each gap  that  separates the successive transition chains, as in condition (A7), then there exists an orbit  that, besides shadowing  the  transition chains, it  also shadows the Aubry-Mather sets in the prescribed order. Note that the tori in the transition chains are shadowed in the sense that the diffusing orbit gets arbitrarily close to these tori. However, the Aubry-Mather sets are shadowed in the sense of the cyclical ordering: for each prescribed  Aubry-Mather set,  the diffusing orbit stays between the orbits of two points in the Aubry-Mather set, relative to the $\phi$-coordinate, for any prescribed time interval.

Now we explain  each assumption.

Assumption (A1) describes a $C^r$-differentiable, discrete dynamical system.
In applications, the map $f$ represents the first return map to a Poincar\'e section associated to a flow. In many examples of interest the flow is a Hamiltonian flow.

Assumption (A2) prescribes  the existence of a normally hyperbolic invariant manifold $\Lambda$ for $f$, which is diffeomorphic to an annulus.
The differentiability class $r$  and the contraction  and expansion rates along the stable and unstable bundles on   $\Lambda$ are chosen  so that the manifolds $\Lambda$, $W^u(\Lambda)$, $W^s(\Lambda)$, $\Gamma$ are at least $C^2$-differentiable, and the scattering map $S$ associated to $\Gamma$ is at least $C^1$-differentiable. The relations between the  rates and the differentiability of these objects is given explicitly in Subsection \ref{section:scattering}.

In many examples of nearly integrable Hamiltonian systems, e.g. \cite{Arnold64}, one can identify   a normally hyperbolic invariant manifold $\Lambda_0$ in the unperturbed system, and use the standard theory of normal hyperbolicity
to establish the persistence of a normally hyperbolic invariant manifold $\Lambda_\eps$ diffeomorphic to $\Lambda_0$ for the perturbed system, for all sufficiently small perturbation parameters $\eps\neq 0$.
There also exist examples, e.g. from celestial mechanics, were the existence of a normally hyperbolic manifold can be established through a computer assisted proof (see \cite{Capinski09}).
We note that the assumption (A2) does not require that the stable and unstable manifolds of $\Lambda$ have equal dimensions, thus our setting includes dynamical systems that are not Hamiltonian.

Assumption (A3) is satisfied  automatically in examples like the weakly-coupled pendulum-rotator system considered in \cite{DelshamsLS2006}, or the periodically perturbed geodesic flow on a torus considered in \cite{DelshamsLS00,Kaloshin2003}. Some properties of area preserving, monotone twist maps of the annulus are reviewed in Section~\ref{section:aubrymathersets}.

Assumption (A4) asserts that the stable and unstable manifolds of $\Lambda$ have a transverse intersection $\Gamma$ along a homoclinic manifold $\Gamma$, such that the scattering map associated to $\Gamma$ is well defined, and hence is  $C^1$. See Subsection \ref{section:scattering}. In perturbed  systems one often uses a Melnikov method to establish the existence, and the persistence for all sufficiently small values of the perturbation, of a transverse intersection of the invariant manifolds.

Assumption (A5) prescribes the existence of a bi-infinite collection  $\{T_{i}\}_{i\in\mathbb{Z}}$ of invariant primary Lipschitz tori that can be grouped into transition chains of the type $\{T_i\}_{i=i_k+1,\ldots,i_{k+1}}$.

Assumption (A5)-(i) requires that each torus $T_i$ intersects the domain and the range of the scattering map.

Assumption (A5)-(ii) requires  that each torus in a transition chain\break $\{T_i\}_{i=i_k+1,\ldots,i_{k+1}}$ is mapped by the scattering map topologically transversally across the next torus in the chain.  Since the tori are only Lipschitz, topological transversality (topological crossing)   is used in place of differentiable transversality. The definition of topological crossing can be found in \cite{BurnsW95}.  Roughly speaking, two manifolds are topologically crossing if they can be made differentiably transverse with non-zero  oriented intersection number by the means
of a sufficiently small homotopy. In our case the two manifolds are two $1$-dimensional arcs of $S(T_i)$ and of $T_{i+1}$ in $\Lambda$. From (A5)-(ii), it follows as in \cite{DelshamsLS08a}  that $W^u(T_i)$ has a topologically transverse  intersection point with $W^s(T_{i+1})$. Condition (A5)-(iii) requires that each torus in $\{T_{i}\}_{i=i_{k}+1, \ldots, i_{k+1}}$, except for the end tori, are topologically transitive.
Assumption (A5)-(iv) says that all  tori in the transition chain except for the end ones  can be $C^0$-approximated from both ways by some other invariant primary tori in $\Lambda$, not necessarily from the transition chain. This means that for each $i\in\{i_{k}+2, \ldots, i_{k+1}-1\}$ there exist two sequences of invariant primary tori
$(T_{j^-_l(i)})_{l\geq 1},(T_{j^+_l(i)})_{l\geq 1}$ in $\Lambda$ that approach $T_i$ in
the $C^0$-topology, such that the annulus bounded by $T_{j^-_l(i)}$
and $T_{j^+_l(i)}$ contains $T_i$ in its interior for all $l$.

Assumption (A6) says that every pair of successive transition chains\break   $\{T_{i}\}_{i=i_{k-1}+1, \ldots, i_{k}}$ and $\{T_{i}\}_{i=i_{k}+1, \ldots, i_{k+1}}$ is separated by the region between $T_{i_{k}}$ and $T_{i_{k}+1}$ which contains no invariant primary torus in its interior. A region in an annulus that is bounded by two invariant primary tori and contains no invariant primary torus in its interior is referred as a Birkhoff Zone of Instability (BZI). The boundary tori have in general only Lipschitz regularity.

Assumptions (A5) and (A6) describe a geometric structure that is typical for the large gap problem for \textit{a priori} unstable Hamiltonian systems. In such systems, Melnikov theory implies that $W^u(\Lambda)$ intersects transversally  $W^s(\Lambda)$ at an angle of order $\eps$, where $\eps$ is the size of  the perturbation. The KAM theorem yields a Cantor family of smooth, invariant primary tori that survives the perturbation. The family of tori is interrupted by `large gaps' of order  $\eps^{1/2}$ located at the  resonant regions. Using the transverse intersection between $W^u(\Lambda)$ and $W^s(\Lambda)$, one can find heteroclinic connections between KAM tori that are sufficiently close, within order $\eps$, from one another, and thus form transition chains of tori. Since the large gaps are of order $\eps^{1/2}$ and the splitting size of $W^u(\Lambda)$, $W^s(\Lambda)$ is only order $\eps$,  the transition chain mechanism cannot be extended across the large gaps. In our model, the large gaps  are modeled by BZI's, as in assumption (A6).  This can be achieved by extending the transition chains to maximal transition chains, that go from the boundary of one large gap to the boundary of the next large gap. The intermediate tori in the chain can be chosen as  KAM tori: therefore the assumption that these tori are topologically transitive and are $C^0$-approximable from both sides by other tori is satisfied in such cases. This may not be the case for the  tori at the ends of the transition chains.

Assumption (A7) says that inside each BZI between $T_{i_k}$ and $T_{i_{k}+1}$ there is a  prescribed collection of Aubry-Mather sets $\{\Sigma _{\omega^k_1}, \Sigma _{\omega^k_2},\ldots, \Sigma _{\omega^k_{s_k}}\}$ that is vertically ordered. The vertical ordering means that the Aubry-Mather sets lie on essential (non-invariant) circles $C_{\omega^k_s}$ that are graphs over the $\phi$-coordinate of the annulus, and with
 $C_{\omega^k_s}$ below $C_{\omega^k_{s'}}$ provided $\omega^k_s<\omega^k_{s'}$; we write $C_{\omega^k_s}\prec C_{\omega^k_{s'}}$. The vertical ordering of the Aubry-Mather sets is shown for example in \cite{Gole01}.

Assumption (A6$'$) is a relaxation of (A6).  Instead of requiring that the region in $\Lambda$ between $T_{i_{k}}$ and $T_{i_{k}+1}$ is a BZI, it allows the existence of finitely many  invariant primary tori $\{\Upsilon_{h^k_j}\}_{j=1, \ldots , l_k}$ that separate the region into disjoint components. These invariant primary tori are either isolated or else they consist of  hyperbolic periodic  points together with branches of their stable and unstable manifolds which are assumed to coincide.
We require that the image of each $\Upsilon_{h^k_j}$ under $S$  satisfies a certain transversality condition with $\Upsilon_{h^k_{j+1}}$ that allows one to use the scattering map in order to move points from one side of the set to the other side of the set. We note that isolated invariant tori and
hyperbolic periodic points whose stable and unstable manifolds coincide do not occur in generic systems.

\section{Application}

We apply Theorem \ref{thm:main1} to show  the existence of diffusing orbits in an example of a nearly integrable Hamiltonian system. Let
\begin{eqnarray}\label{eqn:hamiltonian}H_\eps(p, q, I,\phi,
t)&=&h_0(I)\pm\left(\frac{1}{2}p^2+V(q)\right)+\eps h(p, q,I,\phi, t; \eps),
\end{eqnarray}
where $(p, q, I,\phi, t)\in\mathbb{R}\times \mathbb{T}^1\times
\mathbb{R}\times\mathbb{T}^1\times\mathbb{T}^1$ and $h$ is a trigonometric polynomial in $(\phi,t)$.
Here $h_0(I)$ represents a rotator, $P_{\pm}(p,q)=\pm(\frac{1}{2}p^2+V(q))$ represents a pendulum, and $\eps h$ a small, periodic coupling.
We assume that $V$, $h_0$ and $h$ are uniformly $C^r$ for some $r$ sufficiently large.
We  assume that  $V$ is periodic in $q$ of
period $1$ and has a unique non-degenerate global maximum; this implies
that the pendulum has  a   homoclinic orbit $(p^0 (\sigma),q^0 (\sigma))$ to $(0,0)$, with  $\sigma\in\mathbb{R}$.
We also assume  that
$h_0$ satisfies a uniform twist condition $\partial^2
h_0/\partial I^2>\theta$, for some $\theta>0$, and for all $I$ in
some interval $(I^-, I^+)$, with $I^-<I^+$ independent of $\eps$.

The Melnikov potential for the homoclinic orbit $(p^0 (\sigma),q^0 (\sigma))$ is defined by
\begin{eqnarray*}
\mathcal{L}(I,\phi,t)=&\displaystyle-\int_{-\infty}^{\infty}&\left
[h(p^0(\sigma),q^0(\sigma),I,\phi+\omega(I)\sigma,t+\sigma;0)\right .\\
& &\left .-h(0,0,I,\phi+\omega(I)\sigma,t+\sigma;0)\right ]d\sigma,
\end{eqnarray*}
where $\omega(I)=(\partial h_0/\partial I)(I)$.

We assume the following non-degeneracy conditions on the Melnikov
potential:
\begin{itemize}
\item[(i)] For each $I\in(I^-,I^+)$, and each
$(\phi, t)$ in some open set in $\mathbb{T}^1\times \mathbb{T}^1$, the
map
\[\tau\in\mathbb{R} \to
\mathcal{L}(I,\phi-\omega(I)\tau,t-\tau)\in\mathbb{R}\] has a non-degenerate
critical point $\tau^*$, which can be parameterized as
\[\tau^*=\tau^*(I,\phi,t).\]

\item[(ii)] For each $(I,\phi, t)$ as above, the
function\[(I,\phi,t)\to
\frac{\partial{\mathcal{L}}}{\partial\phi}(I,\phi-\omega(I)\tau^*,t-\tau^*)
\] is non-constant, negative in the case of  $P_{-}$,  and positive in the case of $P_{+}$. \end{itemize}

This example and the above conditions are considered in \cite{DelshamsLS2006}. There are some additional  non-degeneracy conditions on $h$ and $\partial h/\partial \eps$ that are required in \cite{DelshamsLS2006}; we do not need to assume those conditions here.

Now we verify the conditions (A1)-(A6) from Section \ref{sec:mainresult} for  this model. We will rely heavily on the estimates from \cite{DelshamsLS2006}.

\textit{Condition (A1).} The time-dependent Hamiltonian in \eqref{eqn:hamiltonian} is transformed into an autonomous Hamiltonian by introducing a new variable $A$, symplectically conjugate with $t$ obtaining  \begin{equation}\label{eqn:hamiltonianex} \tilde H_\eps(p, q, I,\phi,
A,t)=h_0(I)\pm(\frac{1}{2}p^2+V(q))+A+\eps h(p, q,I,\phi, t; \eps),\end{equation}
where $(p, q, I,\phi, A,t)\in(\mathbb{R}\times \mathbb{T}^1)^3$.
We  fix an energy manifold $\{\tilde H_\eps=\tilde h\}$ for some $\tilde h$, and restrict to the Poincar\'e section $\{t=1\}$ for the Hamiltonian flow. The resulting  manifold is a $4$-dimensional manifold $M_\epsilon$ parametrized by some coordinates $(p_\eps,q_\eps,I_\eps,\phi_\eps)$.
The first return map to $M_\eps$ of the Hamiltonian flow is a $C^r$-differentiable map $f_\epsilon$.

\textit{Condition (A2).} In the unperturbed case $\eps=0$, the manifold \[\Lambda_0:=\{(p,q,I,\phi)\,|\, p=q=0 \}\] is a normally hyperbolic invariant manifold for $f_0$. The dynamics on $\Lambda_0$ is given by an integrable twist map, and $\Lambda_0$ is foliated by invariant $1$-dimensional tori. For the perturbed system, $\Lambda_0$ can be continued to a manifold $\Lambda_\eps$ diffeomorphic to $\Lambda_0$,  that  is locally invariant for the perturbed flow for all $\eps$ sufficiently small.
The theory of normal hyperbolicity \cite{HirschPS77} can be used as in \cite{DelshamsLS2006} to show that $\Lambda_\eps$   can be extended so that it is a
normally hyperbolic invariant manifold for the flow.   Each point in $\Lambda_\eps$ has $1$-dimensional stable and unstable manifolds. The regularity of $h_0$ and the uniform twist condition allows one to apply the KAM theorem and conclude the existence of a KAM family of primary invariant tori in $\Lambda_\eps$ that survives the perturbation, for all $\eps$ sufficiently small.

\textit{Condition (A3).} The map $f_\eps$ is symplectic. This plus the twist condition on $h_0$ implies that $f_\eps$ restricted to $\Lambda_\eps$ is an area preserving, monotone twist map.

\textit{Condition (A4).} The non-degeneracy conditions on the Melnikov function imply that $W^u(\Lambda_\eps)$ and $W^s(\Lambda_\eps)$ have a transverse intersection along a homoclinic manifold $\Gamma_\eps$, provided  $\eps$ is sufficiently small. Moreover, it is shown in \cite{DelshamsLS2006,DelshamsLS08a} that by restricting to some convenient homoclinic manifold $\Gamma_\eps$ one can ensure that the maps $\Omega_\eps^\pm:\Gamma_\eps\to\Lambda_\eps$ are diffeomorphisms onto their images; thus the scattering map $S_\eps: U^- _\eps\to  U^+ _\eps$ is a diffeomorphism between some two open sets $U^-_\eps, U^+_\eps\subseteq \Lambda_\eps$, of size $O(1)$. For $\eps$ fixed to some sufficiently small value, we let $\Gamma:=\Gamma_\eps$ and $S:=S_\eps$.

\textit{Conditions (A5),(A6) and (A6$'$)} The paper \cite{DelshamsLS2006} applies an averaging procedure to reduce the dynamics on $\Lambda_\eps$ to a normal form up to $O(\eps^2)$ away from resonances. The averaging procedure fails within the resonant  regions, corresponding to the values $I_\eps(k,l)$ of the action variable where $k\omega(I)+l=0$. A resonance is said to be of order $j$ if the $j$-th order averaging cannot be applied about the corresponding action level set.

Since $h$ is a trigonometric polynomial, one has to deal with only finitely many resonant regions.
Outside the resonant regions one applies the KAM theorem and obtain KAM tori that are at a distance of order $O(\eps^{3/2})$ from one another. The resonant regions yield  gaps between KAM tori of size $O(\eps^{j/2})$, where $j$ is the  order of the resonance. Only the resonances of order $1$ and $2$ are of interest, as they produce gaps of size $O(\eps)$ and $O(\eps^{1/2})$ respectively. Inside each resonant region, the system can be approximated by a system similar to a pendulum. In such a region, under appropriate non-degeneracy conditions, it is shown that there exist primary KAM tori close to the separatrices of the pendulum, secondary KAM tori (homotopically trivial), and stable and unstable manifolds of hyperbolic periodic orbits that pass close to the separatrices of the pendulum. Moreover, these objects can be chosen to be $O(\eps^{3/2})$ from one another.
In the generic case when the stable and unstable manifolds of hyperbolic periodic orbit intersect transversally, and there are no isolated invariant tori, a resonant region determines a BZI as in (A6). In the non-generic case when the   stable and unstable manifolds of a hyperbolic periodic orbit  coincide, or there exist isolated invariant tori, the resonant region is as described  as in  (A6$'$).
The estimates from \cite{DelshamsLS2006} imply that there exist primary KAM tori that are within $O(\eps^{3/2})$ from the boundaries of the  gap, or to the stable and unstable manifolds of the hyperbolic periodic orbits inside the resonant regions.  These estimates do not allow one to precisely locate the boundaries of the BZI's or to say anything about their dynamics.

The Melnikov conditions imply that the scattering map $S_\eps$
associated to this homoclinic channel $\Gamma_\eps$ can be computed in terms of
the Melnikov potential $\mathcal{L}$. If $S_\eps(x^-)=x^+$, then the change in the $I_\eps$-coordinate  under  $S_\eps$ is given by
\begin{equation}\label{eqn:scatteringjump} I_\eps(x^+)-I_\eps(x^-)=-\eps\frac{\partial \mathcal{L}}{\partial \phi}(I_\eps,\phi_\eps -\omega(I_\eps)\tau^*, t-\tau^*)+O_{C^1}(\eps^{1+\varrho}),\end{equation}
for some $\varrho>0$. Condition (ii) implies that there are points in the domain of the scattering map $S$ whose $I_\eps$-coordinate is increased by $O(\eps)$ under $S$.

We can use these estimates to construct transition chains of invariant primary tori alternating with gaps, as in (A5) and (A6), or as in (A5) and  (A6$'$).  For $\eps$ sufficiently small and fixed, we let $M:=M_\eps$, $f:=f_\eps$, and $\Lambda$ be the annulus in $\Lambda_\eps$ bounded by  a pair of tori $T_{I_a}$, $T_{I_b}$ with $I^-<I_a<I_b<I^+$.

First, we choose a sequence of resonant regions and non-resonant regions that intersect the domain $U^-$ and the range $U^+$ of the scattering map.
Since the KAM primary tori are within $O(\eps^{3/2})$ from one another, and the scattering map makes jumps of order $O(\eps)$ in the increasing direction of $I_\eps$, then we can find smooth KAM primary tori $\{T_{i_{k}+2}, T_{i_{k}+2}, \ldots, T_{i_{k+1}-1}\}$ such that $W^u(T_i)$ has a transverse intersection with $W^s(T_{i+1})$ for all $i\in\{{i_{k}+2}, {i_{k}+3}, \ldots, {i_{k+1}-1}\}$, and that $T_{i_{k}+2}$ and $T_{i_{k+1}-1}$ are within $O(\eps^{3/2})$ from the separatrices of the penduli corresponding to two consecutive resonant gaps of orderer $1$ or $2$.   The dynamics on each such a torus is quasi-periodic, so is topologically transitive. This ensures condition (A5)-(iii). Moreover, we can choose these KAM tori so that they are `interior' to the Cantor family of tori, i.e. they can be approximated from both sides by other KAM primary tori. This ensures condition (A5)-(iv).
To the transition chain $\{T_{i_{k}+2}, T_{i_{k}+3}, \ldots, T_{i_{k+1}-1}\}$ we add, at each end, a torus $T_{i_{k}+1}$ and  a torus $T_{i_{k+1}}$. These end tori bound resonant gaps  that are either BZI's or consist of hyperbolic periodic orbits together with their invariant manifolds. Since $T_{i_{k}+1},T_{i_{k+1}}$ are within $O(\eps^{3/2})$ from $T_{i_{k}+2},T_{i_{k+1}-1}$, respectively, and the scattering map $S_\eps$ makes jumps by order $O(\eps)$, it follows that $S(T_{i_{k}+1})$ topologically crosses $T_{i_{k}+2}$, and $S(T_{i_{k+1}-1})$ topologically crosses $T_{i_{k+1}}$. This ensures condition (A5)-(ii). Condition (A1)-(i) is ensured automatically by our initial choice of the resonant regions and the non-resonant regions so that they intersect the domain $U^-$ and the range $U^+$ of $S_\eps$. The end tori $T_{i_{k}+1}$ and $T_{i_{k+1}}$ are at the boundaries of  two consecutive resonant gaps. This construction is continued for all resonant and non-resonant regions. Thus, for $\eps$ fixed and sufficiently small,  we obtain sequences of tori $\{T_{i_{k}+1}, T_{i_{k}+2}, \ldots, T_{i_{k+1}}\}$ as in (A5), interspersed with gaps between  $T_{i_{k}}$ and $T_{i_{k}+1}$, and also between $T_{i_{k+1}}$ and $T_{i_{k+1}+1}$, as in (A6).

We are under the assumption of  Theorem \ref{thm:main1}. Then there exists a  diffusing orbit that shadows the transition chains of invariant primary tori and crosses the prescribed gaps. In particular, if we choose an initial torus $T_{I_a}$ and a final torus $T_{I_b}$ so that they are $O(1)$ apart, we obtain a diffusing orbit whose action variable changes by $O(1)$. We note our theorem applies even for the choice of the pendulum $P_-$, when the unperturbed Hamiltonian does not have  positive-definite normal torsion. The assumption of positive definiteness seems to be very important for variational methods.

We emphasize that, although we are using many of the estimates from  \cite{DelshamsLS2006}, we obtain a  different mechanism of diffusion.  Our mechanism still involves transition chains of invariant primary tori, but uses the inner dynamics restricted to the normally hyperbolic invariant manifold to cross over the large gaps. The paper \cite{DelshamsLS2006} identifies secondary tori and hyperbolic invariant manifolds of lower dimensional tori inside the large gaps, and forms   transition chains of such  objects that can be joined with the transition chains of primary tori. Hence, it still uses the outer dynamics to cross over those gaps.

Since in our approach we do not use transition chains of secondary tori or of hyperbolic invariant manifolds of lower dimensional tori, we do not need to assume the additional  non-degeneracy conditions  on the scattering map acting on these objects as in  \cite{DelshamsLS2006}.

Moreover, one can combine the topological mechanism in this paper with the one in \cite{GideaL06} and obtain diffusing orbits that visit any given collection of primary tori, secondary tori, invariant manifolds of lower dimensional tori, and Aubry-Mather sets, in any prescribed order.

\section{Background on twist maps and Aubry-Mather sets}\label{section:aubrymathersets}

Let $\tilde A = \mathbb{T}^1 \times [0,1]= \{(x, y)\in \mathbb{T}^1 \times[0,1]\}$ be an annulus, and let $A=\mathbb{R}\times [0,1]$ be its universal cover  with the natural projection $\pi:A\to \tilde A$ given by $\pi(x,y)=(\tilde x,\tilde y)$, where $\tilde x=x (\textrm{mod } 1)$ and $\tilde y=y$. Let $\pi_x$ be the projection onto the first component, and $\pi_y$ be the projection onto the second component. Let $\tilde{f}:\tilde A\to \tilde A$ be a $C^1$  mapping on $\tilde A$, and let $f: A\to A$ be the unique lift of $\tilde f$ to $A$ satisfying $\pi_x(f(0,0))\in [0,1)$ and $\pi\circ f=\tilde f\circ \pi$.  In order to simplify the notation, below we will not make distinction between $\tilde{A}$ and $A$, and between $\tilde f$ and $f$.

We assume that $f$ is orientation preserving, boundary preserving, area preserving, and its satisfies a    monotone twist condition, i.e.,
  $|\partial (\pi_{x}\circ  f)/\partial y|>0$ at all points in the annulus.

We note that the above properties imply that $f$ is exact symplectic, i.e. $ f$ has zero flux, meaning that for any rotational curve $\gamma$ the area of the regions above $\gamma$ and below $f(\gamma)$ equals the area below $\gamma$ and
above $f(\gamma)$.

In the sequel we will assume that   $f$ is a positive twist, meaning that\break $\partial (\pi_{x}\circ f)/\partial y>0$ at all points. The map $ f$ restricted to the boundary components $\mathbb{T}^1 \times \{0\}$, $\mathbb{T}^1 \times \{1\}$ of the annulus  has well defined rotation numbers $\omega_-,\omega_+$, respectively, with $\omega^-<\omega^+$.  We will assume that $\omega^-,\omega^+>0$.

By an invariant primary
torus (essential invariant circle) we mean a
$1$-dimensional torus $T$ invariant under $f$  that cannot
be homotopically deformed into a point inside the annulus. Since $f$
is a monotone twist map, each invariant   primary torus $T$ is the
graph of some Lipschitz function (see \cite{Birkhoff1,Birkhoff2}).

A region in $A$ between
two invariant primary tori $T_1$ and $T_2$  is called a Birkhoff Zone of Instability (BZI) provided that there is no invariant primary torus in
the interior of the region.

It is known that, for an area preserving monotone twist map $f$ of $A$,
given a BZI, there exist Birkhoff connecting orbits that go
from any neighborhood of one boundary torus
to any neighborhood of the other boundary
torus (see \cite{Birkhoff1,Birkhoff2,LeCalvez87}). We have the following results:

\begin{thm}[Birkhoff Connecting Theorem]\label{thm:birkhoff} Suppose that $T_1$ and $T_2$ bound a BZI.
For every pair of neighborhoods $U$ of $T_1$ and $V$ of $T_2$ there exist a point $z\in U$  and an integer  $N>0$ such that $f^{N}(z)\in V$.

\end{thm}

\begin{cor}\label{cor:birkhoff} Suppose that $T_1$ and $T_2$ bound a BZI, and that the restrictions of $f$ to $T_1$ and $T_2$ are topologically transitive.
For every  $\zeta_1\in T_1,\zeta_2\in T_2$  and every pair of neighborhoods $U$ of $\zeta_1$ and $V$ of $\zeta_2$,  there exist a point $z\in U$  and an integer $N>0$ such that $f^{N}(z)\in V$.
\end{cor}

A subset $M\subseteq A$ is said to be monotone (cyclically ordered) if $\pi_x(z_1)<\pi_x(z_2)$ implies $\pi_x(f(z_1))<\pi_x(f(z_2))$  for all $z_1,z_2\in M$. For $z\in A$ the extended orbit of $z$ is the set $EO(z)=\{f^n(z)+(j,0)\,:\,n,j\in\mathbb{Z}\}$. The orbit of $z$ is said to be monotone (cyclically ordered) if the set $EO(z)$  is monotone. If the orbit of $z\in A$ is monotone, then the rotation number $\rho(z)=\lim_{n\to\infty}(\pi_x(f^n(z))/n)$ exists. We denote $\textrm{Rot}(\omega)=\{z\in A\,:\,\rho(z)=\omega\}$.
All points in the same monotone set have the same rotation number.

\begin{defn}\label{defn:aubrymather}  An Aubry-Mather set for $\omega\in \mathbb{T}^1$ is a minimal, monotone,  $f$-invariant subset of $\textrm{Rot}(\omega)$.
\end{defn}

Here by a minimal set we mean a closed invariant set that does not contain any proper closed invariant subsets. (Equivalently, the orbit of every point in the set is dense in the set.)  This should not be confused with action-minimizing or $h$-minimal sets, where $h$ is a generating function for $f$.

\begin{thm}[Aubry-Mather Theorem]\label{thm:aubrymather}
For every $\omega\in[\omega^-,\omega^+]$, there exists a non-empty Aubry-Mather set $\Sigma_\omega$ in $\textrm{Rot}(\omega)$.
\end{thm}

Aubry-Mather sets defined as above can be obtained as limits of monotone Birkhoff periodic orbits \cite{Katok82}.  There may be many Aubry-Mather sets with the same rotation number \cite{Mather85}. On the other hand, if one requires Aubry-Mather sets to be action minizing, there exists a unique recurrent Aubry-Mather set for any given irrational rotation number.

In the sequel we will use the following result on the vertical ordering of Aubry-Mather sets from \cite{Gole01}.

\begin{thm}\label{thm:gole} There exists a family of essential circles $C_\omega$ in $A$ for $\omega\in[\omega^-,\omega^+]$ such that:
\begin{itemize}
\item [(i)] Each $C_\omega$ is  a graph over $y=0$;
\item [(ii)] The circles $C_\omega$ are mutually disjoint, and if $\omega'>\omega$ then $C_{\omega'}$ is above $C_{\omega}$;
\item[(iii)] Each $C_\omega$ contains an Aubry-Mather set $\Sigma_\omega$.
\end{itemize}
 \end{thm}

The above circles have Lipschitz regularity, and are projections of so called `ghost circles' that are objects in $\mathbb{R}^\mathbb{Z}$.  See \cite{Gole01} for details.  A similar result to Theorem \ref{thm:gole}  appears in \cite{KatznelsonOrnstein1997} who find Aubry-Mather sets lying on pseudo-graphs that are (not strictly) vertically ordered.

There are  some analogues of the Birkhoff Connecting Theorem for Aubry-Mather sets.
The following lemma is used in \cite{Kaloshin2003} to provide a topological proof for  Mather Connecting Theorem stated below.

\begin{lem}\label{lem:kaloshin}
Suppose that $T_1$ and $T_2$ bound a BZI.  Let $\Sigma_\omega$ be an Aubry-Mather set of rotation number $\omega$ inside the BZI. Let $p$ be a recurrent point in $\Sigma_\omega$ and $W(p)$ be a neighborhood of $p$ inside the BZI. The
following hold true: \begin{itemize}
\item [(i)] For some positive number $n^+$ (resp. $n^-$) depending on $W(p)$  the set\break $\bigcup _{j=0}^{n^+}f^j(W(p))$ (resp. $\bigcup _{j=0}^{n^-}f^{-j}(W(p))$) separates the cylinder.
\item[(ii)] The set $W^{+\infty}:=\bigcup _{j=0}^{\infty}f^j(W (p))$ (resp. the set $W^{-\infty}:=\bigcup _{j=0}^{\infty}f^{-j}(W(p) )$),  is connected and open.
\item [(iii)] The closure of $W^{+\infty}$
(resp. $W ^{-\infty}$) contains both boundary tori $T_1$ and $T_2$.
\item [(iv)] The set $W ^\infty:=\bigcup _{j=-\infty}^{\infty}f^j(W(p))$ is invariant, and both $W^{+\infty} $ and $W^{-\infty} $
are open and dense in $W ^\infty$.
\end{itemize}
\end{lem}

The following result says that there exist orbits that visit any prescribed bi-infinite sequence of Aubry-Mather sets inside a BZI (see \cite{Mather91,Xia-preprint,Hall1989,KatznelsonOrnstein1997,Kaloshin2003}).

\begin{thm}[Mather Connecting Theorem]\label{thm:mathershadowing}
Suppose that $T_1$ and $T_2$ bound a BZI, and $\{\Sigma_{\omega_i}\}_{i\in\mathbb{Z}}$  is a bi-infinite sequence of Aubry-Mather sets inside the BZI.  Let  $\eps_i>0$ for $i\in\Z$. Then there exist  a point $z$ inside the BZI and an increasing bi-infinite sequence of integers $\{j_i\}_{i\in\Z}$ such that $f^{j_i}(z)$ is within $\eps_i$ from $\Sigma_{\omega_i}$ for all $i\in \Z$.
\end{thm}

The Aubry-Mather sets in Theorem \ref{thm:mathershadowing}  are action minimizing. The following topological version of Mather Connecting Theorem, due to Hall \cite{Hall1989}, provides shadowing orbits of Aubry-Mather sets that are not necessarily action minimizing. This approach can be implemented in rigorous computer experiments \cite{Jungreis91}.

\begin{thm} \label{thm:diagonal} Suppose that $T_1$ and $T_2$ bound a BZI, and  $\{z_s\}_{s\in\mathbb{Z}}$ is  a bi-infinite sequence of monotone $(p_s/q_s)$-periodic points, with the rotation numbers $p_s/q_s$ mutually distinct, inside the BZI. Given a bi-infinite sequence $\{n_s\}_{s\in\mathbb{Z}}$  of positive integers, then there exist a point $z$ and a  bi-infinite sequence $\{m_s\}_{s\in\mathbb{Z}}$ of positive integers such that, for each $s\geq 0$, there exist some points $w_s,\bar w_s$  in the extended orbit of $z_s$ such that
\begin{equation}\label{eqn:hall1}\begin{split}\pi_x(f^j(w_s))<\pi_x(f^j(z))<\pi_x(f^j(\bar w_s))\textrm { for }\\\sum_{t=0}^{s-1} n_t+\sum_{t=0}^{s-1}m_t\leq j\leq \sum_{t=0}^{s} n_t+\sum_{t=0}^{s-1}m_t.\end{split}\end{equation}

A similar statement holds for each $s<0$.
\end{thm}

In the above, $n_s$ represents  the number of iterates for which the orbit of $z$ shadows -- in the sense of the cyclical ordering --  the extended orbit  of $z_s$, and $m_s$ represents the number of iterates it takes the orbit of $z$ to pass from the extended orbit of $z_s$ to the extended orbit of $z_{s+1}$.

They main tool used in Hall's arguments is that of a positive (negative) diagonal. Denote by $\mathcal{Z}$ the BZI bounded by the tori $T_1$ and $T_2$. Let
\begin{eqnarray}
I_z&=&\{w\in\mathcal{Z}\,|\, \pi_x(w)=\pi_x(z)\},\\
I^+_z&=&\{w\in I_z\,|\, \pi_y(w)\geq \pi_y(z)\},\\
I^-_z&=&\{w\in I_z\,|\, \pi_y(w)\leq \pi_y(z)\},\\
B_{z_0,z_1}&=&\{w\in\mathcal{Z}\,|\, \pi_x(z_0)< \pi_x(w)< \pi_x(z_1)\},
\end{eqnarray}
where $z,z_0,z_1$ are points in the annulus.

A positive diagonal $D$ in $B_{z_0,z_1}$ is a set $D\subseteq \textrm{cl}(B_{z_0,z_1})$ such that
\begin{itemize}
\item [(i)] $D$ is simply connected and the closure of its interior;
\item [(ii)] $\partial D \cap \textrm{cl}(B_{z_0,z_1})\subseteq I^-_{z_0}\cup I^+_{z_1} \cup T_1\cup T_2$;
\item [(iii)] $\partial D\cap I^-_{z_0}\neq \emptyset$ and $\partial  D\cap I^+_{z_1}\neq \emptyset$.
\end{itemize}

The  set $\partial D\cap B_{z_0,z_1}$ has exactly two components connecting $I^-_{z_0}\cup T_1$ to $I^+_{z_1}\cup T_2$, which are called the upper and lower edges of $D$, respectively. We informally say that these components `stretch across' $B_{z_0,z_1}$. See Fig.~\ref{fig:halldiag}.

A negative diagonal and its upper and lower edges are defined similarly.

\begin{figure} \centering
\includegraphics*[width=0.75\textwidth, clip, keepaspectratio]
{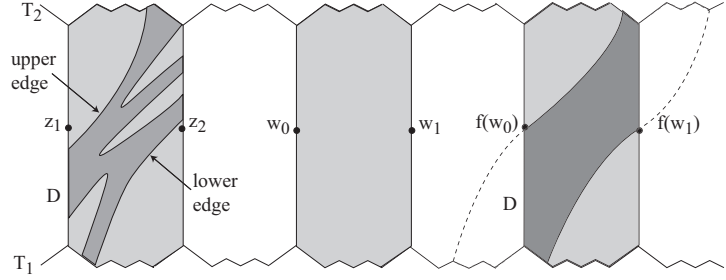}
    \caption[]{Two positive diagonal sets. The positive diagonal set on the left has its upper and lower edges marked. The positive diagonal set on the right is obtained by intersecting $f(B_{w_0,w_1})$ with $B_{f(w_0),f(w_1)}$. The upper edge of this diagonal set is contained in $f(I^+_{w_0})$ and the lower edge is contained in $f(I^-_{w_1})$.}
    \label{fig:halldiag}
\end{figure}

An important feature of positive diagonals is the following hereditary property.
Given $z_0,z_0$  such that $\pi_x(z_0)<\pi_x(z_1)$ and $\pi_x(f(z_0))<\pi_x(f(z_1))$, if $D$ is a positive diagonal in $B_{z_0,z_1}$, then $f(D)\cap B_{f(z_0),f(z_1)}$ has a component $D'$ that is a positive diagonal in $B_{f(z_0),f(z_1)}$.

One way to generate a positive diagonal set is by taking a component of the intersection between $f^k(B_{w_0,w_1})$ and $B_{f^k(w_0),f^k(w_1)}$. In this case, there exists a positive diagonal in $B_{f^k(w_0),f^k(w_1)}$ whose upper edge is contained on $f^k(I_{w_0}^+)$ and lower edge is contained in $f^k(I_{w_1}^-)$.
More general, one has the following important property. If $D$ has the upper edge contained in $f^k(I^+_{w_0})$ and the lower edge contained in $f^k(I^-_{w_1})$, and
$\partial D\cap B_{z_0,z_1}\subseteq f^k(I^+_{w_0}\cup I^-_{w_1})$, for some $w_0,w_1$ with  $\pi_x(w_0)<\pi_x(w_1)$ and some $k>0$,  then
then $f(D)\cap B_{f(z_0),f(z_1)}$ has a component $D'$  which can be chosen so that its upper edge is contained in $f^{k+1}(I^+_{w_0})$ and its lower edge is contained in $f^{k+1}(I^-_{w_1})$. See Fig.~\ref{fig:halldiag}. A similar property holds for negative diagonals.

The proof of Theorem \ref{thm:diagonal}  in \cite{Hall1989} is an inductive argument which, for a given pair of adjacent points $w_0,\bar w_0$ in the extended orbit of $z_0$,
and for each $\sigma\geq 0$, produces a nested sequence $D_0\supseteq D_1\supseteq \ldots \supseteq D_{\sigma}$   of negative diagonals of $B_{w_0,\bar w_0}$ such that, for each $s\in\{0,\ldots,\sigma\}$, the following hold: (a) the orbit of each point $z \in D_s$  satisfies the ordering relation \eqref{eqn:hall1},  and (b) there is a sufficiently large $j_s>0$ such that $f^{j_s+j}(D_s)$ contains a component that is  a positive diagonal in $B_{f^j(w_s),f^j(\bar w_s)}$, for some adjacent points $w_s,\bar w_s\in EO(z_s)$,  and for all $j=1,\ldots, n_s$. In the above, $j_s=\sum_{t=0}^{s-1} n_t+\sum_{t=0}^{s-1}m_t$.
Moreover, in this  inductive argument one can choose the diagonal sets $D_s$ so that $f^{j_s+j}(D_s)$ has the upper edge contained in $f^{j_s+j}(I^+_{w_0})$, lower edge contained in $f^{j_s+j}(I^-_{\bar w_0})$, and $\partial f^{j_s+j}(D_s)\cap B_{f^{j_s+j}(w_0),f^{j_s+j}(\bar w_0)} \subseteq f^{j_s+j} (I^+_{w_0}\cup I^-_{\bar w_0})$.

For the basis step, starting with $w_0,\bar w_0$ and applying the hereditary property from above $n_0$ times, one obtains  a negative diagonal set $D_0$ of $B_{w_0,\bar w_0}$ with the properties that  each point $z \in D_0$  satisfies the ordering relation \eqref{eqn:hall1} for $s=0$, and $f^{n_0}(D_0)$  has a component that is a positive diagonal of $B_{f^{n_0}(w_0),f^{n_0}(\bar w_0)}$.

For the inductive step, one assumes a negative diagonal $D_\sigma$  of $B_{w_0,\bar w_0}$ as above, and wants to produce a negative diagonal $D_{\sigma+1}\subseteq D_\sigma$  of $B_{w_0,\bar w_0}$  which fulfils the corresponding properties. The key idea  is  to use the existence of points near $y=0$ that get near $y=1$, and of points near $y=1$ that get near $y=0$, as provided by Theorem \ref{thm:birkhoff}, in order to show that  for some $j_\sigma$ sufficiently large
$f^{j_\sigma}(D_\sigma)$ contains a component that stretches all the way across a fundamental interval of the annulus. Hence  $f^{j_\sigma}(D_\sigma)$ contains a subset that is a positive diagonal of $B_{w_{\sigma+1}, \bar w_{\sigma+1}}$ for two adjacent points $w_{\sigma+1}, \bar w_{\sigma+1}\in EO(z_{\sigma+1})$. From this it follows that $f^{j_\sigma+j}(D_\sigma)$ contains a component that is  a positive diagonal in $B_{f^j(w_{\sigma+1}),f^j(\bar w_{\sigma+1})}$ for all $j=1,\ldots, n_{\sigma+1}$. This completes the inductive step.

Applying a similar argument for the negative iterates of $f$ produces a nested sequence of positive diagonals of $B_{w_0,\bar w_0}$. A positive diagonal of  $B_{w_0,\bar w_0}$ always has a non-empty intersection  with a negative diagonal  of $B_{w_0,\bar w_0}$. This implies the existence of points $z$ whose forward  orbits satisfy the ordering conditions in \eqref{eqn:hall1} and whose backwards orbits satisfy similar ordering conditions.

Using limit arguments as in \cite{Katok82}, one can obtain shadowing of Aubry-Mather sets of irrational rotation numbers as well. These topological ideas will be used in the proof of Theorem \ref{thm:extensionmathershadowing} below.

\begin{rem}\label{rem:lemma}
An immediate consequence of Lemma \ref{lem:kaloshin} is that, given a neighborhood $W$ of a point $p\in\Sigma_\omega$, where $\Sigma_\omega$ is an Aubry-Mather set inside a BZI bounded by $T_1$ and T$_2$, and given a neighborhood $U$  of $T_1$ or  $T_2$, there exists an arbitrarily large $j>0$ such that $f^j(W(p))\cap U\neq \emptyset$. Also, there exists an arbitrarily large $j'>0$ such that $f^{-j'}(W(p))\cap U\neq \emptyset$.
\end{rem}
\begin{rem}\label{rem:tilt} The results in this section hold  if we replace conditions (i) and (ii) from the definition of an area preserving, monotone twist map with the following weaker conditions:
\begin{itemize}
\item [(i')] $f$ satisfies the following  `condition B': for every pair of neighborhoods $U_1$ of $T_1$ and $U_2$ of $T_2$, there exist $z_1,z_2\in A$ and $n_1,n_2>0$ such that $z_1\in U_1$ and $f^{n_1}(z_1)\in U_2$, and $z_2\in U_2$ and $f^{n_2}(z_2)\in U_1$. Note that this condition only makes sense if we restrict the dynamics to a BZI.
    \item [(ii')] $f$ satisfies  the following positive tilt condition: if we denote by $\theta_{z}$ the angle deviation from  the vertical, measured from the vertical vector $(0,1)$ to $Df_{z}(0,1)$, with the clockwise direction taken as the positive direction,  and defined in such a way that $\theta_{(x,0)}\in [-\pi/2,\pi/2]$ and $\theta$ is continuous, then $\theta_{z}> 0$ at all points. See Fig. \ref{fig:tiltmap}.
\end{itemize}
Compositions of positive twist maps, are for example, positive tilt maps. We shall note that the Aubry-Mather  theory applies to positive tilt maps as well (see \cite{Hu98}).\end{rem}

\begin{figure} \centering
\includegraphics*[width=0.75\textwidth, clip, keepaspectratio]
{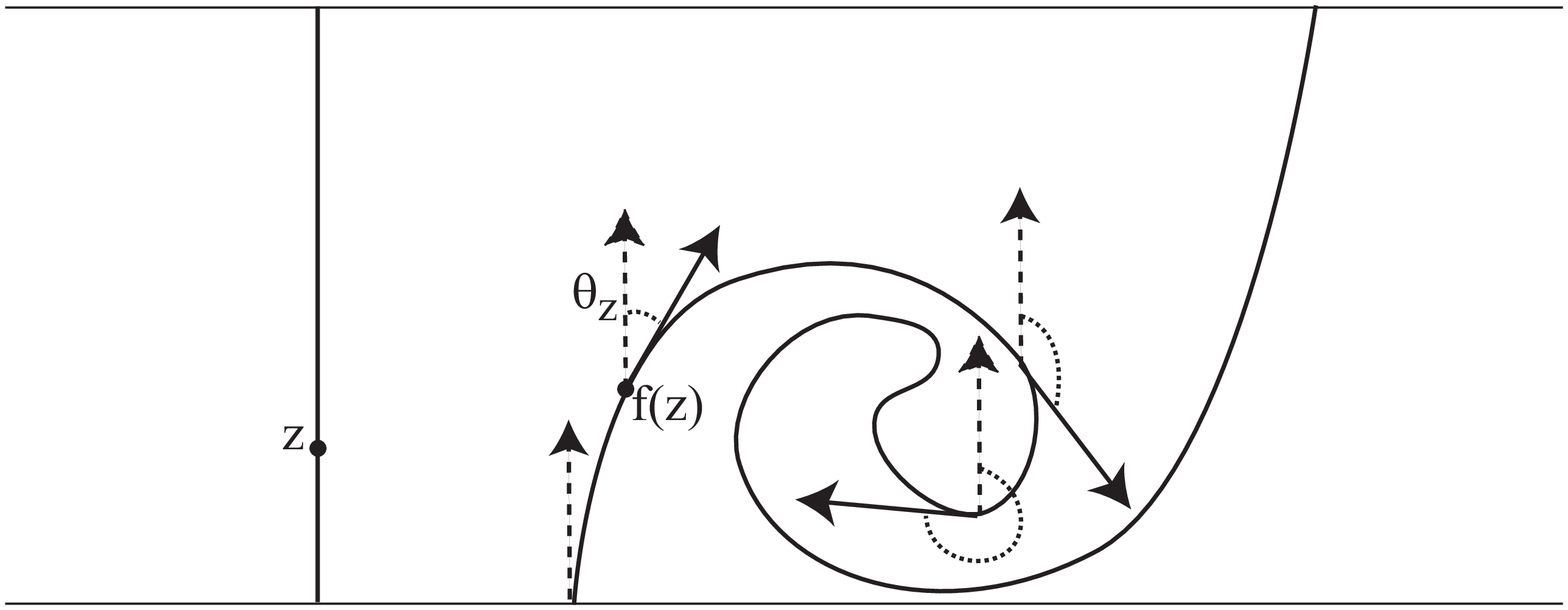}
    \caption[]{Positively tilted map.}
    \label{fig:tiltmap}
\end{figure}

\begin{rem} Aubry-Mather theory  and the above shadowing result also hold for generalized twist maps of the higher dimensional annulus $S^1\times \mathbb{R}^n$. See \cite{Angenent90}.
\end{rem}
\begin{rem}
The topological approach in this section does not yield trajectories that get close to each set in the prescribed collection of Aubry-Mather sets, as in Theorem \ref{thm:mathershadowing}. It seems  possible, however, that these topological methods can be combined with the variational methods in \cite{Mather91} to obtain trajectories that go very close to each  Aubry-Mather set in the given collection.
\end{rem}

\section{Background on the scattering map and on the topological method of correctly aligned windows}
\subsection{Scattering map}\label{section:scattering}

The scattering map acts on the
normally hyperbolic invariant manifold $\Lambda$ and relates the past asymptotic trajectory
of each orbit in the homoclinic manifold
to its future asymptotic behavior. We review its properties following \cite{DelshamsLS08a}.

For the general case, we consider a $C^r$-differentiable manifold $M$, a $C^r$-differentiable map $f:M\to M$, and an $l$-dimensional normally hyperbolic invariant manifold $\Lambda$ for $f$. By the definition of normal hyperbolicity, there exists a splitting of the tangent bundle
of $TM$ into sub-bundles
\[TM=E^u\oplus E^s\oplus T\Lambda,\]
that are invariant under $df$, and there exist a constant $C>0$ and
rates $0<\lambda<\mu^{-1}<1$, such that for all $x\in\Lambda$ we
have
\[\begin{split}
v\in E^s_x  \Leftrightarrow \|Df^k_x(v)\|\leq C\lambda^k\|v\|  \textrm{ for all } k\geq 0,\\
v\in E^u_x  \Leftrightarrow \|Df^k_x(v)\|\leq C\lambda^{-k}\|v\|  \textrm{ for all } k\leq 0,\\
v\in T_x\Lambda \Leftrightarrow \|Df^k_x(v)\|\leq C\mu^{|k|}\|v\|
\textrm{ for all } k\in\mathbb{Z}.
\end{split}\]

The smoothness of the invariant objects defined by the normally hyperbolic structure
depends on the rates $\lambda$ and $\mu$.
The map $f$ is said to be {\em $\ell$-normally hyperbolic along $\Lambda$} provided that
$1 \leq \ell \leq r$ is an integer satisfying $\lambda\,\mu^\ell <1$, i.e.,
$\ell < (\log \lambda^{-1})(\log \mu)^{-1}$.  Then the stable and unstable manifolds
$W^s(\Lambda)$ and $W^u(\Lambda)$ and the normally hyperbolic manifold $\Lambda$ are all
$C^\ell$-differentiable.
The splitting $E^s_z=T_z(W^s(x))$  depends $C^{\ell-1}$ smoothly on
$z$ in $W^s(\Lambda)$ so $\{\,W^s(x)\, |\, x \in \Lambda\,\}$\, is a $C^{\ell-1}$ foliation of
$W^s(\Lambda)$.  (This is stated explicitly in \cite{Robinson1971} and follows from the
$C^r$ Section Theorem of \cite{HirschPS77}.)
Similarly, $\{\,W^u(x)\, |\, x \in \Lambda\,\}$\, is a $C^{\ell-1}$ foliation of
$W^u(\Lambda)$.

Since the stable (resp. unstable) manifolds of $\Lambda$ are foliated by stable
(resp. unstable) manifolds of points, we have that for each $x\in W^s(\Lambda)$
(resp. $x\in W^u(\Lambda)$), there exists a unique $x^+ \in \Lambda$
(resp. $x^- \in \Lambda$) such that $x \in W^s(x^+)$
(resp. $x \in W^u(x^-)$).
We define the maps $\Omega^+ : W^s(\Lambda) \to \Lambda$ by $\Omega^+(x) = x^+$ and
$\Omega^- : W^u(\Lambda) \to \Lambda$ by $\Omega^-(x) = x^-$.
The maps $\Omega^+$ and $\Omega^-$ are $C^{\ell -1}$-smooth since the foliations are smooth.

We now describe the scattering  map.  Assume that $W^u(\Lambda)$ and
$W^s(\Lambda)$ have a differentiably transverse intersection along a
homoclinic $l$-dimensional $C^{\ell-1}$-differentiable manifold $\Gamma$.
This means that $\Gamma\subseteq W^u(\Lambda) \cap W^s(\Lambda)$
and, for each $x\in\Gamma$, we have
\begin{equation}\label{eqn:channel1} \begin{split}
T_xM=T_xW^u(\Lambda)+T_xW^s(\Lambda),\\
T_x\Gamma=T_xW^u(\Lambda)\cap T_xW^s(\Lambda).
\end{split} \end{equation}
We assume the additional condition that for each $x\in\Gamma$ we
have
\begin{equation}\label{eqn:channel2}\begin{split}
T_xW^s(\Lambda)=T_xW^s(x^+)\oplus T_x(\Gamma),\\
T_xW^u(\Lambda)=T_xW^u(x^-)\oplus T_x(\Gamma),
\end{split} \end{equation}
where $x^-,x^+$ are the uniquely defined points in $\Lambda$
corresponding to $x$.

The restrictions $\Omega^+_\Gamma,\Omega^-_\Gamma$ of
$\Omega^+,\Omega^-$ to $\Gamma$ are  local
$C^{\ell-1}$-diffeomorphisms.  By replacing $\Gamma$ to a
submanifold of it  (which, with an abuse of notation, we still denote $\Gamma$) we can ensure that
$\Omega^+_\Gamma:\Gamma\to U^+,\Omega^-_\Gamma:\Gamma\to U^-$ are $C^{\ell-1}$-diffeomorphisms from $\Gamma$ to the open sets $U^+,U^-
$ in $\Lambda$, respectively.

\begin{defn}A homoclinic manifold $\Gamma$ satisfying \eqref{eqn:channel1} and \eqref{eqn:channel2}, and for which the corresponding
restrictions of the wave maps are $C^{\ell-1}$-diffeomorphisms,  is
 referred as a homoclinic channel.\end{defn}  \begin{defn} Given a
homoclinic channel $\Gamma$, the scattering map associated to
$\Gamma$ is the $C^{\ell-1}$-diffeomorphism
$S_\Gamma=\Omega_\Gamma^+\circ (\Omega_\Gamma^-)^{-1}$ from
the open subset $U^-:=\Omega_\Gamma^-(\Gamma)$ in $\Lambda$ to the
open subset $U^+:=\Omega_\Gamma^+(\Gamma)$ in $\Lambda$.
\end{defn}

In the sequel we will regard $S_\Gamma$ as a partially defined map, so the
image of a set $A$ by $S_\Gamma$ means the set $S_\Gamma(A\cap U^-)$.

In this paper, we need the following property of the scattering map.

\begin{prop}\label{prop:trans} Assume that $T_1$ and $T_2$ are two invariant submanifolds of complementary dimensions in $\Lambda$. Then $W^u(T_1)$ has a topologically transverse intersection with $W^s(T_2)$ inside $\Gamma$  if and only if $S_\Gamma(T_1\cap U^-)$ has a topologically transverse intersection with $T_2\cap U^+$ in $\Lambda$.
\end{prop} \begin{proof}
Let $\bar T_1=(\Omega^-_\Gamma)^{-1}(T_1\cap U^-)\subseteq \Gamma$ and $\bar T_2=(\Omega^+_\Gamma)^{-1}(T_2\cap U^+)\subseteq \Gamma$. A topologically transverse intersection of $W^u(T_1)$ with  $W^s(T_2)$ in $\Gamma$ occurs if and only if $\bar T_1$ intersects  $\bar T_2$ topologically transversally in $\Gamma$, which is equivalent to
$S_\Gamma(T_1\cap U^-)$ intersects topologically transversally  $T_2\cap U^+$ in $\Lambda$.
\end{proof}

For the definition of topological transversality (topological crossing) see \cite{BurnsW95}. For the main result of this paper  the normally hyperbolic invariant manifold $\Lambda$ is assumed to be $2$-dimensional, i.e., $l=2$, and the invariant submanifolds $T_1,T_2$ in Proposition \ref{prop:trans} are $1$-dimensional invariant tori.

\subsection{Topological method of correctly aligned windows}\label{section:topological}

We describe briefly the topological method of correctly aligned windows. We follow \cite{GideaZ04a}. See also \cite{GideaR03,GideaL06,LochakM05}.

\begin{defn}
An $(n_1,n_2)$-window in an $n$-dimensional manifold $M$, where $n_1+n_2=n$, is a
compact subset $W$ of $M$ together with a
homeomorphism $\chi$ from some open neighborhood  of $[0,1]^{n_1}\times [0,1]^{n_2}$ in $\mathbb{R}^{n_1}\times \mathbb{R}^{n_2}$ to an open subset of $M$, such that  \[W=\chi([0,1]^{n_1}\times [0,1]^{n_2}),\] and with a choice of an `exit set' \[W^{\rm
exit} =\chi \left(\partial[0,1]^{n_1}\times [0,1]^{n_2} \right )\]
and  of an `entry set'  \[W^{\rm entry}
=\chi \left([0,1]^{n_1}\times
\partial[0,1]^{n_2}\right ).\]
\end{defn}

Denote by $\pi_{1}: \mathbb{R}^{n_1}\times \mathbb{R}^{n_2}\to \mathbb{R}^{n_1}$  the   projection onto the first component, and by $\pi_{2}: \mathbb{R}^{n_1}\times \mathbb{R}^{n_2}\to \mathbb{R}^{n_2}$  the projection onto the second component.

\begin{defn}\label{defn:corr}
Let  $W_1$ and $W_2$ be $(n_1,n_2)$-windows, and let $\chi_1$ and $\chi_2$ be the corresponding local parametrizations.
Let $f$ be a continuous
map on $M$ with $f(\textrm {im}(\chi_1))\subseteq \textrm
{im}(\chi_2)$.  We say that $W_1$ is correctly aligned with
$W_2$ under $f$ if the following conditions are satisfied:
\begin{itemize}
\item[(i)]
   $f(W^{\rm exit}_1)  \cap (W_2)  =  \emptyset$ and
 $f(W_1) \cap W_2^{\rm entry}   =  \emptyset$;

\item[(ii)] There exists $y_0 \in [0,1]^{n_2}$ such that the curve $x\in [0,1]^{n_1} \mapsto\hat f (x,y_0)$, where $\hat f:=\chi_2^{-1}\circ f\circ \chi_1$,
has the following properties:

\begin{eqnarray*}\hat f_{y_0}\left ( [0,1]^{n_1}\right )\subseteq \mathbb{R}^{n_1}\times (0,1)^{n_2},\\
\hat f_{y_0}\left ( \partial[0,1]^{n_1}\right )\subseteq
(\mathbb{R}^{n_1}\setminus [0,1]^{n_1})\times (0,1)^{n_2}=\emptyset,\\
\deg(\pi_{1}\circ \hat f_{y_0},0)=w\neq 0.\end{eqnarray*}

\end{itemize}
We call the integer $w\neq 0$ in the above definition the degree of
the alignment.
\end{defn}

The following result is a topological version of the Shadowing
Lemma.
\begin{thm}
%[Existence of orbits with prescribed trajectories]
\label{thm:detorb} Let $\{W_i\}_{i\in\mathbb{Z}}$, be a collection of
$(n_1,n_2)$-windows in $M$,  and let $f_i$ be a collection of continuous maps on
$M$. If for each  $i\in\mathbb{Z}$, $W_i$ is correctly aligned with $W_{i+1}$ under $f_i$, then
there exists a point $p\in W_0$ such that
\[(f_{i}\circ \cdots\circ f_{0})(p)\in W_{i+1}, \textrm{ for all } i\in\mathbb{Z}. \]

Moreover, assuming that there exists $k>0$ such that $W_{i}=W_{(i\,{\rm mod}\, k)}$  and $f_{i}=f_{(i\,{\rm mod}\, k)}$ for all
$i\in \mathbb{Z}$,
then there exists a
point $p$ as above that is periodic in the sense
\[(f_{k-1}\circ \cdots\circ f_{0})(p)=p.\]
\end{thm}

The correct alignment of windows is robust, in the sense that if two
windows are correctly aligned under a map, then they remain
correctly aligned under a sufficiently small $C^0$-perturbation of the
map. Robustness makes the method of correctly aligned windows
appropriate for perturbative arguments,  as well as for rigorous numerical experiments.

Also, the correct alignment satisfies a natural product property.
Given two windows and a map, if each window can be written as a
product of window components, and if the components of the first
window are correctly aligned with the corresponding components of
the second window under the appropriate components of the map, then
the first window is correctly aligned with the second window under
the given map. For example, if we consider a pair of windows in a neighborhood of a
normally hyperbolic invariant manifold, if the center components
of the windows are correctly aligned and the hyperbolic components of the windows
are also correctly aligned, then the windows are correctly aligned.
Although the product property is quite intuitive, its
rigorous statement is rather technical, so we will omit it here.
The details can be found in \cite{GideaL06}.

In Sections \ref{section:shadowing} and \ref{section:proof}, we will consider various  windows lying on one of the manifolds $\Lambda$, $W^u(\Lambda)$, $W^s(\Lambda)$, or $M$.
Without explicit mention, every such a window will be represented by the image of a rectangle through a local parametrization of the appropriate manifold.   We will also consider correct alignment relations of windows lying on the same manifold. All the correct alignment relations in the arguments presented later in this paper have degree $w=1$.

\section{Existence of Birkhoff connecting orbits}

In this section we state and prove an extension of the Corollary \ref{cor:birkhoff} of the Birkhoff connecting orbit theorem, and an extension of Mather's theorem on shadowing of Aubry-Mather sets, specifically of Theorem \ref{thm:diagonal}. The  methodology is based on  the topological approach of Hall and on  the Jordan Curve Theorem. The statements below will be used  in the proof of the main theorem. Throughout the  section we assume that $f$ is an orientation preserving, boundary preserving, area preserving, monotone twist map of the annulus $A$, as in Section \ref{section:aubrymathersets}, and we adopt the notation conventions from that section.

In Corollary \ref{cor:birkhoff}, it was assumed that the restrictions of the map to the boundary tori of the BZI are topologically transitive, and it was inferred  the existence of connecting orbits from an arbitrarily small neighborhood of some prescribed point on one boundary torus to an arbitrarily small neighborhood of some prescribed point on the other boundary torus. In the statements below we prove the same result without the topological transitivity assumption.

\begin{thm}\label{thm:extensionbirkhoff2}
Suppose that $T_1$ and $T_2$ bound a BZI $\mathcal{Z}$. Assume that $\zeta_1\in T_1$ and $\zeta_2\in T_2$. Fix a  pair of neighborhoods $U$ of $\zeta_1$ and $V$ of $\zeta_2$. Then   there exists a point $z\in U$ and an integer $N>0$ such that $f^{N}(z)\in V$ for some $N>0$ that can be chosen arbitrarily large.

Moreover, if we assume that $U,V$ are chosen so that $U\cap\mathcal{Z},V\cap\mathcal{Z}$ are topological disks, then there exists a point $z'\in \partial U$ such that $f^N(z')\in \partial V$.
\end{thm}
\begin{proof}
Consider the neighborhood $U$ of $\zeta_1\in T_1$.
Choose two points $\zeta'_1,\zeta''_1\in T_1\cap U$ such that $\zeta_1$ is between $\zeta'_1$ and $\zeta''_1$ and the portion of $T_1$ between $\zeta'_1$ and $\zeta''_1$ is contained in $\textrm{int}(U)$.
Choose a simple curve $\gamma_0$ inside $U$,  with endpoints at $\zeta'_1$ and $\zeta''_1$.
The curve $\gamma_0$ together with the portion of $T_1$ between  $\zeta'_1$ and $\zeta''_1$ determines  a closed topological disk $U_0\subseteq U$, which is a one-sided compact neighborhood of $\zeta_1$ in $\mathcal{Z}$.

Similarly, we can   choose a one-sided compact neighborhood $V_0\subseteq V$ of $\zeta_2\in T_2$, whose boundary consists of a simple curve $\eta_0$  connecting two points $\zeta'_2,\zeta''_2\in T_2$, with $\zeta_2$ between $\zeta'_2$ and $\zeta''_2$, and the portion of $T_2$ between $\zeta'_2$ and $\zeta''_2$.

In this way, for proving the theorem  we can consider the one-sided neighborhoods  $U_0,V_0$ instead of $U,V$, respectively.

Assume first that the interior of $U_0$ meets some Aubry-Mather set $\Sigma_{\rho_1}\subseteq \mathcal{Z}$, and that the interior of $V_0$ meets some Aubry-Mather set $\Sigma_{\rho_2}\subseteq \mathcal{Z}$. Since $U_0$ and $V_0$ are neighborhoods of points in the Aubry-Mather sets $\Sigma_{\rho_1}$ and $\Sigma_{\rho_2}$ respectively, Theorem \ref{thm:mathershadowing} yields the existence of a forward orbit that goes from $U_0$ to $V_0$. Hence there exists $N>0$ such that    $f^N(U_0)\cap V_0\neq \emptyset$. Since $f^N(U_0)$ and $V_0$ are open topological disks that have common points as well as non-common points (e.g., the points in $T_1$ and $T_2$, respectively), the Jordan Curve Theorem implies that $f^N(\partial  U_0)\cap \partial  V_0\neq \emptyset$.

Assume now that the interiors of $U_0,V_0$  do not meet any Aubry-Mather set.

We choose  three Aubry-Mather sets $\Sigma_{\rho_1}$, $\Sigma_{\rho_1'}$, $\Sigma_{\rho_1''}$ in $\mathcal{Z}$,  lying on three essential circles $C_{\rho_1},C_{\rho'_1},C_{\rho''_1}$, respectively, with ${\rho_1}<{\rho'_1}<{\rho''_1}$ irrational rotation numbers, and $C_{\rho_1}\prec C_{\rho'_1}\prec C_{\rho''_1}$.
The existence of such vertically ordered Aubry-Mather sets follows from Theorem \ref{thm:gole}.

The proof of the theorem uses  the following intermediate step.

\textit{Claim. There exist $j'_*>j_*>0$ such that  $\mathcal{Z}\setminus[f^{j_*}(U_0)\cup f^{j'_*}(U_0)]$ contains a component $\mathcal{U}$ which is an open topological disk that is a neighborhood of some point in $\Sigma_{\rho_1}$.
Moreover, $j_*,j'_*$ can be chosen arbitrarily large.
A similar statement holds for $T_2$.}

\textit{Proof of the claim.} Let $p_1$ be a point in $\Sigma_{\rho''_1}$. Let $W(p_1)$ be a small neighborhood of $p_1$ inside the BZI, which does not intersect $\Sigma_{\rho_1}$  and $\Sigma_{\rho_1'}$. By assumption, $U_0$ does not meet any of the sets $\Sigma_{\rho_1}$, $\Sigma_{\rho'_1}$, $\Sigma_{\rho''_1}$. By Lemma \ref{lem:kaloshin} (iii) the closure of $\bigcup _{j=0}^{\infty}f^{-j}(W(p_1))$ contains $T_1$, and in particular $\zeta_1$. Since $U_0$ is a one-sided neighborhood of $\zeta_1$, there exists $j_1>0$ such that $f^{j_1}(U_0)\cap W(p_1)\neq \emptyset$.
In the covering space of the annulus, $f^{j_1}(U_0)$ intersects some copy $W^{h_1}:=W(p_1)+(h_1,0)$ of $W(p_1)$, where $h_1$ is some positive integer.
Since $U_0$ does not intersect $\Sigma_{\rho_1}$ and $\Sigma_{\rho'_1}$, it follows that $f^{j_1}(U_0)$ does not intersect $\Sigma_{\rho_1}$ and $\Sigma_{\rho'_1}$. On the other hand,  $f^{j_1}(U_0)$ intersects the  essential circles $C_{\rho_1},C_{\rho'_1}$, containing the Aubry-Mather sets $\Sigma_{\rho_1},\Sigma_{\rho'_1}$, respectively.

Let $\gamma_{1}:[0,1]\to U_0$ be a vertical curve, i.e., $\gamma_1(0)\in T_1$ and  $\pi_x(\gamma_1(t))=\pi_x(\gamma_1(0))$ for  all $t$, such that $f^{j_1}(\gamma_1(1))$ is an intersection point of $f^{j_1}(U_0)$ with $W^{h_1}$. The curve $f^{j_1}(\gamma_1)$  is a positively tilted curve which crosses both essential circles $C_{\rho_1}$ and $C_{\rho'_1}$. (See Remark  \ref{rem:tilt}.)
Since $f^{j_1}(U_0)$ is disjoint from $\Sigma_{\rho_1}$, the intersections between $f^{j_1}(\gamma_1)$ and $C_{\rho_1}$
occur within the `gaps' of  $\Sigma_{\rho_1}$, i.e., within the open interval components  of $C_{\rho_1}\setminus \Sigma_{\rho_1}$.

We can assign an oriented intersection number   between $f^{j_1}(\gamma_1 )$ and each gap of $\Sigma_{\rho_1}$.
Consider a homotopy $h_s:A\to A$, $s\in[0,1]$, such that $f^{j_1}(h_s(\gamma_1))$ keeps the endpoints of $f^{j_1}(\gamma_1)$ fixed for all $s$, $f^{j_1}(h_s(\gamma_1))$ does not intersect  $\Sigma_{\rho_1}$ for any $s\in[0,1]$, and $f^{j_1}(h_1(\gamma_1))$ is transverse to $C_{\rho_1}$ (see \cite{BurnsW95}).
We set the oriented intersection number of $f^{j_1}(h_1(\gamma_1))$ with $C_{\rho_1}$ to be $+1$ at a point where the curve moves from below $C_{\rho_1}$ to above $C_{\rho_1}$ as $t$ increases, and to be $-1$ at a point where the curve moves from above $C_{\rho_1}$ to below $C_{\rho_1}$ as $t$ increases. Then we   assign an oriented intersection number between $f^{j_1}(h_1(\gamma_1))$  and a gap of $\Sigma_{\rho_1}$, by adding the oriented intersection numbers for all of the intersection points within that gap.
Since the oriented intersection number is preserved by homotopy, the oriented intersection number between $f^{j_1}( \gamma_1)$ and a gap is, by definition, the oriented intersection number between $f^{j_1}(h_1(\gamma_1))$ and that gap. Then the oriented intersection number between $f^{j_1}( \gamma_1 )$ and $C_{\rho_1}$ is the sum of the oriented intersection numbers over all gaps.

Since the curve $f^{j_1}(\gamma_1)$ starts from below $C_{\rho_1}$ and ends above $C_{\rho_1}$,   there exists a gap for which the oriented intersection number with $f^{j_1}(\gamma_1)$ is positive.

We follow the curve $t\mapsto  f^{j_1}(\gamma_1(t))$ starting with $t=0$ and we mark the first gap of $\Sigma_{\rho_1}$ that is crossed  by
$f^{j_1}(\gamma_1)$ with a positive oriented intersection number; we denote by $a^1_{\rho_1},b^1_{\rho_1}$   the endpoints of this gap.
This means that if $f^{j_1}(\gamma_1)$ crosses other gaps of $\Sigma_{\rho_1}$  that are to the left of this gap, it does so with $0$ oriented intersection number. (If the first gap crossed would be crossed with negative oriented intersection number, it would violate the positive tilt condition of $f^{j_1}(\gamma_1)$.) Thus,
when the curve  $f^{j_1}(\gamma_1)$ crosses the gap between $a^1_{\rho_1},b^1_{\rho_1}$, it comes from below the circle $C_{\rho_1}$.
Following the  curve segment of $f^{j_1}(\gamma_1(t))$  after crossing the gap of endpoints $a^1_{\rho_1},b^1_{\rho_1}$,  we mark the first gap of $\Sigma_{\rho'_1}$ that is crossed  by
$f^{j_1}(\gamma_1)$ with positive oriented intersection number, and we  denote by $a^1_{\rho'_1},b^1_{\rho'_1}$  its endpoints.
Similarly, when the curve  $f^{j_1}(\gamma_1)$ crosses the gap between $a^1_{\rho'_1},b^1_{\rho'_1}$, it comes from below the circle $C_{\rho'_1}$.

We claim that the left endpoint of the lower gap is to the right of the right endpoint of the upper gap,  i.e., $\pi_x(a^1_{\rho_1})<\pi_x(b^1_{\rho'_1})$.
Otherwise, if $\pi_x(a^1_{\rho_1})\geq\pi_x(b^1_{\rho'_1})$, then there exists an arc $f^{j_1}(\gamma_1(s))$, $s\in[s_1,s_2]$, such that $\pi_x(f^{j_1}(\gamma_1(s_1))=\pi_x(f^{j_1}(\gamma_1(s_2))=\pi_x(a^1_{\rho_1})$, $\pi_y(f^{j_1}(\gamma_1(s_1))<\pi_y(f^{j_1}(\gamma_1(s_2))$, and $\pi_x(f^{j_1}(\gamma_1(s))\geq\pi_x(a^1_{\rho_1})$ for all $s\in(s_1,s_2)$. This implies that either the angle deviation from  the vertical $\theta(s)$ along the curve  $f^{j_1}(\gamma_1)$  becomes non-positive for some $s\in(s_1,s_2)$, or that there exists another arc   $f^{j_1}(\gamma_1(\tau))$, $\tau\in[\tau_1,\tau_2]$, with $\tau_1<\tau_2<s_1<s_2$, such that $\pi_x(f^{j_1}(\gamma_1(\tau_1))=\pi_x(f^{j_1}(\gamma_1(\tau_2))=\pi_x(a^1_{\rho_1})$, $\pi_y(f^{j_1}(\gamma_1(\tau_1))>\pi_y(f^{j_1}(\gamma_1(\tau_2))$, and $\pi_x(f^{j_1}(\gamma_1(s))\leq\pi_x(f^{j_1}(\gamma_1(\tau))$ for all $s\in(s_1,s_2)$ and $\tau\in(\tau_1,\tau_2)$.
In the first case we obtain a   contradiction with the positive tilt condition on the curve  $f^{j_1}(\gamma_1)$. See Fig. \ref{fig:violatetilt}. In the second case we obtain a contradiction with the fact that $f^{j_1}(\gamma_1)$ comes from below $C_{\rho_1}$ before crossing the gap of endpoints $a^1_{\rho_1},b^1_{\rho_1}$.
The conclusion of this step is that the curve $f^{j_1}(\gamma_1(t))$ passes through the gap  between $a^1_{\rho_1}$ and $b^1_{\rho_1}$ of  $\Sigma_{\rho_1}$,  from below $C_{\rho_1}$ to above  $C_{\rho_1}$,   then it
passes through the gap between $a^1_{\rho'_1}$ and $b^1_{\rho'_1}$ of  $\Sigma_{\rho'_1}$,  from below $C_{\rho'_1}$ to above  $C_{\rho'_1}$, and $\pi_x(a^1_{\rho_1})<\pi_x(b^1_{\rho'_1})$.

\begin{figure} \centering
\includegraphics*[width=0.75\textwidth, clip, keepaspectratio]
{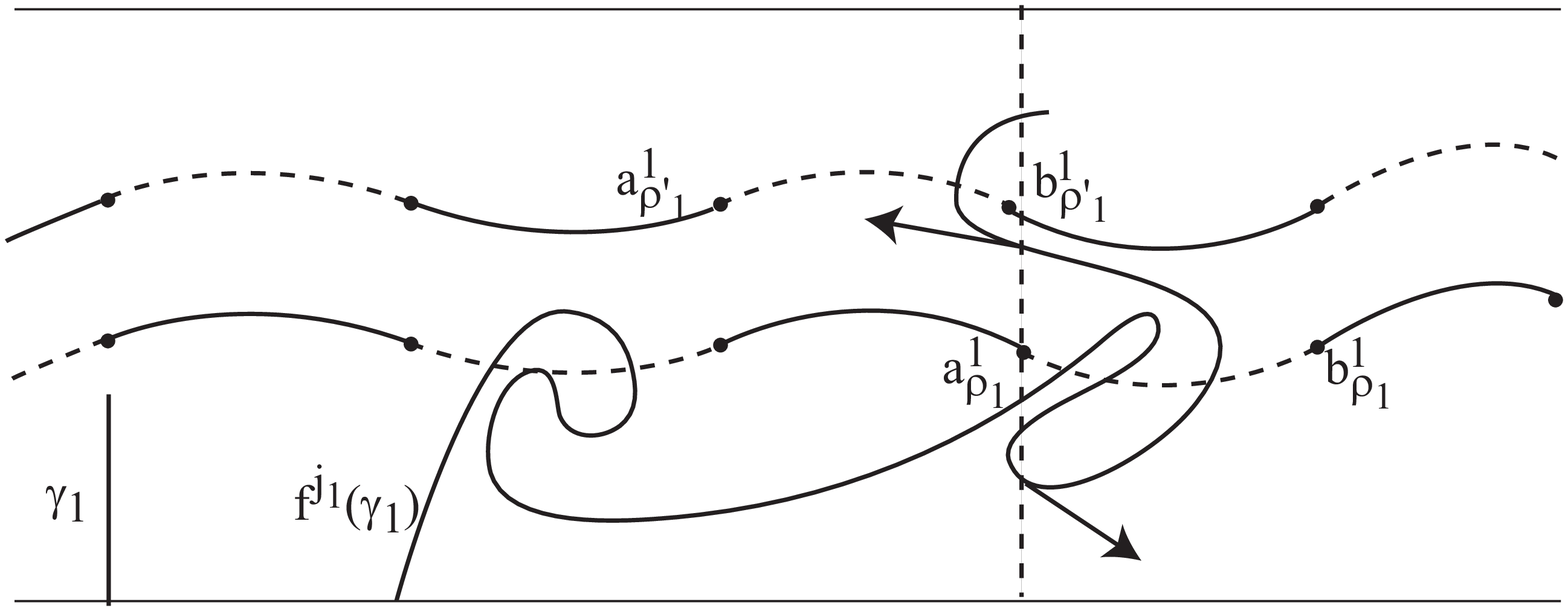}
    \caption[]{Violation of the positive tilt condition.}
    \label{fig:violatetilt}
\end{figure}

Now we consider a one-sided rectangular neighborhood $U_1\subseteq U_0$ of some point in $T_1$, bounded below by $T_1$, to the left by $\gamma_1$, and to the right by some other vertical curve segment $\gamma_1'$. If $\gamma'_1$ is sufficiently close to $\gamma_1$, then, by continuity,  the image of each vertical curve in $U_1$ under $f^{j_1}$  crosses the gap  between $a^1_{\rho_1}$ and $b^1_{\rho_1}$ of  $\Sigma_{\rho_1}$  with positive oriented intersection number, and crosses the gap between $a^1_{\rho'_1}$ and $b^1_{\rho'_1}$ of   $\Sigma_{\rho'_1}$ with positive oriented intersection number. We choose and fix a set $U_1\subseteq U_0$ with these properties.  By Lemma \ref{lem:kaloshin} (iii) (see also Remark \ref{rem:lemma}) the closure of $\bigcup _{j=0}^{\infty}f^{-j}(W(p_1))$ contains $T_1$, so there exists $j_2>j_1$ such that, in the annulus,  $f^{j_2}(U_1)\cap W(p_1)\neq \emptyset$, and, in the covering space of the annulus, $f^{j_2}(U_1)$ intersects some copy $W^{h_2}:=W(p_1)+(h_2,0)$ of $W(p_1)$ for some positive integer $h_2>h_1$. (Due to the positive twist condition on $f$ and the assumption  that the rotation numbers on the boundary components of the annulus are positive, the  vertical line $\{x=\pi_x(p_1)+h_1\}$ is mapped by $f^{j_2-j_1}$  to a positive tilted map to the right of $\{x=\pi_x(p_1)+h_1\}$, hence, if $j_2$ is large enough, we can choose $h_2>h_1$.)

Then there exists a vertical curve  $\gamma_2:[0,1]\to U_1$, such that $f^{j_2}(\gamma_2(1))$ is an intersection point of $f^{j_2}(U_1)$ with $W^{h_2}$. The curve $f^{j_2}(\gamma_2(t))$  crosses $C_{\rho_1}$ and $C_{\rho'_1}$. Let   $a^2_{\rho_1},b^2_{\rho_1}$  be the endpoints of the leftmost gap of $\Sigma_{\rho_1}$ that is crossed by $f^{j_2}(\gamma_2(t))$ with  positive oriented intersection number  equal, and let $a^2_{\rho'_1},b^2_{\rho'_1}$  be the endpoints of the  leftmost gap of $C_{\rho'_1}$ that is crossed by $f^{j_2}(\gamma_2(t))$ with positive oriented intersection number. The image curve  $f^{j_2}(\gamma_2)$ is a positively tilted curve located on the `right side' of the positively tilted curve $f^{j_2}(\gamma_1)$, in the sense that any graph over $x$ that intersects both $f^{j_2}(\gamma_1)$ and $f^{j_2}(\gamma_2)$ has the leftmost intersection point with the  $f^{j_1}(\gamma_1)$.
Therefore, the gap endpoints $a^2_{\rho_1},b^2_{\rho_1}$ are either  the image under $f^{j_2-j_1}$ of the gap endpoints $a^1_{\rho_1},b^1_{\rho_1}$ found at the previous step, or are the image
under $f^{j_2-j_1}$ of some other gap endpoints of $\Sigma_{\rho_1}$  located to the right of the gap between $a^1_{\rho_1}$ and $b^1_{\rho_1}$. Similarly, the gap endpoints  $a^2_{\rho'_1},b^2_{\rho'_1}$  are either the image under $f^{j_2-j_1}$ of the gap endpoints $a^1_{\rho'_1},b^1_{\rho'_1}$ from the previous step, or are the image
under $f^{j_2-j_1}$ of some other gap  of $\Sigma_{\rho'_1}$  located to the right of the gap between $a^1_{\rho'_1}$ and $b^1_{\rho'_1}$.

Then, there exists a one-sided rectangular neighborhood $U_2\subseteq U_1$ of some point in $T_1$, bounded below by $T_1$, to the left by $\gamma_2$, and to the right by some other vertical curve segment $\gamma'_2$, such that  the image of each vertical curve in $U_2$ under $f^{j_2}$   crosses the gap between $a^2_{\rho_1}$ and $b^2_{\rho_1}$ of  $\Sigma_{\rho_2}$  with positive oriented intersection number, and  crosses the gap between $a^2_{\rho'_1}$ and $b^2_{\rho'_1}$ of   $\Sigma_{\rho'_1}$ with positive oriented intersection number.

Recursively, we obtain a nested  sequence of one-sided neighborhoods of points in $T_1$, denoted $U_1\supseteq U_2\supseteq \ldots U_m\supseteq  \ldots$, all contained in $U_0$, and two  sequences of positive integers $j_1<j_2<\ldots<j_m<\ldots$ and $h_1<h_2<\ldots<h_m<\ldots$   with the following properties:
\begin{itemize}
\item [(i)] each set  $U_m$ is a topological rectangle consisting of vertical curves starting from $T_1$, bounded on the left-side by a vertical curve $\gamma_m$ and on the right by a vertical curve $\gamma'_m$;
\item [(ii)] $f^{j_m}(U_m)\cap W^{h_m}\neq \emptyset$, where $W^{h_m}:=W(p_1)+(h_m,0)$;
\item [(iii)] the image of each vertical curve in $U_m$ under $f^{j_m}$ crosses $C_{\rho_1}$ with positive oriented intersection number  through a gap between $a^m_{\rho_1}$ and $b^m_{\rho_1}$ of $\Sigma_{\rho_1}$, and it crosses $C_{\rho'_1}$ with positive  oriented intersection number through a gap between $a^m_{\rho'_1}$ and $b^m_{\rho'_1}$ of $\Sigma_{\rho'_1}$;
\item[(iv)]  $\pi_x(a^m_{\rho_1})<\pi_x(b^m_{\rho'_1})$;
\item [(v)] the gap endpoints $ a^m_{\rho_1},b^m_{\rho_1} $ are  the images under $f^{j_m-j_{m-1}}$ of  the gap endpoints $ a^{m-1}_{\rho_1},b^{m-1}_{\rho_1}$,  or of the endpoints of some other gap in  $\Sigma_{\rho_1}$ located to the right  side of this gap; the gap endpoints $ a^m_{\rho'_1},b^m_{\rho'_1}$ are  the images under $f^{j_m-j_{m-1}}$   of  the gap endpoints $a^{m-1}_{\rho_1}$, $b^{m-1}_{\rho'_1}$, or of the endpoints of some other gap in  $\Sigma_{\rho'_1}$ located to the right  side of this gap.
\end{itemize}

The endpoints of a gap of  $\Sigma_{\rho_1}$ or  $\Sigma_{\rho'_1}$ are mapped by $f$  into the endpoints of some other gap of $\Sigma_{\rho_1}$ or  $\Sigma_{\rho'_1}$, respectively. Also, the order of the gaps is preserved under iteration. The endpoints of the gap in  $\Sigma_{\rho_1}$   are iterated with rotation number $\rho_1$, and the endpoints of the gap  in  $\Sigma_{\rho'_1}$   are iterated with rotation number $\rho'_1>\rho_1$. Then, for some sufficiently large iterate $j_m$ the order of the gaps gets  reversed  in the annulus. That is, we  get the following ordering in terms of the angle coordinate in the covering space of the  annulus:
\begin{itemize}
\item[(i)] $\pi_x(a^m_{\rho_1})<\pi_x(b^m_{\rho_1})< \pi_x(a^1_{\rho_1})+h_m<\pi_x(b^1_{\rho_1})+h_m$,
\item[(ii)] $\pi_x(a^1_{\rho'_1})+h_m<\pi_x(b^1_{\rho'_1})+h_m<\pi_x(a^m_{\rho'_1})<\pi_x(b^m_{\rho'_1})$.
\end{itemize}

Since $f^{j_1}(U_1)$ and $f^{j_m}(U_m)$ are connected, the above ordering of the crossings with $C_{\rho_1}$ and $C_{\rho'_1}$ implies that  $f^{j_1}(U_1)$ and $f^{j_m}(U_m)$ have an intersection point above $C_{\rho_1}$. As $U_1, U_m\subseteq U_0$, it follows that $f^{j_1}(U_0)$ and $f^{j_m}(U_0)$ have an intersection point above $C_{\rho_1}$. Since $f^{j_1}(U_0)$ and $f^{j_m}(U_0)$ are connected and  disjoint from the Aubry-Mather set $\Sigma_{\rho_1}\subseteq C_{\rho_1}$, there exists a component $\mathcal{U}$ of the  complement $\mathcal{Z}\setminus [f^{j_1}(U_0)\cup f^{j_m}(U_0)]$ which is an open topological disk containing some point $\xi_1\in\Sigma_{\rho_1}$. The  boundary of $\mathcal{U}$ consists of  a finite union of sub-arcs of the boundaries of $f^{j_1}(U_0)$ and $f^{j_m}(U_0)$, and possibly  of curve segments of $T_1$. See Figure~\ref{arch2}. Letting $j_*=j_1$ and $j'_*=j_m$  ends the proof of the claim.

\begin{figure} \centering
\includegraphics*[width=0.6\textwidth, clip, keepaspectratio]
{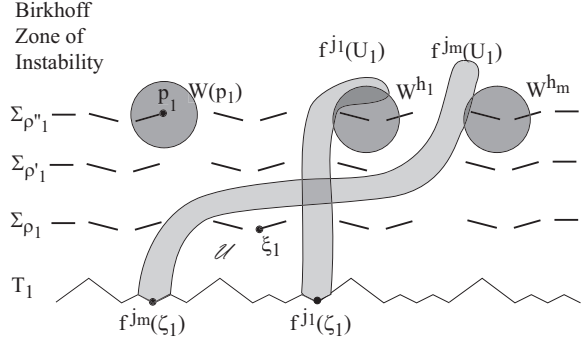}
    \caption[]{An open topological disk forming a neighborhood  of a point in the Aubry-Mather set $\Sigma_{\rho_1}$.}
    \label{arch2}
\end{figure}

We now apply the statement of the claim  to $T_2$, starting with the one-sided neighborhood $V_0$ of $\zeta_2\in T_2$.
We choose  three Aubry-Mather sets $\Sigma_{\rho_2}$, $\Sigma_{\rho_2'}$, $\Sigma_{\rho_2''}$ in $\mathcal{Z}$,  lying on three essential circles $C_{\rho_2},C_{\rho'_2},C_{\rho''_2}$, respectively, with ${\rho_2}>{\rho'_2}>{\rho''_2}$ irrational rotation numbers, and $C_{\rho_2} \succ C_{\rho'_2} \succ C_{\rho''_2}$.

The statement in the claim implies that there exists a neighborhood $\mathcal{V}$ of some point $\xi_2\in \Sigma_{\rho_2}$, homoeomorphic to an open disk,  whose boundary consists of  a finite union of sub-arcs in the boundaries of $f^{-l_*}(V_0)$ and $f^{-l'_*}(V_0)$, for some $l_*<l'_*$, and possibly a finite union of curve segments of $T_2$.

By Theorem \ref{thm:mathershadowing} there is an orbit that goes from  $\mathcal{U}$ to  $\mathcal{V}$, i.e., there exists $k_*$ such that $f^{k_*}(\mathcal{U})\cap \mathcal{V}\neq \emptyset$.
The boundary of $\mathcal{U}$ (resp. $\mathcal{V}$) is a simple closed curved separating the annulus into two connected components.
Also, the boundary of each of $f^{j_*}(U_0),f^{j'_*}(U_0),f^{-l_*}(V_0),f^{-l'_*}(V_0)$ is a simple closed curve.
Since $f^{k_*}(\mathcal{U})\cap \mathcal{V}\neq \emptyset$, the Jordan Curve Theorem implies that either $\partial f^{k_*}(\mathcal{U}) \cap \textrm{int}[f^{-l_*}(V_0)\cup f^{-l'_*}(V_0)] \neq\emptyset $ or  $\partial \mathcal{V} \cap \textrm{int}[f^{j_*}(U_0)\cup f^{j'_*}(U_0)]\neq\emptyset$.

It follows that   $\textrm{int}(f^{k_*+j_*}(U_0))$ or $\textrm{int}(f^{k_*+j'_*}(U_0))$ has a non-empty intersection with $\textrm{int}(f^{-l_*}(V_0))$ or $\textrm{int}(f^{-l'_*}(V_0))$. Thus, some forward iterate of $\textrm{int}(U_0)$  intersects $\textrm{int}(V_0)$. Hence there exists  $N>0$ such that $f^N(\textrm{int}(U _0))\cap\textrm{int}(V_0) \neq\emptyset$.   Since the sets $f^{N}(U_0)$ and $V_0$ are topological disks have interior points in common, but also points that are not in common (namely the points lying on $T_1$ and $T_2$, respectively), the Jordan Curve Theorem implies that   $f^{N}(\partial  U _0)\cap\partial V_0\neq\emptyset$.

The remaining case of the proof, when the interior of $U_0$  does intersect  some Aubry-Mather set and the interior of $V_0$ does not, or when the interior of $U_0$ does intersect  some Aubry-Mather set and the interior of $V_0$ does not, follows easily from the above arguments.
\end{proof}

The next statement says that given  two points on the boundary tori of  a BZI, and  a finite sequence of Aubry-Mather sets inside the zone, there exists an orbit that starts in a prescribed neighborhood of the point on the lower boundary torus, then moves on and shadows, in the sense of the ordering of the orbit,  each Aubry-Mather set in the sequence, and ends in a prescribed neighborhood of the point on the upper boundary torus. This result extends  Theorem \ref{thm:diagonal}, and relies on the topological argument of  Hall. As in the previous theorem, we do not need any extra conditions on the dynamics on the boundary tori. The resulting shadowing orbits are not necessarily minimal.

\begin{thm}\label{thm:extensionmathershadowing}
Suppose that $T_1$ and $T_2$ bound a BZI $\mathcal{Z}$. Let $\zeta_1\in T_1,\zeta_2\in T_2$, $U$ be a neighborhood of $\zeta_1$, and $V$ a neighborhood of $\zeta_2$. Let $\{\Sigma_{\omega_s}\}_{s\in\{1,\ldots,\sigma\}}$  be a finite collection  of Aubry-Mather sets inside $\mathcal{Z}$ such that each $\Sigma_{\omega_s}$ lies on some essential circle $C_{\omega_s}$ that is a graph over the $x$-coordinate, with $C_{\omega_s}\prec C_{\omega_{s'}}$ provided $\omega_s<\omega_{s'}$. Let  $\{n_s\}_{s=1,\ldots, \sigma}$ be sequence of positive integers. Then there exist   a point $z\in U$,  and a sequence of positive integers $\{m_s\}_{s=0,\ldots,\sigma}$,  such that, for each $s\in\{1,\ldots,\sigma\}$,
\begin{equation}\label{eqn:hallorder}\begin{split}\pi_x(f^j(w_s))<\pi_x(f^j(z))<\pi_x(f^j(\bar w_s))\textrm{ for } \\ \sum_{t=1}^{s-1} n_t+\sum_{t=0}^{s-1}m_t\leq j\leq \sum_{t=1}^{s} n_t+\sum_{t=0}^{s-1}m_t,\end{split}\end{equation}
where $w_s$ and $\bar w_s$ are some points in the Aubry-Mather set $\Sigma_{\omega_s}$,
and $f^{N}(z)\in V$ for $N=\sum_{t=1}^{\sigma} n_t+\sum_{t=0}^{\sigma}m_t$. The number $N$ can be chosen arbitrarily large.

Moreover, if $U,V$ are chosen so that $U\cap\mathcal{Z},V\cap\mathcal{Z}$ are topological disks, then there exists  a point $z'\in \partial  U$ satisfying the ordering condition \eqref{eqn:hallorder} such that $f^N(z')\in \partial V$.
\end{thm}
\begin{proof}

We use the construction of diagonal sets  described in the sketch of the proof of Theorem \ref{thm:diagonal}; for details see  \cite{Hall1989}.

\textit{Part 1.}
As in the proof of Theorem \ref{thm:extensionbirkhoff2} we can choose one sided compact neighborhoods $U_0$ of $\zeta_1\in T_1$ and  $V_0$ of $\zeta_2\in T_2$, such that $U_0\cap \mathcal{Z}$ and $V_0\cap \mathcal{Z}$  are topological disks.

We  first prove the existence of  a point $z\in U_0$ satisfying \eqref{eqn:hallorder} and such that $f^N(z)\in V_0$.
Let  $C_{\omega_1}$ be the essential circle containing $\Sigma_{\omega_1}$. We choose an Aubry-Mather set $\Sigma_{\rho_1}$ lying on some essential circle $C_{\rho_1}$, such that $C_{\omega_1}\prec C_{\rho_1}$. We choose a point $p_1\in\Sigma_{\rho_1}$ and a small neighborhood $W(p_1)$ of $p_1$ which does not intersect $\Sigma_{\omega_1}$.  We assume that the interior of $U_0$  does not meet  $\Sigma_{\omega_1}$ and $\Sigma_{\rho_1}$, otherwise the proof follows as in \cite{Hall1989}. By Lemma \ref{lem:kaloshin} (iii) the closure of $\bigcup _{j=0}^{\infty}f^{-j}(W(p_1))$ contains $T_1$, and in particular $\zeta_1$. Proceeding as in the proof of Theorem \ref{thm:extensionbirkhoff2}, we obtain a nested sequence $U_1\supseteq U_2\supseteq \ldots \supseteq U_i$  of one-sided neighborhoods of points in $T_1$,  all contained in $U_0$,  and two sequences of positive integers $j_1<j_2<\ldots<j_i<\ldots$  and $h_1<h_2<\ldots<h_i<\ldots$ with the following properties:
\begin{itemize}
\item [(i)] each set  $U_i$ is a topological rectangle consisting of vertical curves starting from $T_1$, bounded on the left-side by a vertical curve $\gamma_i$ and on the right by a vertical curve $\gamma'_i$;
\item [(ii)] $f^{j_i}(U_i)\cap W^{h_i}\neq \emptyset$, where $W^{h_i}:=W(p_1)+(h_i,0)$;
\item [(iii)] the image of each vertical curve in $U_i$ under $f^{j_i}$ crosses $C_{\omega_1}$ with positive oriented intersection number through a gap in $\Sigma_{\omega_1}$ of endpoints $a^i_{\omega_1}$ and $b^i_{\omega_1}$; the gap is chosen to be the first gap that is crossed  over with positive oriented intersection number;
\item[(iv)] the  endpoints of the gap between $a^i_{\omega_1}$ and $b^i_{\omega_1}$ are  the images under $f^{j_i-j_{i-1}}$ of either the  endpoints of the gap between $a^{i-1}_{\omega_1}$ and  $b^{i-1}_{\omega_1}$, or of a gap in  $\Sigma_{\omega_1}$ located to the right  side of that gap.
\end{itemize}

Since the rotation number of $\Sigma_{\omega_1}$ is smaller than the rotation number  of $\Sigma_{\rho_1}$,  any pair of points chosen on these two sets  shift apart from one another under positive iterations. Therefore there exists some $i$ large enough so that the gap  of endpoints $a^i_{\omega_1}$ and $b^i_{\omega_1}$ is on the left side of    $W^{h_i}$, in the sense that  $\pi_x(b^i_{\omega_1})<\pi_x(z)$ for all $z\in W^{h_i}$.

We claim that, by choosing $i$ large enough and $\gamma'_i$ sufficiently close to $\gamma_i$, we can ensure that the set $f^{j_i}(U_i)$ has a part which is a positive diagonal set in $B_{a^{i}_{\omega_ 1},b^{i}_{\omega_1}}$. Now we justify the claim. The image of the left-side $\gamma_i$ of $U_i$ is mapped by $f^{j_i}$ onto a positively tilted curve that crosses the gap between   $a^i_{\omega_1}$ and $b^i_{\omega_1}$ with positive oriented intersection number. Cutting the curve $f^{j_1}(\gamma_1)$ with the vertical strip $B_{a^i_{\omega_1},b^i_{\omega_1}}$ yields at least one component that connects $I_{a^i_{\omega_1}}=\{x=a^i_{\omega_1}\}$ to $I_{b^i_{\omega_1}}=\{x=b^i_{\omega_1}\}$ with positive oriented intersection number. Take the first such a component and follow it in the direction of the increase of the parameter $t$. The   intersection of this component with
 $I_{a^i_{\omega_1}}$ needs to occur  at a point $r_1$  below $a^i_{\omega_1}$, i.e.  $r_1\in I^{-}_{a^i_{\omega_1}}$,  otherwise this component does not come from below $C_{\omega_1}$. Following this component forward starting from $r_1$, the curve cannot intersect $I_{a^i_{\omega_1}}$ above $a^i_{\omega_1}$ as this would violate the positive tilt condition, or the choice of the gap between   $a^i_{\omega_1}$ and $b^i_{\omega_1}$ being the first gap that is crossed with positive oriented intersection number. Following the component starting from $r_1$, it must first meet    $I_{b^i_{\omega_1}}$  at a point $r'_1$  above $b^i_{\omega_1}$, i.e.   $r'_1\in I^{+}_{b^i_{\omega_1}}$; otherwise this component  does not have positive oriented intersection number   with the gap between $a^i_{\omega_1}$ and $b^i_{\omega_1}$.
Therefore, the component of $f^{j_i}(\gamma_i)$  between $r_1$ and $r'_1$   goes from $I^-_{a^i_{\omega_1}}$ to  $I^+_{b^i_{\omega_1}}$  without intersecting again $I^+_{a^i_{\omega_1}}$ or $I^-_{b^i_{\omega_1}}$. Now taking a curve $\gamma'_i$ sufficiently close to $\gamma_i$  results in a set $U_i$ with the property that $f^{j_i}(U_i)$ has a part which is a positive diagonal set in $B_{a^{i}_{\omega_ 1},b^{i}_{\omega_1}}$.

We change notation at this point:  we denote $m_0:=j_i$, $w_1:=a^{i}_{\omega_1}$, and $\bar w_1:=b^{i}_{\omega_1}$, $U'=U_i$ for $i$  fixed  as above. Thus, the points $w_1$ and $\bar w_1$ are the endpoints of a gap in $\Sigma_{\omega_1}$, and  $f^{m_0}(U')$ has a part which is  a positive diagonal  in $B_{w_1,\bar w_1}$. The positive integer $m_0$ is the first term of the sequence $\{m_s\}_{s=0,\ldots,\sigma}$ in the statement of the theorem.  Note that   $U'$ consists of a union of vertical segments emerging from $T_1$. See Figure~\ref{diagonal}.

\begin{figure} \centering
\includegraphics*[width=0.6\textwidth, clip, keepaspectratio]
{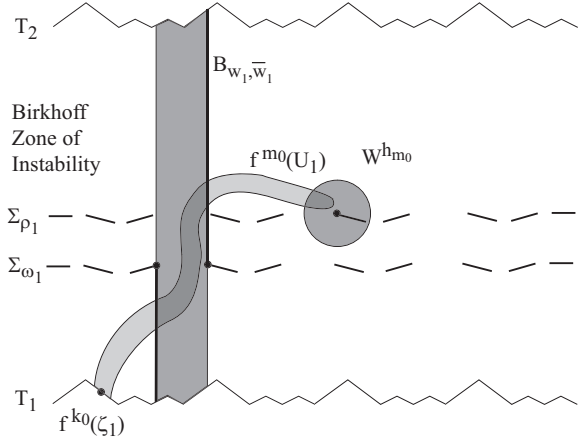}
    \caption[]{A  positive diagonal set.}
    \label{diagonal}
\end{figure}

Using the construction described in the sketch of the proof of Theorem \ref{thm:diagonal}, we obtain  a nested sequence $D_0\supseteq D_1\supseteq \ldots \supseteq D_\sigma$ of negative diagonals of $B_{w_1,\bar w_1}$ and a  sequence of  positive integers   $\{m_s\}_{s=0,\ldots,\sigma-1}$  such that for each $s\in\{1,\ldots,\sigma\}$ and each $z \in D_s$ we have
\begin{equation}\label{eqn:order}\pi_x(f^j(w_s))<\pi_x(f^j(z))<\pi_x(f^j(\bar w_s)) \textrm{ for } j_s\leq j\leq j_s+n_s,\end{equation}
where $j_s:=\sum_{t=1}^{s}n_t+\sum_{t=0}^{s-1}m_t$,  $w_s$ and $\bar w_s$ are the endpoints of some gap in the Aubry-Mather set $\Sigma_{\omega_s}$, and $f^{(j_s+n_s)}(D_s)$ is a positive diagonal in $B_{f^{n_s}(w_s),f^{n_s}(\bar w_s)}$.
In particular, $f^{(j_\sigma+n_\sigma)}(D_\sigma)$ is a positive diagonal in $B_{f^{n_\sigma}(w_\sigma),f^{n_\sigma}(\bar w_\sigma)}$, where  $w_\sigma$ and $\bar w_\sigma$ are the endpoints of some gap in the Aubry-Mather set $\Sigma_{\omega_\sigma}$.

Since $f^{m_0}(U')$ has a part which is  a positive diagonal in $B_{w_1,\bar w_1}$ and $D_{\sigma}$ is a negative diagonal  of  $B_{w_1,\bar w_1}$, then $f^{m_0}(U')$ and $D_{\sigma}$ have a non-empty intersection. Also, $f^{(j_\sigma+n_\sigma)}(U')$ has a part which is  a positive diagonal in $B_{f^{n_\sigma}(w_\sigma),f^{n_\sigma}(\bar w_\sigma)}$ that stretches across $f^{(j_\sigma+n_\sigma)}(D_\sigma)$.

Now we start with the one-sided  neighborhood $V_0$ of $\zeta_2\in T_2$. Let  $C_{\omega_\sigma}$ be an essential circle  containing $\Sigma_{\omega_\sigma}$. We choose
an Aubry-Mather set $\Sigma_{\rho_2}$ lying on an essential circle $C_{\rho_2}$, such that $C_{\rho_2}$ is below $C_{\omega_\sigma}$. We assume that the interior of $V_0$  does not meet  $\Sigma_{\omega_\sigma}$ and $\Sigma_{\rho_2}$, otherwise the proof follows as in \cite{Hall1989}.
Using  Lemma \ref{lem:kaloshin} (iii) and following the procedure described above for negative iterations, we produce a one-sided neighborhood $V'$ of a point in $T_2$, with $V'\subseteq V_0$, and a positive integer $m'_\sigma$ such that  $f^{-m'_\sigma}(V')$ contains a part which is a negative diagonal in $B_{w'_\sigma,\bar w'_\sigma}$, where   $w'_\sigma$ and $\bar w'_\sigma$  are the endpoints of a gap in $\Sigma_{\omega_\sigma}$.

By using the existence of orbits passing from near $T_2$ to near  $T_1$, and of orbits passing from  near $T_1$ to near $T_2$, as  in the sketch of the proof of Theorem \ref{thm:diagonal}, we can further iterate the positive diagonal $f^{(j_\sigma+n_\sigma)}(U')$ from above, so that we obtain an iterate $f^{(j_\sigma+n_\sigma+m''_\sigma)}(U')$  which contains a component  that stretches all the way across a fundamental interval of the annulus. In particular, $f^{(j_\sigma+n_\sigma+m''_\sigma)}(U')$ contains a component that  is a positive diagonal of $B_{w'_\sigma,\bar w'_\sigma}$. Since $f^{-m'_\sigma}(V')$ is a negative diagonal in $B_{w'_\sigma,\bar w'_\sigma}$, then $f^{(j_\sigma+n_\sigma+m''_\sigma)}(U')$ has a nonempty intersection with $f^{-m'_\sigma}(V')$. Equivalently, $f^{(j_\sigma+n_\sigma+m_\sigma)}(U')$ has a non-empty intersection with $V'$, where $m_\sigma:=m'_\sigma+m''_\sigma$.

Thus, each point   $z \in  U'\cap f^{-(j_\sigma+n_\sigma+m_\sigma)}(V')$  goes from the neighborhood $U_0$ of $\zeta_1$ to the neighborhood $V_0$ of $\zeta_2$ and it shadows, in the sense of the ordering,  each of the Aubry-Mather set $\Sigma_{\omega_s}$, $s=1,\ldots,\sigma$, along the way.

\textit{Part 2.}  Now we explain how to modify the above proof to show that there exists a point $z'\in \partial  U_0$ that satisfies \eqref{eqn:hallorder} and $f^{N}(z')\in \partial V_0$. This does not follow immediately from the above argument since the image of  $\partial U_0$ under iteration may fail being a positively tilted curve; hence  we cannot infer that $f^{m_0}(\partial U_0)$  intersects the negative diagonal set $D_\sigma$ from above.

By Theorem \ref{thm:extensionbirkhoff2}, there exist $l>0$ and a point $q\in U_0$ depending on $l$ such that $f^l(q)$ is in some prescribed neighborhood of a point $r\in T_2$, where $l$ can be chosen arbitrarily large.
Since the points of $T_1$ and $T_2$ have different rotation numbers hence move apart under iteration, there exists $l_0$ sufficiently large such that $f^{l_0}(T_1\cap U_0)$ and $f^{l_0}(q)$ are separated by a fundamental interval of the annulus, i.e., $\pi_x(f^{l_0}(r))-\pi_x(f^{l_0}(q))>1$ for all $r\in T_1\cap U_0$. Let $x_0,\bar x_0\in\mathbb{R}$ be such that $\pi_x(f^{l_0}(r))<x_0<\bar x_0 <\pi_x(f^{l_0}(q))$ and $1<\bar x_0-x_0$. Let $w_0$ be a point on the vertical line $\{x=x_0\}$ whose $y$-coordinate is larger than that of any point in $f^{l_0}(U_0)\cap  \{x={x_0}\}$. Similarly, let  $\bar w_0$ be a point on the vertical line $\{x=\bar x_0\}$ whose $y$-coordinate is smaller than that of any point in $f^{l_0}(U_0)\cap  \{x=\bar x_0\}$. Then $f^{l_0}(U_0)\cap \textrm{cl}(B_{w_0, \bar w_0})$ has a component that is a  positive diagonal in $B_{w_0, \bar w_0}$.

Let $\Sigma_{\omega_1}$ be the first   set in the prescribed sequence of Aubry-Mather sets.
Now we want to show  that, by slightly adjusting the vertical strip $B_{w_0, \bar w_0}$ to a new vertical strip $B_{z_0, \bar z_0}$,   there exists a diagonal set in $f^{l_0}(U_0)\cap B_{z_0, \bar z_0}$ such that a sufficiently large     iterate of this diagonal set has a component that is a positive diagonal in $B_{w_1,\bar w_1}$, for some points $w_1,\bar w_1\in \Sigma_{\omega_1}$.

There exists $x'_0$ sufficiently close to $x_0$ such that for all $x''_0$ between $x_0$ and $x'_0$, the point $(x''_0, \pi_y(w_0))$ has the $y$-coordinate larger than that of any point in $f^{l_0}(\textrm{cl}(U_0))\cap  \{x=x''_0\}$. Then the set  $W_0=\{(x''_0,y''_0)\,|\, x_0<x''_0<x'_0, \pi_y(w_0)<y''_0\}$  is a neighborhood of an arc  in $T_2$, with the property that each point $(x''_0,y''_0)\in W_0$ has the $y$-coordinate larger than that of any point in $f^{l_0}(\textrm{cl}(U_0))\cap \{x=x''_0\}$. Similarly, there is $\bar x'_0$ sufficiently close to $\bar x_0$ such that for all $\bar x''_0$ between $\bar x_0$ and $\bar x'_0$, the point $(\bar x''_0, \pi_y(\bar w_0))$ has the $y$-coordinate smaller than that of any point in $f^{l_0}(\textrm{cl}(U_0))\cap  \{x=\bar x''_0\}$. Then the set  $\bar W_0=\{(\bar x''_0,\bar y''_0)\,|\, \bar x_0<\bar x''_0<\bar x'_0, \pi_y(\bar w_0)>\bar y''_0\}$  is a neighborhood of an arc in $T_1$, with the property that each point $(\bar x''_0,\bar y''_0)\in \bar W_0$ has the $y$-coordinate smaller than that of any point in $f^{l_0}(\textrm{cl}(U_0))\cap \{x=\bar x''_0\}$.

Let $\Sigma_{\rho_1}$ be an Aubry-Mather set lying on an essential circle $C_{\rho_1}$ that is below the essential circle $C_{\omega_1}$ containing $\Sigma_{\omega_1}$,  let $p_1\in C_{\rho_1}$, and let $W(p_1)$ be a small neighborhood of $p_1$ that does not intersect $\Sigma_{\omega_1}$. Then there exists $z_0\in W_0$ and ${j_0}$ sufficiently large such that $f^{j_0}(z_0)\in W(p_1)$. Since the curve $f^{j_0}(I^+_{z_0})$ is a positively tilted curve emerging from $T_2$, the arguments  used in Part 1 show  that there is a gap of the Aubry-Mather set $\Sigma_{\omega_1}$, between a pair of points $w_1,w'_1\in \Sigma_{\omega_1}$, such that $f^{j_0}(I^+_{z_0})$  crosses this gap with negative intersection number (where the parametrization of $f^{j_0}(I^+_{z_0})$ is chosen so that to $t=0$ it corresponds a point on $T_2$) and has its first intersection with $I_{w_1}$ below the point $w_1$, provided ${j_0}$ is chosen large enough.
This implies that the image of the diagonal component of $f^{l_0}(\textrm{cl}(U_0))\cap B_{z_0,\bar w_0}$ under $f^{j_0}$ has a component $\Delta_0$ in $B_{w_1,w'_1}$  that satisfies the positive diagonal set conditions relative to its left side. See Fig. \ref{fig:part2}.

\begin{figure} \centering
\includegraphics*[width=0.6\textwidth, clip, keepaspectratio]
{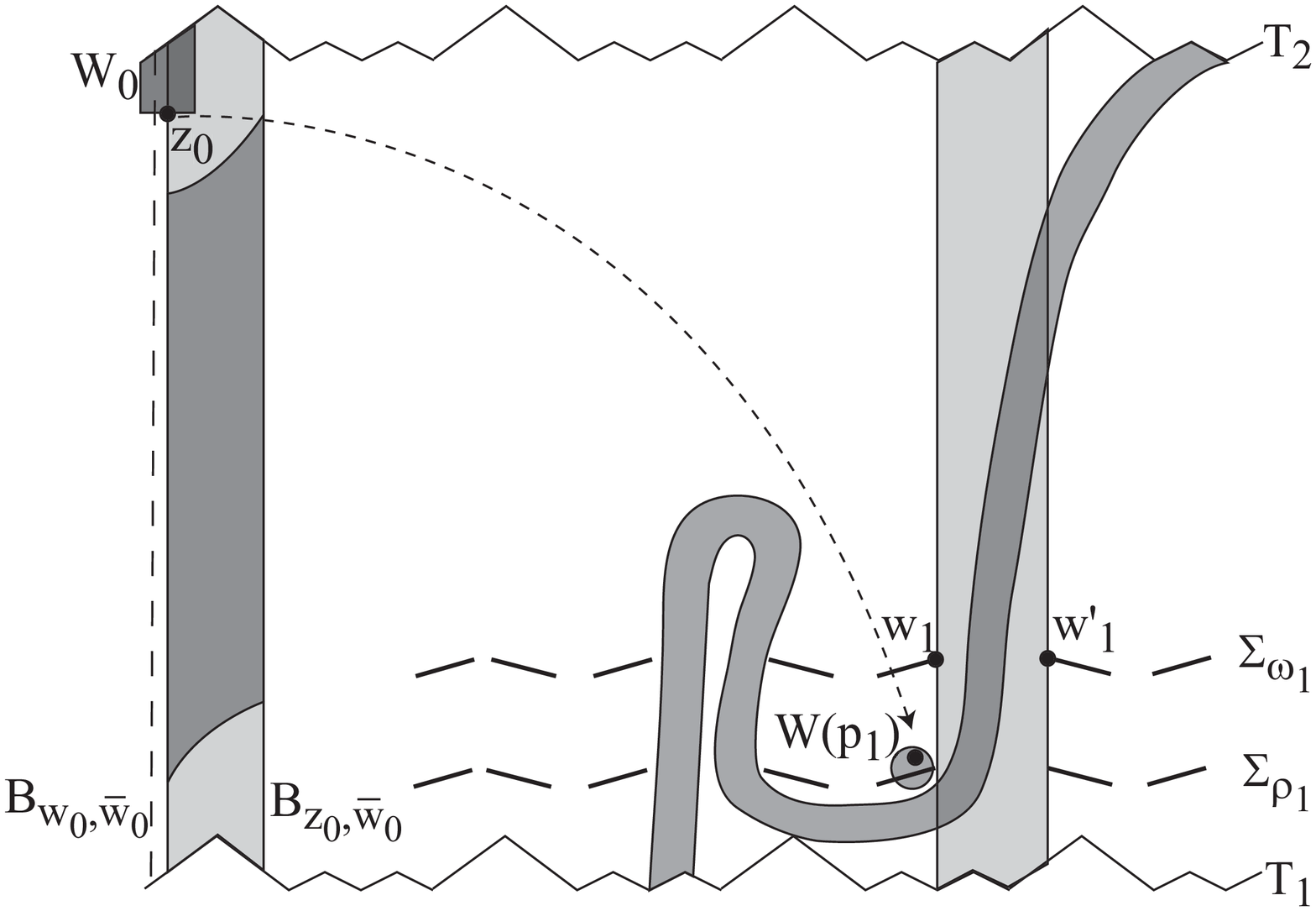}
    \caption[]{Construction of a positive diagonal set.}
    \label{fig:part2}
\end{figure}

Now, for all $j'>0$, the image of $\Delta_0$ under $f^{j'}$  also has a component  that satisfies the positive diagonal set conditions relative to the left side of $B_{f^{j'}(w_1),f^{j'}(w'_1)}$. (This follows from the hereditary property of diagonal sets, see Section \ref{section:aubrymathersets}.)
In a similar fashion, there exist a point $\bar z_0\in \bar W_0$ and  $j'_0$ sufficiently large such that  the  positively tilted curve  $f^{j'_0+j_0}(I^-_{\bar z_0})$, emerging from $T_1$,   crosses a gap of the Aubry-Mather set $\Sigma_{\omega_1}$  between a pair of points $\bar w_1,\bar w'_1\in \Sigma_{\omega_1}$,   the oriented intersection number between  $f^{j'_0+j_0}(I^-_{\bar z_0})$  and this gap is positive,  and the first intersection with $I_{\bar w'_1}$ occurs above the point $\bar w'_1$.
Thus, the image  of $\Delta_0$  under $f^{j'_0}$ has a component that satisfies the positive diagonal set conditions relative to its right side.

In summary, the image  of the diagonal set component of $f^{l_0}(\textrm{cl}(U_0))\cap B_{z_0, \bar z_0}$ under $f^{j_0+j'_0}$ contains a component $\Delta'_0$ that is a positive diagonal in $B_{w_1,\bar w'_1}$, where $w_1,\bar w'_1$ are two points in the Aubry-Mather set $\Sigma_{\omega_1}$. Moreover, the upper edge and the lower edge of  $\Delta'_0$ are contained in $f^{j'_0+j_0+l_0}(\partial  U_0)$.

We apply an analogous argument at the other boundary torus $T_2$. Given $V_0$ a neighborhood of a point $\zeta _2\in T_2$, there exist  $L>0$ and a pair of points $w_\sigma,\bar w_\sigma\in \Sigma_{\omega_\sigma}$ such that $f^{-L}(\textrm{cl}(V_0))\cap B_{w_\sigma,\bar w_\sigma}$ has a component $D''_\sigma$ that is a negative diagonal in $B_{w_\sigma,\bar w_\sigma}$. The upper edge and the lower edge  of   this positive diagonal set are contained in $f^{-L}(\partial  V_0)$.

Now we apply the argument from Part 1. There exists a negative diagonal set $D_\sigma$  in $B_{w_1,\bar w_1}$ such that all points $z\in D_\sigma$  satisfy \eqref{eqn:hallorder}. By the argument for Theorem \ref{thm:diagonal}, the upper and lower edge of $D_{\sigma}$ lie on $f^{-j_\sigma}(I^-_{w_\sigma}),f^{-j_\sigma}(I^+_{\bar w_\sigma})$ where $j_\sigma=\sum_{s=0}^{\sigma}n_s+\sum_{s=0}^{\sigma-1}m_s$.
Since negative and positive diagonal sets in the same vertical strip always intersect, the negative diagonal set $D_\sigma$ intersects $\Delta'_0$, and in particular it intersects its upper and lower edges that are contained in
$f^{j'_0+j_0+l_0}(\partial  U_0)$. Iterating $\Delta'_0$ forward for $j_\sigma$ times yields a positive diagonal set $D'_\sigma$ in $B_{w_\sigma,\bar w_\sigma}$. The upper and lower edges of $D'_\sigma$ are contained in $f^{j_\sigma+j'_0+j_0+l_0}(\partial U_0)$. The positive diagonal set $D'_\sigma$ intersects the negative diagonal component $D''_\sigma$ of $f^{-L}(\textrm{cl}(V_0))\cap B_{w_\sigma,\bar w_\sigma}$. In particular the upper and lower edges of $D'_{\sigma}$, that are contained in $f^{j_\sigma+j'_0+j_0+l_0}(\partial  U_0)$, intersect the upper and lower edges of $D''_\sigma$ that are contained in $f^{-L}(\partial  V_0)$.

Thus, there exists a point  $z'\in\partial  U_0$ that is taken by $f^{j_\sigma+j'_0+j_0+l_0}$ to $\partial  V_0$ and satisfies the ordering relations  \eqref{eqn:hallorder}.
\end{proof}

\section{A shadowing lemma in normally hyperbolic invariant manifolds}
\label{section:shadowing}

In this section we present a shadowing lemma-type of result
saying that, given a sequence of windows within a normally hyperbolic
invariant manifold, consisting of pairs of windows correctly aligned under
the scattering  map,   alternating with pairs of windows correctly aligned under some iterate of the inner map,
then there exists a true
orbit in the full space dynamics that follows these windows.
This result reduces the construction of windows within the full dimensional phase
space to  the construction of lower dimensional windows  within the normally hyperbolic invariant manifold.

For this section, we assume a diffeomorphism $f:M\to M$ on a manifold $M$, and an  $l$-dimensional normally hyperbolic invariant manifold $\Lambda\subseteq M$ as in Subsection~\ref{section:scattering}.

In the subsequent sections,  we will apply Lemma \ref{lem:shadowing2} only in the case when $\Lambda$ is a $2$-dimensional normally hyperbolic invariant manifold in $M$, i.e., $l=2$. A more general version of this lemma which does not involve the scattering map, and some additional details and applications,  appear in \cite{DelshamsGR11}.

\begin{lem}\label{lem:shadowing2} Let $\{R_i,R'_i\}_{i\in\mathbb{Z}}$ be a bi-infinite
sequence of $l$-dimensional windows contained in  $\Lambda$. Assume that
the following properties hold for all $i\in\mathbb{Z}$:
\begin{itemize}
\item[(i)] $R_{i}\subseteq U^-$ and $R'_{i}\subseteq
U^+$.
\item[(ii)] $R_{i}$ is
correctly aligned with $R'_{i+1}$ under the scattering map
$S$.\item[(iii)] for  each pair $R'_{i+1},R_{i+1}$ and for each $L>0$ there exists $L'>L$  such that   $R'_{i+1}$ is correctly aligned
with   $R_{i+1}$ under the iterate $f_{\mid\Lambda}^{L'}$ of the restriction $f_{\mid\Lambda}$ of $f$ to $\Lambda$.
\end{itemize}
Fix any bi-infinite sequence of positive real numbers  $\{\eps_i\}_{i\in\mathbb{Z}}$.
Then there exist an orbit $(f^{n}(z))_{n\in\mathbb{Z}}$
of some point $z\in M$, an increasing sequence of integers
$(n_i)_{i\in\mathbb{Z}}$, and some sequences of positive integers $\{N_i\}_{i\in\mathbb{Z}}, \{K_i\}_{i\in\mathbb{Z}},
\{M_i\}_{i\in\mathbb{Z}}$,  such that, for all $i\in\mathbb{Z}$:
\[ \begin{split}
d(f^{n_i}(z),\Gamma)<\eps_i,\\
d(f^{n_i+N_{i+1}}(z), f_{\mid\Lambda}^{N_{i+1}}(R'_{i+1}))<\eps_{i+1},
\\d(f^{n_{i}-M_{i}}(z), f_{\mid\Lambda}^{-M_{i}}(R_{i}))<\eps_{i},\\
n_{i+1}=n_i+N_{i+1}+K_{i+1}+M_{i+1}.
\end{split}\]
\end{lem}

\begin{proof}
The idea of this proof is to `thicken' some appropriate iterates of the windows $R_{i},R'_{i}$ in
$\Lambda$ to  full dimensional windows $W_i,W'_i$ in $M$, so that
$\{W_i,W'_i\}_{i\in\mathbb{Z}}$ form a sequence of windows  that  are correctly aligned under some appropriate maps.
We start with a brief sketch of  the construction before we proceed to the formal proof.
We constructs some copies  $\bar R_{i},\bar R'_{i+1}$ in $\Gamma$ of $R_{i},R'_{i+1}$, respectively,  through the inverses of the wave maps (see Section \ref{section:scattering}).  We then expand the rectangles $\bar R_{i},\bar R'_{i+1}$ into the hyperbolic directions to produce a pair of  windows $\bar W_{i},\bar W'_{i+1}$, respectively, which are correctly aligned under the identity map. Then we take a backwards  iterate  of $\bar W_{i}$ such that $f^{-M_{i}} (\bar W_{i})$ is sufficiently close to $\Lambda$, and  we   construct a  new window $W_i$ about $f^{-M_i}(R_i)$ such that $W_i$ is correctly aligned with $\bar W_i$ under $f^{M_i}$. Similarly, we construct a   window $W'_{i+1}$ about $f^{N_i}(R'_{i+1})$ such that $\bar W'_{i+1}$ is correctly aligned with $W'_{i+1}$ under $f^{N_{i+1}}$. We are given that we can align $R'_{i+1}$ with $R_{i+1}$ under some high enough iterate. Hence we can align $W'_{i+1}$ with $W_{i+1}$ under some iterate $f^{K_{i+1}}$. This construction can be continued inductively.

Notationwise, the $N_i$'s are associated to   forward iterations along the stable manifold, the $M_i$'s to   backwards iterations along the unstable manifold, and the $K_i$'s to iterations following the inner dynamics  of $f$ restricted to $\Lambda$. See Fig. \ref{shadowing}.

\begin{figure} \centering
\includegraphics*[width=0.8\textwidth, clip, keepaspectratio]
{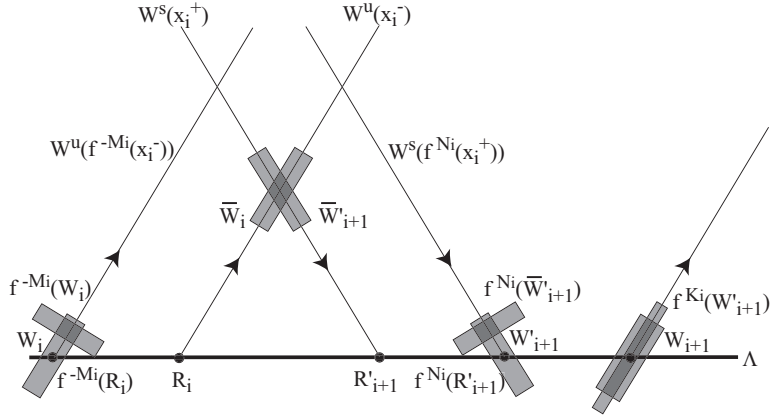}
    \caption[]{Schematic illustration of the construction of windows for the shadowing lemma. The windows $R_i,R'_{i+1}$ are depicted as large dots.}
    \label{shadowing}
\end{figure}

\textit{Step 1.} Let $\{R,R'\}$ be a pair of $l$-dimensional windows of the type $\{R_{i},R'_{i+1}\}$, and let $\{\eps,\eps'\}$ stand  for the corresponding $\{\eps_i,\eps_{i+1}\}$.
Let $\bar R =(\Omega^-_\Gamma)^{-1}(R )$ and
$\bar R' =(\Omega^+_\Gamma)^{-1}(R' )$ be the copies of $R $ and
$R' $, respectively, in the homoclinic channel $\Gamma$.  By
making some arbitrarily small changes in the sizes of their exit and entry directions, we
can alter the windows $\bar R $ and $\bar R' $ such that $R $ is correctly aligned with $\bar R $
under $(\Omega^-_\Gamma)^{-1}$, $\bar R$ is correctly aligned with $\bar R' $ under the identity mapping, and
$\bar R' $ is correctly aligned with $R $ under
$\Omega^+_\Gamma$.

We `thicken' the $l$-dimensional windows $\bar R $ and $\bar R' $
in $\Gamma$, which are correctly aligned under the identity mapping,
to $(l+n_u+n_s)$-dimensional windows $\bar W $ and $\bar W' $, respectively, that are correctly aligned in $M$ under the
identity mapping as well. We now explain the `thickening' procedure.

First, we describe how to thicken $\bar R $ to a full dimensional window $\bar W $.
We choose some $0<\bar\delta <\eps $ and $0<\bar\eta <\eps $. At each
point $x \in \bar R $ we choose an $n_u$-dimensional closed ball $\bar
B_{\bar\delta }(x )$ of radius $\bar\delta $ centered at $x $ and
contained in $W^u({x^- })$, where $x^- =\Omega^-_\Gamma(x )$. Let $\bar\Delta :=\bigcup_{x \in \bar R }\bar
B^{u}_{\bar\delta }(x )$. Note that $\bar\Delta $ is contained in
$W^u(\Lambda)$ and is homeomorphic to a $(l+n_u)$-dimensional
rectangle.  We define the exit set and the entry set of this
rectangle as follows:
\[\begin{split}(\bar\Delta )^{\rm exit}:=\bigcup_{x \in (\bar R )^{\rm exit}}
 \bar B^{u}_{\bar\delta }(x ) \cup
\bigcup_{x \in \bar R }\partial \bar B^{u}_{\bar\delta }(x ),\\
(\bar\Delta )^{\rm entry}:=\bigcup_{x \in (\bar R )^{\rm entry}}
 \bar B^{u}_{\bar\delta }(x ).
\end{split}\]

We consider the normal bundle $N$ to $W^u(\Lambda)$. At
each point $y \in \bar\Delta$,
we choose an $n_s$-dimensional closed ball $\bar B^s_{\bar\eta }(y)$
centered at $y$ and contained in the image of  $N_y$ under the
exponential map $\exp_{y }:T_{y }M\to M$. We let
$\bar W :=\bigcup_{y \in\bar\Delta }\bar B^{s}_{\bar\eta }(y )$.
By the Tubular Neighborhood Theorem (see, e.g.,
\cite{BurnsG05}), we have that for $\bar\eta $ sufficiently
small,  $\bar W  $ is a homeomorphic copy of a
$(l+n_u+n_s)$-dimensional rectangle. We now define the exit set and the
entry set of $\bar W $ as follows:
\[\begin{split} (\bar W  )^{\rm
exit}:=\bigcup_{y \in(\bar\Delta )^{\rm
exit}}\bar B^{s}_{\bar\eta }(y ),\\
(\bar W )^{\rm entry}:=\bigcup_{y \in(\bar\Delta )^{\rm
entry}}\bar B^{s}_{\bar\eta }(y ) \cup
\bigcup_{y \in(\bar\Delta )}\partial \bar B^{s}_{\bar\eta }(y ).
\end{split}
\]

Second, we describe in a similar fashion how to thicken $\bar R' $ to a full dimensional window
$\bar W' $. We choose $0<\bar\delta' <\eps' $ and
$0<\bar\eta' <\eps'$. We consider the $(l+n_s)$-dimensional rectangle
$\bar\Delta' :=\bigcup_{x' \in \bar R' }\bar
B^{s}_{\bar\eta' }(x')\subseteq W^s(\Lambda)$, where $\bar
B^s_{\bar\eta' }(x' )$ is the $n$-dimensional closed ball of radius
$\bar\eta' $ centered at $x' $ and contained in $W^s(x^+ )$, with
$x^+ =\Omega^+_\Gamma(x' )$. Its exit set and entry sets
are defined as follows:
\[\begin{split}(\bar\Delta' )^{\rm exit}:=\bigcup_{x \in (\bar R' )^{\rm exit}}
 \bar B^{s}_{\bar\eta' }(x' ),\\
(\bar\Delta' )^{\rm entry}:=\bigcup_{x' \in (\bar R' )^{\rm
entry}}
 \bar B^{s}_{\bar\eta' }(x' )\cup \bigcup_{x\in (\bar R' )}
\partial \bar B^{s}_{\bar\eta' }(x' ).
\end{split}\]

We define
$\bar W' :=\bigcup_{y' \in\bar\Delta' }\bar B^{u}_{\bar\delta' }(y' )$,
where $\bar B^u_{\bar\delta' }(y' )$ is the  $n$-dimensional closed ball
centered at $y' $ and contained in the image of  $N'_{y' }$ under the
exponential map $\exp_{y' }:T_{y' }M\to M$, with $N'$ being the normal bundle to $W^s(\Lambda)$.
For $\bar\delta' >0$
sufficiently small   $\bar W'$ is a homeomorphic copy of a
$(l+n_u+n_s)$-dimensional rectangle. The exit set and the entry set of
$\bar W' $ are defined by:
\[\begin{split} (\bar W' )^{\rm exit}:=\bigcup_{y' \in(\bar\Delta' )^{\rm
exit}}\bar B^{u}_{\bar\delta' }(y' )\cup
\bigcup_{y' \in(\bar\Delta' )}\partial \bar B^{u}_{\bar\delta' }(y' ),\\
(W' )^{\rm entry}:=\bigcup_{y' \in(\bar\Delta' )^{\rm
entry}}\bar B^{u}_{\bar\delta'_{i+1}}(y' ).
\end{split}
\]

This completes the description of the thickening of the
$l$-dimensional window $\bar R $ into a $(l+n_u+n_s)$-dimensional window
$\bar W $, and of the thickening of the $l$-dimensional window
$\bar R' $ into a  $(l+n_u+n_s)$-dimensional window $\bar W' $. Note that
by construction $\bar W $ is contained in an $\eps$-neighborhood of $\Lambda$ and  $\bar W' $ is contained in an
$\eps'$-neighborhood of $\Gamma$.

In order to make $\bar W $ correctly aligned with $\bar W' $
under the identity map, we choose
$\bar\delta' $ sufficiently small relative to $\bar\delta $, and  $\bar\eta $ sufficiently small relative to $\bar\eta' $.

\textit{Step 2.} We take a negative iterate $f^{-M}(\bar R)$ of $\bar R$, where $M>0$.
We have that $f^{-M}(\Gamma)$ is $\eps$-close to $\Lambda$ in the $C^1$-topology,  for all $M$ sufficiently large. The vectors tangent to the fibers $W^u(x^-)$ in $\bar R$ are contracted, and the vectors transverse to
$W^u(\Lambda)$ along $\bar R$ are expanded by the
derivative of $f^{-M}$. We choose  $M$ sufficiently large so that
$f^{-M}(\bar R)$ is  $\eps$-close to $f^{-M}(R)$.

We now construct a window $W$ about $f^{-M}(R)$ that is correctly
aligned with $f^{-M}(\bar W)$ under the identity. Note that each closed ball $\bar B^{u}_{\delta}(x)$, which is a part of  $\bar\Delta$, gets exponentially contracted  as it is mapped onto  $W^u(f^{-{M }}(x^-))$ by $f^{-{M }}$. By the Lambda Lemma (see the version in \cite{Marco08}), each closed ball $\bar B^s_{\eta }(y)$, $y\in \bar\Delta $, $C^1$-approaches some subset of $W^s(f^{-{M }}(y^-))$ under $f^{-{M }}$, as $M \to\infty$. For $M $ sufficiently large, we can assume that $f^{-{M }}(\bar B^s_{\bar\eta }(y))$  is $\eps $-close to some ball in $W^s(f^{-{M }}(y^-))$ in the $C^1$-topology, for all $y\in \bar\Delta  $. We fix $M$ with the above properties.
As $R $ is correctly aligned with $\bar R$ under $(\Omega^-_\Gamma)^{-1}$, we have that $f^{-M }(R )$ is correctly aligned with $f^{-M }(\bar R )$ under $({\Omega^-}_{f^{-M }(\Gamma)})^{-1}$. In other words, $f^{-M }(R )$ is correctly aligned under the identity mapping with the projection of $f^{-M }(\bar R )$ onto $\Lambda$ along the unstable fibers.

To define the window $W$,  we use a local linearization of the normally hyperbolic
invariant manifold.
By Theorem 1 in \cite{PughS70}, there exists a homeomorphism $h$ of an open neighborhood of
$(E^u\oplus E^s)_{ \mid \Lambda}$ to an open neighborhood of $\Lambda$ in $M$ such that
$h\circ Df = f\circ h$.
We select a point  $x  \in f^{-M}( R )$.
Since  $R$ is contractible the bundles are trivial on
$f^{-M}(R)$ and we can identify $(E^u\oplus E^s)_{\mid f^{-M }(R)}$ with
$f^{-M }(R) \times E^u_{x } \times E^s_{x }$.
Let us consider $0<\delta <\eps$ and $0<\eta <\eps $.
We define  a window $W $ as
\[W =h(f^{-M }( R ) \times \bar B^u_{\delta }(0) \times \bar B^s_{\eta }(0)),\]
where $\bar B^u_{\delta }(0)$ is the closed ball centered at $0$ of radius $\delta $ in
$E^u_{x }$ and
$\bar B^s_{\eta }(0)$ is the closed ball centered at $0$ of radius $\eta $ in
$E^s_{x }$.
We define the exit set of $W $ as
\[(W )^{\text{exit}}
= h(f^{-M }(\bar R )  \times \partial \bar B^u_{\delta }(0) \times \bar B^s_{\eta }(0))
\cup h(f^{-M}(\bar R^{\text{exit}}) \times B^u_{\delta}(0) \times \bar
B^s_{\eta}(0)).\]

Similarly, the entry set of $W$  is defined  as
\[(W)^{\text{entry}}
= h(f^{-M}(\bar R)  \times \bar B^u_{\delta}(0) \times \partial \bar B^s_{\eta}(0))
\cup h(f^{-M}(\bar R^{\text{entry}}) \times B^u_{\delta}(0) \times \bar
B^s_{\eta}(0)).\]

In order to ensure the correct alignment of $W$ with $f^{-M}(\bar W)$ under the identity map, we choose $\delta,\eta$ such that
$h(f^{-M}( R) \times \bar B^u_{\delta}(0) \times \{0\})$ is correctly aligned with $f^{-M }(\bar\Delta )$ under the identity map (the exit sets of both windows being in the unstable directions),
and that each closed ball $f^{-M }(\bar B^s_{\eta})$ intersects $W$ in a closed ball  that is contained in the interior of $f^{-M }(\bar B^s_{\eta })$. The existence of suitable $\delta,\eta$ follows from the
exponential contraction of  $\bar\Delta $ under negative iteration, and from the Lambda Lemma applied to  $\bar B^s_{\eta }(y)$ under negative iteration.

In a similar fashion, we construct a window $W'$   contained in an $\eps'$-neighborhood of $\Lambda$ such  that $\bar W'$ is correctly aligned with $W'$ under $f^{N}$.
The window $W'$, and its entry and exit sets, are  defined by:
\[\begin{split}
W' =&h(f^{N}( R') \times \bar B^u_{\delta'}(0) \times \bar B^s_{\eta'}(0)),\\
(W' )^{\text{exit}}
=& h(f^{N}( R') \times \partial \bar B^u_{\delta'}(0) \times \bar B^s_{\eta'}(0))\\&
\cup h(f^{N}( (R')^\text{exit}) \times \bar B^u_{\delta'}(0) \times \bar B^s_{\eta'}(0)),\\
(W' )^{\text{entry}}
=&h(f^{N}( R') \times \bar B^u_{\delta'}(0) \times \partial\bar B^s_{\eta'}(0))
\\&\cup h(f^{N}(( R')^\text{entry}) \times \bar B^u_{\delta'}(0) \times \bar B^s_{\eta'}(0)),
\end{split}\]
for some appropriate choices of radii $0<\delta' ,\eta' <\eps'$.

\textit{Step 3.} Suppose that we have constructed, as in Step 2,  a window $W' $ about the $l$-dimensional rectangle $f^{N }(R' )\subseteq \Lambda$, and a window $W $ about the $l$-dimensional rectangle $f^{-M }(R )\subseteq \Lambda$. Under positive iterations,  the rectangle $\bar B^u_{\delta' }(0)\times \bar B^s_{\eta'}(0)\subseteq E^u\oplus E^s$ gets exponentially expanded in the unstable direction and exponentially contracted in the stable direction by $Df$. Thus $\bar B^u_{\delta' }(0)\times \bar B^s_{\eta' }(0)$ gets correctly aligned with $\bar B^u_{\delta }(0)\times \bar B^s_{\eta }(0)$ under  some  power $Df^{L}$ of $Df$, provided $L $ is sufficiently large.
This implies that  $f^{L }(h(\{x \}\times \bar B^u_{\delta' }(0)\times \bar B^s_{\eta' }(0)))$ is correctly aligned with $h({f^{L }(x )}\times \bar B^u_{\delta }(0)\times \bar B^s_{\eta }(0))$ under the identity map (both rectangles are contained in $h({f^{L }(x )}\times E^u\times E^s)$.

By assumption (iii), there exists $L' >\max\{L ,N +M \}$ such that $R' $ is correctly aligned with  $R $ under $f^{L' }$. This means that $f^{N }(R' )$ is correctly aligned with $f^{-M }(R )$ under $f^{K }$ with  $K  :=L' -N -M >0$.

The product property of correctly aligned windows implies that $W' $ is correctly aligned with $W $ under $f^{K }$, provided that $K $ is chosen as above.

\textit{Step 4.}
We will now describe the process  of constructing,   based on Steps 1, 2, and 3, two bi-infinite sequences of windows $\{W_{i},W'_{i}\}_{i\in\mathbb{Z}}$ and $\{\bar W_{i},\bar W'_{i}\}_{i\in\mathbb{Z}}$ such that, for all $i\in \mathbb{Z}$,
$W_i$ is correctly aligned with $\bar W_i$ under $f^{N_i}$, $\bar W_i$ is correctly aligned with $\bar W'_{i+1}$ under the identity mapping, $\bar W'_{i+1}$ is correctly aligned with $W'_{i+1}$ under $f^{M_i}$, and $W'_{i+1}$ is correctly aligned with $W_{i+1}$ under $f^{K_{i+1}}$.
The point of this step is that we can repeatedly choose the rectangles and the parameters at Steps 1, 2, and 3, in a consistent way, in order to produce infinite sequences of correctly aligned windows.

Staring with $i=0$ and continuing for  all $i\geq 0$, we do the following.

For a  given pair of $l$-dimensional windows $\{R_i,R'_{i+1}\}$,  we consider the corresponding copies $\{\bar R_i,\bar R'_{i+1}\}$ in $\Gamma$ such that $\bar R_i$ is correctly aligned with $\bar R'_{i+1}$ under the identity. As in Step 1, we construct a pair of $(l+n_u+n_s)$-dimensional windows $\bar W_i$ about $\bar R_i$, and $\bar W'_{i+1}$ about $\bar R'_{i+1}$.
The size of $\bar W_i$ in the hyperbolic directions is given by some disks radii $\bar\delta_i,\bar\eta_i<\eps_i$,  and  that of $\bar W'_{i+1}$ by some disk radii $\bar\delta'_{i+1}, \bar\eta'_{i+1}<\eps_{i+1}$.
We choose the quantities $\bar\delta_i,\bar\eta_i, \bar\delta'_{i+1}, \bar\eta'_{i+1}$  such that  $\bar W_i$ is correctly aligned with $\bar W'_{i+1}$ under the identity.

Then, we  consider the pair of $l$-dimensional windows $\{R_{i+1},R'_{i+2}\}$ in $\Lambda$, and their corresponding copies
$\{\bar R_{i+1},\bar R'_{i+2}\}$ in $\Gamma$. As in Step  1, we construct a pair of $(l+n_u+n_s)$-dimensional windows $\bar W_{i+1}$ about $\bar R_{i+1}$, and $\bar W'_{i+2}$ about $\bar R'_{i+2}$. By choosing the quantities $\bar\delta_{i+1},\bar\eta_{i+1}, \bar\delta'_{i+2}, \bar\eta'_{i+2}$ as in Step 1 we can ensure that  $\bar W_{i+1}$ is correctly aligned with $\bar W'_{i+2}$.

We choose $N_{i+1},M_{i+1}$ large enough so that $f^{N_{i+1}}(\bar W'_{i+1})$ is contained in an $\eps_{i+1}$-neighborhood of $f^{N_{i+1}}(R'_{i+1})$, and $f^{-M_{i+1}}(\bar W_{i+1})$ is contained in an $\eps_{i+1}$-neighborhood of $f^{-M_{i+1}}(R_{i+1})$. As in Step 2, we construct a window $W'_{i+1}$ about $f^{N_1}(R'_{i+1})$ such that $\bar W'_{i+1}$ correctly aligned with $W'_{i+1}$ under $f^{N_{i+1}}$, and a window  $W_{i+1}$ about $f^{-M_1}(R_{i+1})$ such that $W_{i+1}$ correctly aligned with $\bar W_{i+1}$ under $f^{M_{i+1}}$. This amounts to choosing the quantities
$\delta'_{i+1}, \eta'_{i+1},\delta_{i+1},\eta_{i+1}$ as in Step 2 in order to ensure the correct alignment of the windows.

Then, we choose $K_{i+1}$ sufficiently large, and at least as large as $N_{i+1}+M_{i+1}$, such that  $W'_{i+1}$ is correctly aligned with $W_{i+1}$ under $f^{K_{i+1}}$. At this point, we have that $\bar W'_{i+1}$ is correctly aligned with $W'_{i+1}$ under $f^{N_{i+1}}$,
$ W'_{i+1}$ is correctly aligned with $W_{i+1}$ under $f^{K_{i+1}}$, and $W_{i+1}$ is correctly aligned with $\bar W_{i+1}$ under $f^{M_{i+1}}$.

Inductively, we obtain two sequences of windows $\{W_{i},W'_{i}\}_{i\geq 0}$  and $\{\bar W_{i},\bar W'_{i}\}_{i\geq 0}$ that satisfy the desired correct alignment conditions for all $i\geq 0$.  A similar inductive construction of windows can be done backwards starting with $W_0, W'_1$.

In the end, we obtain two bi-infinite sequences of windows $\{W_{i},W'_{i}\}_{i\in\mathbb{Z}}$ and $\{\bar W_{i},\bar W'_{i}\}_{i\geq 0}$ that satisfy the desired correct alignment conditions   for all $i\in\mathbb{Z}$. By Theorem \ref{thm:detorb}, there exists an orbit $\{f^n(z)\}_{n\in\mathbb{Z}}$ with $f^{n_i}(z)\in \bar W_i\cap \bar W'_{i+1}$, $f^{n_i+N_{i+1}}(z)\in  W'_{i+1}$, $f^{n_i+N_{i+1}+K_{i+1}}(z)\in  W_{i+1}$, $f^{n_i+N_{i+1}+K_{i+1}+M_{i+1}}(z)\in  \bar W_{i+1}\cap \bar W'_{i+2}$, for all $i\in\mathbb{Z}$. Thus $n_{i+1}=n_i+N_{i+1}+K_{i+1}+M_{i+1}$. Such an orbit $\{f^n(z)\}_{n\in\mathbb{Z}}$ satisfies the properties required by Lemma \ref{lem:shadowing2}.
\end{proof}

%\subsection{Crossing of the BZI's using the boundary tori}

\section{Construction of correctly aligned windows}
\label{section:proof}
In this section we prove Theorem \ref{thm:main1}. The methodology consists of constructing $2$-dimensional windows in $\Lambda$ about the prescribed invariant primary tori, BZI's, and Aubry-Mather sets inside the BZI's.  The successive  pairs of windows are correctly aligned under the scattering map alternatively with powers of the inner map.
Lemma \ref{lem:shadowing2} implies that there exist  trajectories that follow these windows.

\subsection{Construction of correctly aligned windows across a BZI}
\label{sec:windowsbirkhoff}

In this section we will construct correctly aligned windows across a BZI between two successive transition chains of tori.
On each side of the BZI we will choose a  one-sided neighborhood of a point on the boundary torus, and we will use Theorem
\ref{thm:extensionbirkhoff2} or Theorem \ref{thm:extensionmathershadowing}  to cross over the BZI.
These one-sided neighborhoods are of a special type: their boundaries are  images of some  transition tori under the inner or outer dynamics. We will construct some windows about the boundaries of these one-sided neighborhoods. Then we will construct  some other windows  about the corresponding transition tori; these windows will be chosen so that they have a pair of  sides  lying on some nearby tori.
We will use this feature later  to connect sequence of windows across the BZI's with sequences of windows along the transition chains.

Consider an annular region $\Lambda_k$ in $\Lambda$ that is a BZI, and is between two  transition chains of invariant tori, as in (A5).
To simplify notation, we denote the  tori at the boundary of  $\Lambda_k$ by $T_{a}$ and $T_{b}$. We choose a pair of transition tori $T_{i}, T_{j}$ in $\Lambda$ as in (A5),  ordered as follows:  $T_{j}\prec  T_{i}\prec  T_{a}$. These tori are   outside of the BZI $\Lambda_k$ and on the same side of it as $T_{a}$. By  (A5-iv)   there exist $T_{i'}\prec T_i$ and  $T_{j'}\prec T_j$ such that  $T_{i'}$ is $\eps_i$-close to $T_i$  and $T_{j'}$ is $\eps_j$-close to $T_j$, in the $C^0$ topology. By (A5-ii) $S(T_j)$ intersects $T_i$ in a topologically transverse  manner, so both $S(T_j)$ and $S(T_{j'})$ intersect  both $T_i$ and $T_{i'}$ in a topologically transverse  manner, provided $T_{i'}, T_{j'}$ are sufficiently close to $T_i,T_j$, respectively.

Since $S$ is a diffeomorphism, $T_{i}, T_{i'}$ and the images of $T_{j}, T_{j'}$ under $S$ form a topological rectangle $D_{iji'j'}$ in $\Lambda$. This rectangle may not be contained in the domain $U^-$  of the scattering map $S$. Provided that we choose the tori
$T_{i}, T_{i'}$ sufficiently $C^0$-close  to one another,  and also $T_{j}, T_{j'}$ sufficiently $C^0$-close  to one another, the rectangle $D_{iji'j'}$ will be sufficiently small so that some iterate $\Fin^{K_a}$ of $\Fin$  takes the rectangle $D_{iji'j'}$ into a rectangle $\Fin^{K_a}(D_{iji'j'})$ inside $U^-$. This is possible since each torus intersects $U^-$ by (A5-i), and the motion on the  tori is topologically transitive by (A5-iii). The rectangle $\Fin^{K_a}(D_{iji'j'})$ has a pair of sides lying on the tori $T_{i}, T_{i'}$, and the other pair of sides on    $(\Fin^{K_a}\circ S)(T_{j}),(\Fin^{K_a}\circ S)(T_{j'})$. The curves $(\Fin^{K_a}\circ S)(T_{j}),(\Fin^{K_a}\circ S)(T_{j'})$  are topologically transverse to both $T_{i},T_{i'}$. By assumption (A5-ii), $S(T_i)$ topologically crosses $T_{a}$, so $S^{-1}(T_{a})$ topologically crosses $T_{i}$. We can ensure that the interior of $\Fin^{K_a}(D_{iji'j'})$ intersects $S^{-1}(T_{a})$ by choosing $K_a$ sufficiently large, and the tori in each pair
$T_{i}, T_{i'}$ and  $T_{j}, T_{j'}$ sufficiently $C^0$-close  to one another. This implies that the image of $\Fin^{K_a}(D_{iji'j'})$ under $S$ is a topological rectangle  in $\Lambda$ which intersects $T_{a}$, and the intersection of $S(\Fin^{K_a}(D_{iji'j'}))$ with $\Lambda_k$ forms a one-sided neighborhood in $\Lambda_k$ of some part of $T_{a}$. The boundary of  $S(\Fin^{K_a}(D_{iji'j'}))\cap \Lambda_k$ consists of arcs of the curves $T_{a}, S(T_i), S(T_{i'}),S\circ\Fin^{K_a}\circ S(T_j),S\circ \Fin^{K_a}\circ S(T_{j'})$.

We make a similar construction on the other side of the BZI $\Lambda_k$.
We choose a pair of transition tori $T_{k}, T_{l}$ with $ T_{b}\prec T_{k}\prec  T_{l}$, outside of the BZI $\Lambda_k$ and on the same side as $T_{b}$. We also choose $T_{k}\prec T_{k'}$ and  $T_{l}\prec T_{l'}$ such that  $T_{k'}$ is $\eps_k$-close to $T_k$  and $T_{l'}$ is $\eps_l$-close to $T_l$, and $S^{-1}(T_l), S^{-1}(T_{l'})$ are topologically transverse to  both $T_k,T_{k'}$, and so they form a  a topological rectangle $D_{klk'l'}$.  There exists $K_b$ sufficiently large such that $\Fin^{-K_b}(D_{klk'l'})\subseteq U^-$, its interior intersects $S(T_{b})$, and $S^{-1}(\Fin^{-K_b}(D_{klk'l'}))$  forms a one-sided neighborhood in $\Lambda_k$ of some part of $T_{b}$. The boundary of  $S^{-1}(\Fin^{K_b}(D_{klk'l'}))\cap \Lambda_k$ consists of arcs of the curves $T_{b}, S(T_k), S(T_{k'}),S^{-1}\circ\Fin^{K_b}\circ S^{-1}(T_l),S^{-1}\circ \Fin^{K_b}\circ S^{-1}(T_{l'})$.

\begin{figure} \centering
\includegraphics*[width=0.6\textwidth, clip, keepaspectratio]
{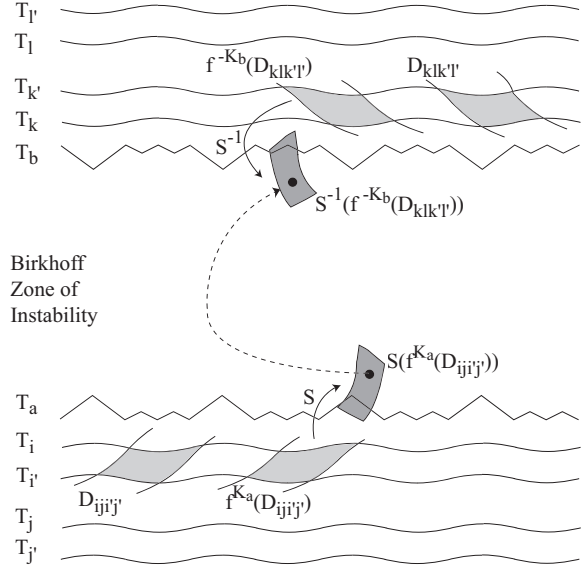}
    \caption[]{Orbits across a BZI}
    \label{transtorigaps.eps}
\end{figure}

At this stage we have obtained  in $\Lambda_k$ a one-sided neighborhood $(S\circ\Fin^{K_a})(D_{iji'j'})$  of an arc  in $T_{a}$, and a one-sided neighborhood $(S^{-1}\circ \Fin^{-K_b})(D_{klk'l'})$ of an arc in $T_{b}$.

If we are under the assumptions (A1)-(A6),  Theorem \ref{thm:extensionbirkhoff2} yields a point $x_a\in\partial (S\circ\Fin^{K_a})(D_{iji'j'})$ whose image $x_b=\Fin^{K_{ab}}(x_a)$ under some power $\Fin^{K_{ab}}$ lies on $\partial (S^{-1}\circ \Fin^{-K_b})(D_{klk'l'})$.

If we also assume (A7),   Theorem \ref{thm:extensionmathershadowing} yields a point $x_a\in\partial (S\circ\Fin^{K_a})(D_{iji'j'})$ with $x_b=\Fin^{K_{ab}}(x_a)\in\partial (S^{-1}\circ \Fin^{-K_b})(D_{klk'l'})$ as above, satisfying the additional conditions \[\pi_\phi(f^j(w^k_s))<\pi_\phi(f^j(x_a))<\pi_\phi(f^j(\bar w^k_s)), \]
for each $s\in\{1,\ldots, s_k\}$, where  $w^k_s,\bar w^k_s\in\Sigma_{\omega^k_s}$, and for all $j$ within a certain interval of integers. The trajectories of all points sufficiently close to $x_a$ will satisfy these conditions as well.

In either case, there exist an arc $\bar e'_a\subseteq \partial (S\circ\Fin^{K_a})(D_{iji'j'})$ containing $x_a$, and another arc $\bar e_b\subseteq \partial (S^{-1}\circ \Fin^{-K_b})(D_{klk'l'})$ containing $x_b$, such that $\Fin ^{K_{ab}}(\bar e'_a)$ is topologically transverse to  $\bar e_b$ at $x_b$.

The arc $\bar e'_a$ lies on one of the sets  $S({T_{i}}), S({T_{i'}}), (S\circ \Fin ^{K_{a}}\circ S)(T_{j}), (S\circ \Fin ^{K_{a}}\circ S)(T_{j'})$. Similarly, the arc $\bar e_b$ lies on one of the sets  $S^{-1}(T_{k}), S^{-1}({T_{k'}}), (S^{-1}\circ \Fin ^{-K_{b}}\circ S^{-1})(T_{l}), (S^{-1}\circ \Fin ^{-K_{b}}\circ S^{-1})(T_{l'})$.

We define a $2$-dimensional window $R'_a$ about $\bar e'_a$, and a $2$-dimensional window $R_b$ about $\bar e_b$  such that $R'_a$ is correctly aligned with $R_b$ under $\Fin^{K_{ab}}$. Informally, the exit direction of $R'_a$ is  along $\bar e'_a$, and the exit direction of $R_b$  is  across $\bar e_b$. The formal  construction now follows.
Since the arc $\bar e'_a$ is an embedded $1$-dimensional $C^0$-submanifold of $\Lambda$, there exists a $C^0$-local parametrization $\chi'_a:\mathbb{R}^2\to \Lambda$ such that $\chi'_a([0,1]\times \{0\})=\bar e'_a$,  provided $\bar e'_a$ is sufficiently small. Then
\[R'_a=\chi_a([0,1]\times [-\eta'_a,\eta'_a]),\]
is a topological rectangle.
We define the exit set of $R'_a$ as
\[{R'}_a^\exit=\chi_a(\partial[0,1]\times [-\eta'_a,\eta'_a]).\]
The definition of the entry set  of of $R'_a$ follows by default.

Similarly, there exists a $C^0$-local parametrization $\chi_b:\mathbb{R}^2\to \Lambda$ such that $\chi_b(\{0\}\times[0,1])=\bar e_b$, and
\[R_b=\chi_b([-\delta_b,\delta_b]\times [0,1]),\]
is a topological rectangle.
We define the exit set of $R_b$ as
\[{R}_b^\exit=\chi_b(\partial[-\delta_b,\delta_b]\times [0,1]).\]

By choosing $\eta'_a, \delta_b$ sufficiently small, we ensure that  $R'_a$ is correctly aligned with $R_b$ under $\Fin^{K_{ab}}$ (see Definition \ref{defn:corr}). See Figure \ref{transtorigaps4-2.eps}.

We now construct other windows outside the BZI $\Lambda_k$.  We consider two cases: first case, when the arc $\bar e'_{a}$ is a part  of $S(T_{i})$ or $S(T_{i'})$; second case, when $\bar e'_a$  is a part  of $(S\circ \Fin ^{K_a}\circ S)(T_{j})$ or $(S\circ \Fin ^{K_a}\circ S)(T_{j'})$.

\textsl{Case 1. } In the first case, when $\bar e'_a$ is a part  of $S(T_{i})$ or $S(T_{i'})$ we proceed with the construction as follows.
Let us say  that $\bar e'_a$ is a part  of $S(T_{i})$.  We take the inverse image of $R'_a$ by $S$. This is a topological rectangle $S^{-1}(R'_a)$ about the torus $T_{i}$. We construct a window $R_{i}$  about an arc $\bar e_{i}$ in $T_{i}$,  such that $R_{i}$ is correctly aligned with $S^{-1}(R'_{a})$ under the identity map, and with the exit direction of $R_{i}$ in the direction of  $T_i$.
Let $\bar e_{i}$ be an arc in $T_{i}$, and $\chi_{i}:\mathbb{R}^2\to\Lambda$  be a local parametrization with $\chi_{i}([0,1]\times\{0\})= \bar e_{i}$.
We define \begin{eqnarray*}
R_{i}=\chi_{i}([0,1]\times [-\eta_i,\eta_i])\\
{R}^\exit_{i}=\chi_{i}(\partial [0,1]\times [-\eta_i,\eta_i])
,\end{eqnarray*}
where $\eta_{i}>0$ is sufficiently small.
By choosing the arc $\bar e_{i}$ sufficiently large so that $\bar e_{i}\supseteq S^{-1}(\bar e'_a)$, and by choosing $\eta_{i}>0$ sufficiently small, we can ensure that $R_{i}$ is correctly aligned with $S^{-1}(R'_{a})$ under the identity map, or, equivalently, $R_{i}$ is correctly aligned with $R'_{a}$ under $S$.
See Figure \ref{transtorigaps4-2.eps}.

\begin{figure} \centering
\includegraphics*[width=0.6\textwidth, clip, keepaspectratio]
{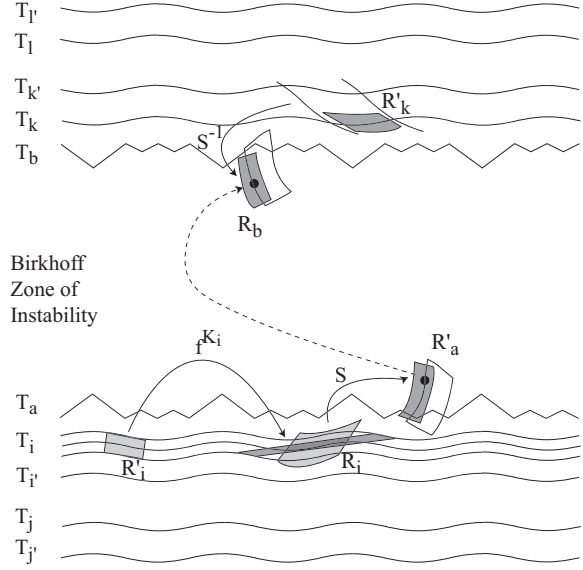}
    \caption[]{Construction of windows near the boundaries of a BZI -- case 1}
    \label{transtorigaps4-2.eps}
\end{figure}

We construct another window $R'_{i}$ about  $T_{i}$   such that $R'_{i}$ is correctly aligned with $R_{i}$ under some power $\Fin ^{K_{i}}$ of $\Fin$, with $K_{i}>0$. The window $R'_{i}$ will have its exit direction  across the torus $T_{i}$. We choose a pair of invariant primary tori $T^{\rm lower}_{i}\prec T_{i}\prec  T^{\rm upper}_{i}$, with $T^{\rm lower}_{i}, T^{\rm upper}_{i}$ located $\eps_{i}$-close to $T_{i}$ in the $C^0$-topology. Such tori exist due to the assumption that $T_{i}$ is not an end torus in the transition chain, as in (A5-iv).
Moreover, we choose the tori $T^{\rm lower}_{i}, T^{\rm upper}_{i}$  sufficiently close to $T_{i}$ so that the components of $R_{i}^{\rm entry}$ are outside the annulus between $T^{\rm lower}_{i}$ and $ T^{\rm upper}_{i}$, with one component on one side and the other component on the other side of this annulus. We choose an arc  $\bar e'_i\subseteq T_i$ that lies on an edge of the topological rectangle $D_{iji'j'}$.  Let   $\chi'_{i}:\mathbb{R}^2\to\Lambda$ be a $C^0$-local parametrization with $\chi'_{i}(\{0\}\times[0,1])= \bar e'_{i}$,  such that for some $\delta'_i$ sufficiently small, $\chi'_i (\{-\delta'_i\} \times[0,1])\subseteq T_i^{\rm lower}$ and $\chi'_i (\{\delta'_i\} \times[0,1])\subseteq T_i^{\rm upper}$.
We define $R'_{i}$ by:
\begin{eqnarray*}
R'_{i}=\chi'_i ([-\delta'_i,\delta'_i]\times[0,1]),\\
{R'}_{i}^\exit=\chi'_i (\partial[-\delta'_i,\delta'_i]\times[0,1]).
\end{eqnarray*}
The exit set components of $R'_i$ lie on the tori $T_i^{\rm upper}, T_i^{\rm lower}$ neighboring $T_i$.

Since the motion on the tori is topologically transitive, by (A5-iii), there exists $K_{i}>0$ such that $R'_{i}$ is correctly aligned with $R_{i}$ under $\Fin ^{K_{i}}$. Indeed, to achieve correct alignment of these windows in the covering space of the annulus we only have to choose $K_{i}$ sufficiently large so that  the two components of ${R'}_{i}^\exit$, which are two arcs in $T^{\rm lower}_{i}$ and $T^{\rm upper}_{i}$, are mapped by $\Fin^{K_{i}}$  on the  opposite sides  of the part  of $R_{i}$ between $T^{\rm lower}_{i}$ and $T^{\rm upper}_{i}$.  We note that the number of iterates $K_{i}$ of $\Fin$ needed to make $R'_{i}$ correctly aligned with $R_{i}$ may be different than the number of iterates $K_a$ which takes the topological rectangle $D_{iji'j'}$ onto $\Fin^{K_a}(D_{ihi'j'})$. See Figure \ref{transtorigaps4-2.eps}.
The conclusion of this step is that we obtain the window $R'_{i}$ around $T_{i}$, with its exit direction across $T_{i}$, such that $R'_{i}$ is correctly aligned with $R_{i}$ under $\Fin ^{K_{i}}$.
Both windows $R'_i,R_i$ are contained in an $\eps_i$-neighborhood of $T_i$.

In the case when the edge  $\bar e'_a$ of $D_{iji'j'}$ is a part  of $S(T_{i'})$ instead of $S(T_{i})$, the construction goes similarly to the one above.

\textsl{Case 2. }We now consider the second case, when the arc $\bar e'_a$  of $\partial (S\circ \Fin ^{K_a})(D_{iji'j'})$ is a part  of $(S\circ \Fin ^{K_a}\circ S)(T_{j})$ or $(S\circ \Fin ^{K_a}\circ S)(T_{j'})$. Let us say that $\bar e'_a$  is a part  of $(S\circ \Fin ^{K_a}\circ S)(T_{j})$. We construct a window $R'_a$ in $\Lambda$ about  $\bar e'_a$ as before; the exit set of $R'_a$ is in a direction along $\bar e'_a$, and the size of $R'_a$ in the direction across $R_a$ is given by some parameter $\eta'_a$. See Figure \ref{transtorigaps7.eps}. We consider  the inverse image $S^{-1}(R'_a)$ of $R'_a$ by $S$,   which is a topological rectangle about $(\Fin^{K_a}\circ S)(T_{j})$. Let $\bar e''_j$ be an arc in $(\Fin^{K_a}\circ S)(T_{j})$ such that $S(\bar e''_j)\supset \bar e'_a$, and let $\chi''_j :\mathbb{R}^2\to\Lambda$ be a local parametrization with $\chi''_{j}([0,1]\times\{0\})= \bar e''_{j}$.
We define $R''_{j}$ by:
\begin{eqnarray*}
R''_{j}=\chi''_j ([0,1]\times[-\eta''_j,\eta''_j]) ,\\
{R''}_{j}^\exit=\chi''_j (\partial [0,1]\times[-\eta''_j,\eta''_j]).
\end{eqnarray*}
By choosing  $\eta''_{j}>0$ sufficiently small, we can ensure that $R''_{j}$ is correctly aligned with aligned with $R'_{a}$ under $S$. The window $R''_{j}$ is  in a neighborhood of an edge   of the topological rectangle $\Fin^{K_a}(D_{iji'j'})$.   The exit set of $R''_{j}$ is in the direction of  $(\Fin^{K_a}\circ S)(T_{j})$. See Figure \ref{transtorigaps7.eps}.

\begin{figure} \centering
\includegraphics*[width=0.6\textwidth, clip, keepaspectratio]
{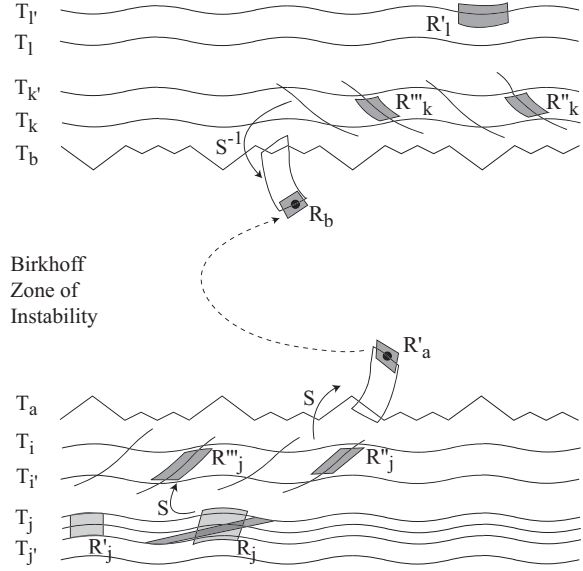}
    \caption[]{Construction of windows near the boundaries of a BZI -- case 2}
    \label{transtorigaps7.eps}
\end{figure}

We construct a window $R'''_{j}$ about  $S(T_j)$ such that
$R'''_{j}$ is correctly aligned with $R''_{j}$ under $\Fin ^{K_{a}}$.
We define \begin{eqnarray*}
R'''_{j} =\chi'''_j ([0,1]\times[-\eta'''_j,\eta'''_j]) ,\\
{R'''}_{j}^\exit=\chi'''_j (\partial[0,1]\times[-\eta'''_j,\eta'''_j]),\end{eqnarray*}
where $\bar e'''_{j}$ is an arc of $S(T_{j})$ with $\Fin^{K_a}(\bar e'''_{j})\supseteq \bar e''_j$, $\chi'''_j :\mathbb{R}^2\to\Lambda$ is a local parametrization with $\chi'''_{j}([0,1]\times\{0\})= \bar e'''_{j}$, and $\eta'''_j>0$ is sufficiently small. The arc $\bar e'''_{j}$ is contained in one of the edges of the topological rectangle $D_{iji'j'}$.
For suitable  $\eta'''_{j}>0$, we can ensure that $R'''_{j}$ is correctly aligned with aligned with $R''_{j}$ under $\Fin^{K_a}$. Moreover, we  choose  $R'''_{j},R''_{j}$ so  that these windows are both contained
in an $\eps_j$ neighborhood of $T_j$.
The exit set of $R'''_{j}$ is in a direction along $S(T_{j})$.
%Note that there is no switch from the  exit  directions of $R'''_{j}$ to the exit direction of $R''_{j}$, unlike the switch in  the exit directions  from  $R'_{i}$ to $R_{i}$ in the first case.

We take the inverse image $S^{-1}(R'''_{j})$ of $R'''_{j}$ under $S$, which is a topological rectangle about an arc in  $T_{j}$. In a fashion similar to Case 1, we construct two windows $R'_j,R_j$ about $T_j$ such that $R'_j$ is correctly aligned with $R_j$ under some power $\Fin^{K_j}$ of $\Fin$, and $R_j$ is correctly aligned with $R'''_j$ under $S$. The exit of $R'_j$ is chosen in a direction across $T_j$, and the exit set components ${R'}^{\rm exit}_j$ of
$R'_j$  lie on two invariant primary tori $T_j^{\rm upper}$, $T_j^{\rm lower}$ neighboring $T_j$, that are $\eps_j$-close to $T_j$ and contain $T_j$ between them. The exit of $R_j$ is in a direction along $T_j$.  The windows $R'_j,R_j$ are chosen to lie in an $\eps_j$ neighborhood of $T_j$ relative to the $C^0$-topology.

The case when the arc $\bar e'_a$  is a part  of $(S\circ \Fin ^{K_a}\circ S)(T_{j'})$ is treated similarly.

This concludes the construction of correctly aligned windows in $\Lambda$, starting with the window $R'_{a}$ about $T_{a}$, and moving backwards along transition tori that are on the same side as $T_{a}$ of the BZI $\Lambda_k$.

This construction yields   a sequence of windows of the type
\begin{eqnarray}\label{eqn:seqa1}
% \nonumber to remove numbering (before each equation)
 R'_{i}, R_{i}, R'_{a},
\end{eqnarray}
 in the first case, or
\begin{eqnarray}\label{eqn:seqa2}
 R'_{j}, R_{j}, R'''_{j}, R''_{j}, R'_{a},
\end{eqnarray} in the second case.
In the first case,  $R'_{i}$ is correctly aligned with $R_{i}$ under some power $\Fin ^{K_{i}}$ of the inner map, and $R_{i}$ is correctly aligned with $R'_{a}$ under the outer map $S$. The exit direction of  $R'_{i}$  is across the torus $T_{i}$, and its exit set components lie on some invariant primary tori that are $\eps_i$-close to $T_{i}$. The windows $R'_i,R_i$ are contained in an $\eps_i$ neighborhood of $T_i$.
In the second case,  $R'_{j}$ is correctly aligned with $R_{j}$ under some power $\Fin ^{K_{j}}$ of the inner map,  $R_{j}$ is correctly aligned with $R'''_{j}$ under the outer map $S$, $R'''_{j}$ is correctly aligned with $R''_{j}$ under  some power $\Fin ^{K_{a}}$ of the inner map, and $R''_{j}$ is correctly aligned with $R'_{a}$ under the outer map $S$. The exit direction of $R_{j}$ is  across the torus $T_{j}$, and its exit set components lie  on some invariant primary tori that are $\eps_j$-close to $T_{j}$. The windows $R'_j,R_j$ are contained in an $\eps_j$ neighborhood of $T_j$, and the windows $R'''_j,R''_j$ are contained in an $\eps_i$ neighborhood of $T_i$.

We proceed with a similar construction on the other side of the BZI between $\Lambda_k$, that is, on the same side of the BZI as $T_b$. We have already defined the window $R_b$ about $T_b$ that $R'_a$ is correctly aligned with $R_b$ under $\Fin^{K_{ab}}$.
Starting with the window $R_{b}$  and moving forward along the transition chain $T_b,T_k,T_l$, we construct a sequence of windows  of the type
\begin{eqnarray}
% \nonumber to remove numbering (before each equation)
R_{b},  R'_{k} \label{eqn:seqb1}
\end{eqnarray}
 or of the type
\begin{eqnarray}\label{eqn:seqb2}
 R_{b},  R'''_{k}, R''_{k}, R'_{l},
\end{eqnarray}
satisfying the correct alignment conditions below.
In the first case,  $R_{b}$ is correctly aligned with $R'_{k}$ under the outer map $S$. The exit direction of $R'_k$ is across $T_k$, and the exit set components lie on two invariant primary tori $\eps_k$-close to $T_k$. Moreover, $R'_{k}$ is contained in an $\eps_k$-neighborhood of $T_k$.
In the second case,  $R'_{b}$ is correctly aligned with $R'''_{k}$ under the outer map $S$, $R'''_{k}$ is correctly aligned with $R''_{k}$ under some power $\Fin ^{K_b}$ of the inner map, and $R''_{k}$ is correctly aligned with $R'_{l}$ under  the outer map $S$. The exit direction of  $R'_{l}$ is  across $T_{l}$, and its exit set components lie on some invariant primary tori that are $\eps_l$-close to $T_{l}$. The windows $R'''_{k}, R''_{k}$ are contained in an $\eps_k$-neighborhood of $T_k$, and $R'_l$ is contained in an $\eps_l$-neighborhood of $T_l$.

Similar statements apply when instead of windows around $T_{i},T_{j},T_{k},T_{l}$ we construct windows
about $T_{i'},T_{j'},T_{k'},T_{l'}$, respectively.

The conclusion of this section is that, by combining a sequence of correctly aligned window of the type \eqref{eqn:seqa1} or \eqref{eqn:seqa2} with a sequence of correctly aligned window of the type \eqref{eqn:seqb1} or \eqref{eqn:seqb2} we obtain a finite sequence of correctly aligned windows that crosses the BZI $\Lambda_k$. The shadowing lemma-type of result Lemma \ref{lem:shadowing2} yields an orbit that visits some prescribed neighborhood in the phase space of each window in $\Lambda$. In particular, the shadowing orbit has points that go  close to the transition tori.

\subsection{Construction of correctly aligned windows across annular regions separated by invariant tori}
\label{sec:windowsbzis}

We consider an annular region in $\Lambda$ between two  transition chains of invariant tori. Inside this annular region, we assume the existence of a finite collection of invariant tori that separate the region, as in (A6$'$-i), and are vertically ordered as in  (A6$'$-ii). Thus, the annular  region between the transition chains is not a BZI. We also assume that the scattering map satisfies  the   non-degeneracy condition (A6$'$-iii) on these invariant tori. The situation presented in this section is non-generic.

Assume that the annular region in $\Lambda$   is  bounded by a pair of invariant Lipschitz tori $T_{a}$ and $T_{b}$. Let  $T_a$ be the end torus of the transition chain of one side, and $T_b$ be the  end torus of the transition chain on the other side.  We assume that inside the region in the annulus bounded by $T_a$ and $T_b$ there exist a finite collection of invariant tori $\{\Upsilon_{h}\}_{h=1,\ldots,k-1}$. Each $\Upsilon_{h}$ is either an  isolated invariant  primary torus, or consists of a hyperbolic periodic orbit together with  the upper branches, or with the lower branches, of its stable and unstable manifolds;  the stable and unstable manifolds are assumed to coincide. In the second case, there is another invariant torus, say $\Upsilon_{h+1}$,  formed by  the remaining branches of the same hyperbolic periodic orbit   as for $\Upsilon_{h}$.  $\Upsilon_{h+1}$ is assumed to be above $\Upsilon_{h}$, relative to the $I$-coordinate, and is assumed to share with $\Upsilon_{h}$ only the points of the periodic orbit. The region in the annulus bounded by $\Upsilon_{h}$ and $\Upsilon_{h+1}$ is referred to as a resonant region.

Thus, the region in the annulus between $T_a$ and $T_b$ is divided into a finite number of BZI's and resonant regions.

The constructions in Subsection \ref{sec:windowsbirkhoff} provide a one-sided neighborhood of the type $(S\circ\Fin^{K_a})(D_{iji'j'})$ of some point in $T_a$, and a one-sided neighborhood  $(S^{-1}\circ \Fin^{-K_b})(D_{klk'l'})$ of some point in $T_b$. Let us denote $D_0=(S\circ\Fin^{K_a})(D_{iji'j'})$ and $D_{k+1}=(S^{-1}\circ \Fin^{-K_b})(D_{klk'l'})$.

Assume that $\Upsilon_1$ is an isolated invariant primary torus. We have that $S^{-1}(\Upsilon_{1})$ forms with $\Upsilon_{1}$  a topological disk $D_1$ between $T_a$ and $\Upsilon_1$, which is mapped by $S$ onto a topological disk $S(D_{1})$ between $\Upsilon_1$ and $\Upsilon_{2}$.
Theorem \ref{thm:extensionbirkhoff2} provides us a trajectory that starts from $\partial  D_0$ and ends at $\partial  D_1$. In particular, there exist $K_1>0$ and a component of $f^{K_1}(D_{0})\cap D_1$ that is a topological disk $D'_1$ whose boundary  contains an arc of $S^{-1}(T_1)$. The image of $D'_1$ under $S$ is a one-sided neighborhood $D''_1\subseteq S(D_1)$ of some point in $\Upsilon_1$, such that $D''_1$ is contained in the region between $\Upsilon_1$ and $\Upsilon_2$. The boundary of $S(D'_1)$ consists of an arc in $\Upsilon_1$ and an arc in $(S\circ f^{K_1}(\partial D_{0}))$. See Fig. \ref{fig:isolated}.
\begin{figure} \centering
\includegraphics*[width=0.6\textwidth, clip, keepaspectratio]
{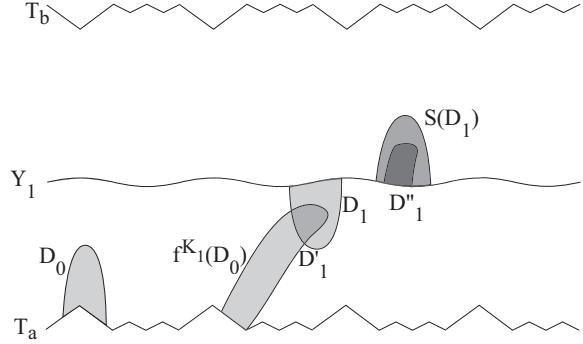}
    \caption[]{Crossing over an isolated invariant primary torus.}
    \label{fig:isolated}
\end{figure}

Now assume that $\Upsilon_1$ consists of a hyperbolic periodic orbit, together with the lower branches of its stable and unstable manifolds, and that $\Upsilon_2$ consists of the same hyperbolic periodic orbit, together with the upper branches of the  stable and unstable manifolds. The stable and unstable manifolds are assumed to coincide.  Excepting for the common points, $\Upsilon_{1}$ is below  $\Upsilon_{2}$. Thus, $\Upsilon_1$ and $\Upsilon _2$ enclose a resonant region within the annulus. We have that $S^{-1}(\Upsilon_{1})$ forms with $\Upsilon_{1}$  a topological disk $D_1$ between $T_a$ and $\Upsilon_1$, which is mapped by $S$ onto a topological disk $S(D_{1})$ between $\Upsilon_1$ and $\Upsilon_{2}$.
We also have that $S^{-1}(\Upsilon_{2})$ forms with $\Upsilon_{2}$  a topological disk $D_2$ between $\Upsilon_1$ and $\Upsilon_2$, which is mapped by $S$ onto a topological disk $S(D_{2})$ between $\Upsilon_2$ and $\Upsilon_3$.
By Theorem \ref{thm:extensionbirkhoff2} there exist  $K_1>0$ and a component of $f^{K_1}(D_{0})\cap D_1$ that is a topological disk $D'_1 $ whose boundary  contains an arc of $S^{-1}(\Upsilon_1)$.
The image of $D'_1$ under $S$ is a one-sided neighborhood $D''_1\subseteq S(D_1)$ of some point in $\Upsilon_1$, contained in the region between $\Upsilon_1 $ and $\Upsilon_2$. The boundary of ${D''}_1$ consists of an arc in $\Upsilon_1$ and an arc in $S(f^{K_1}(\partial D_{0}))$. By the argument in the Homoclinic Orbit Theorem (see e.g. \cite{BurnsW95}), there exists $N_1$ such that $f^{N_1}({D''}_1)$ intersects $D_2$. Let ${D'}_2$ be a component of this intersection that is a topological disk, and whose boundary contains an arc of $S^{-1}(\Upsilon_2)$. Then the image of ${D'}_2$ under $S$ is a one sided neighborhood $D''_2$ of some point in $\Upsilon_2$ that
is contained in the region between $\Upsilon_2$ and $\Upsilon_3$. The boundary of $D''_2$ consists of an arc in $\Upsilon_2$ and an arc in $(S\circ f^{N_1}\circ S\circ f^{K_1})(\partial D_{0})$. See Figure~\ref{fig:resonance.eps}.

\begin{figure} \centering
\includegraphics*[width=0.45\textwidth, clip, keepaspectratio]
{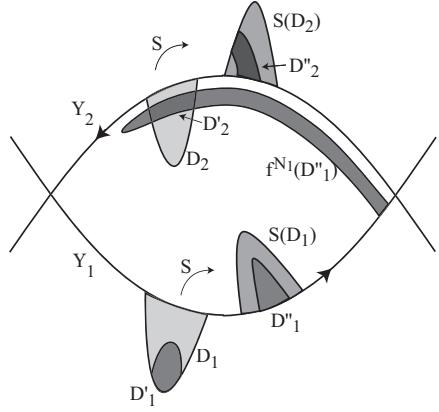}
    \caption[]{Construction of windows across a resonance.}
    \label{fig:resonance.eps}
\end{figure}

The main point of this construction is that, using the inner and outer dynamics, one can cross the successive BZI's and  resonance regions determined by  $\{\Upsilon_h\}_{h=1,\ldots,k}$ one at a time, and obtain at each step a  one-sided neighborhood of some point in some $\Upsilon_h$, which is  between $\Upsilon_h$ and $\Upsilon_{h+1}$, such that the boundary of that neighborhood contains the image of $\partial D_0$ under a suitable composition of $S$ and powers of $f$.

This argument can be repeated for each invariant torus $\Upsilon_h$ with $2\leq h\leq k$,  yielding  a point $x_a\in \partial D_{0}$ that is  mapped by some appropriate composition of $S$ and powers of $f$ onto a point $x_b\in \partial D_{k+1}$. Then the constructions from  Subsection \ref{sec:windowsbirkhoff} can  be applied to obtain a window $R'_a$ about  $\partial D_{0}$, and a window $R'_b$ about $\partial D_{k+1}$, such  that $R'_a$  is correctly aligned with $R_b$ under the appropriate composition of $S$ and powers of $f$. Then there exists a shadowing orbit that goes from a neighborhood of $R_a$ in $M$ to a neighborhood of $R'_b$ in $M$.

If the region between $T_a$ and $T_b$ contains some prescribed collection of Aubry-Mather sets, then Theorem \ref{thm:extensionmathershadowing} yields a shadowing orbit that crosses the region between $T_a$ and $T_b$ and follows the prescribed Aubry-Mather sets as in \eqref{eqn:main22}.

\subsection{Construction of correctly aligned windows along transition chains of tori}
\label{sec:windowstransitionchain}
We consider a finite transition chain of invariant primary tori $\{T_1,T_2,\ldots, T_n\}$ in the annulus $\Lambda$ satisfying (A5). All tori $T_k$, $k=1, \ldots, n$, intersect the subset $U^-_+$ of the domain of the scattering map $S$ where $S$ moves points upwards in the annulus $\Lambda$.
The motion on each torus $T_k$ is topologically transitive. Each torus $T_k$ with $2\leq k\leq n-1$ is an `interior' torus, i.e. it can be $C^0$-approximated from both sides in $\Lambda$ by other invariant primary tori.
For each $k\in\{1,\ldots,n-1\}$, the image $S(T_k)$ of the torus $T_k$ under the scattering map $S$ has a topologically transverse  intersection with the torus $T_{k+1}$.
We would like to show that there exists an orbit that visits some $\eps_k$-neighborhood of each torus $T_k$, $k=1,\ldots,n$,  in the prescribed order. To this end we will construct a sequence of $2$-dimensional windows in $\Lambda$ that are correctly aligned one with another under the scattering map alternatively with some iterates of the inner map, as in Lemma \ref{lem:shadowing2}. Then the lemma will imply the existence of a shadowing orbit to the transition chain.

For each $k\in\{1,\ldots,n-1\}$, we choose and fix  a pair of points $x^-_{k,k+1}\in T_k$ and $x^+_{k,k+1}\in T_{k+1}$  such that $S(x^-_{k,k+1})=x^+_{k,k+1}$ and $S(T_{k})$ intersects $T_{k+1}$ transversally at $x^+_{k,k+1}$.

We construct inductively a sequence of correctly aligned windows in $\Lambda$ along the tori $\{T_1,T_2,\ldots, T_n\}$, such that each window is correctly aligned with the next window in the sequence either by the outer map or by some sufficiently large power of the inner map. Moreover, each window will be contained in some $\eps$-neighborhood of some transition torus.
We start the inductive construction at $T_1$. Consider the point $x^-_{1,2}\in T_1$ with $S(x^-_{1,2})=x^+_{1,2}\in T_{2}$  as above.

We construct a window $R'_1$ about $T_1$ as follows. Let $\bar e'_1$ be an arc contained in $T_1$, and $\chi'_1:\mathbb{R}^2\to \Lambda$ a  $C^0$-local parametrization  such that $\chi'_1([0,1]\times \{0\})=\bar e'_1$. Then we define
\begin{eqnarray*}
R'_{1}=\chi'_1([0,1]\times [-\delta'_1,\delta'_1])
,\\ {R'}_{1}^\exit=\chi'_1(\partial [0,1]\times [-\delta'_1,\delta'_1]),\end{eqnarray*}
where $0<\delta'_1<\eps_1$.
We choose $\bar e'_1$ and $\delta'_1$ sufficiently small so that and $S(\bar e'_1)$ intersects $T_2$ only at $x^+_{1,2}$, and also so that $R'_1$ defined as above is contained in ${U}^-_+$. The exit direction of $R'_1$ is in the direction of the torus $T_1$.

The image $S(R'_1)$ of $R'_1$ under the scattering map is a topological rectangle. Since $S(\bar e'_1)$ is transverse to $T_2$ at $x^+_{1,2}$, then by choosing $\delta'_1$ sufficiently small we ensure that the components of $S({R'}_{1}^\exit)$ lie on opposite sides of $T_2$. Thus, the exit direction of $S(R'_1)$ is across the torus $T_2$.

Next we consider the  point $x^-_{2,3}\in T_2$ with $S(x^-_{2,3})=x^+_{2,3}\in T_{3}$  as above. We construct a window $R'_2$ about $T_2$ in a manner similar to $R'_1$:
\begin{eqnarray*}
R'_{2}=\chi'_2([0,1]\times [-\delta'_2,\delta'_2])
,\\ {R'}_{2}^\exit=\chi'_2(\partial [0,1]\times [-\delta'_2,\delta'_2]),\end{eqnarray*}
where $\bar e'_2$ is an arc contained in $T_2$, $\chi'_2:\mathbb{R}^2\to \Lambda$ is a  $C^0$-local parametrization  such that $\chi'_2([0,1]\times \{0\})=\bar e'_2$, and  $0<\delta'_2<\eps_2$ is chosen sufficiently small.

Now we construct a window $R_2$ about the point $x^+_{1,2}$ such that $R'_1$ is correctly aligned with $R_2$ under $S$ and
$R_2$ is correctly aligned with $R'_2$ under some power of $\Fin$. We choose an arc $\bar e_2$ in $T_2$ containing $x^+_{1,2}$ such that $\bar e_2 \supseteq S(R'_1)\cap T_2$. We choose a pair of invariant tori $T_2^{\rm lower}$ and $T_2^{\rm upper}$ such that $T_2^{\rm lower} \prec  T_2 \prec  T_2^{\rm upper}$ and the exit set components of $S(R'_1)$, as well as the entry set components of $R'_2$, are outside of the annulus bounded by  $T_2^{\rm lower}$ and $T_2^{\rm upper}$, on the both sides of the annulus. Furthermore, we require that  $T_2^{\rm lower}$ and $T_2^{\rm upper}$ should be $\eps_2$-close to $T_2$ in the $C^0$-topology. The existence of such neighboring tori $T_2^{\rm lower}$ and $T_2^{\rm upper}$ to $T_2$ is guaranteed by (A5-iv). Let
$\chi_2:\mathbb{R}^2\to \Lambda$ be a  $C^0$-local parametrization  such that $\chi_2( \{0\}\times[0,1])=\bar e_2$,
$\chi_2( \{-\delta_2\}\times[0,1])\subseteq T_2^{\rm lower}$ and
$\chi_2( \{\delta_2\}\times[0,1])\subseteq T_2^{\rm upper}$,
for some $0<\delta_2<\eps_2$  sufficiently small. Define
\begin{eqnarray*}
R_{2}=\chi_2([-\delta _2,\delta _2]\times [0,1])
,\\ {R}_{2}^\exit=\chi _2(\partial [-\delta _2,\delta _2]\times [0,1]).\end{eqnarray*}
The exit set components of $R_2$ lie on the invariant tori $T_2^{\rm lower},T_2^{\rm upper}$ neighboring $T_2$.

Since the motion on the torus $T_2$ is topological transitive, there exists a power $\Fin ^{K_2}$ of $\Fin$ such that $R_2$ is correctly aligned with $R'_2$ under $\Fin ^{K_2}$. See Figure \ref{transtori.eps}.  We have obtained the windows $R'_1$ about $T_1$ and $R'_2,R_2$ about $T_2$ such that $R'_1$ is correctly aligned with $R_2$ under $S$ and $R_2$ is correctly aligned with $R'_2$ under some power of $\Fin$. This ends the initial step of the inductive construction.

The inductive step goes on similarly, yielding the windows $R'_k$ about $T_k$ and $R'_{k+1},R_{k+1}$ about $T_{k+1}$ such that $R'_k$ is correctly aligned with $R_{k+1}$ under $S$ and $R_{k+1}$ is correctly aligned with $R'_{k+1}$ under some power of $\Fin$. See Figure \ref{transtori.eps}.
\begin{figure} \centering
\includegraphics*[width=0.5\textwidth, clip, keepaspectratio]
{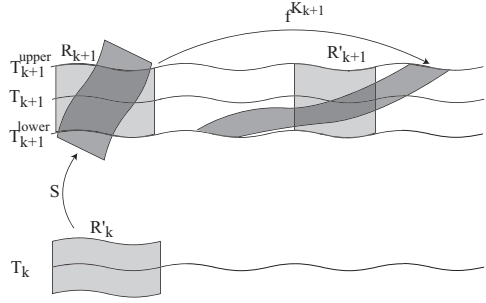}
    \caption[]{Construction of windows along a transition chain.}
    \label{transtori.eps}
\end{figure}

The construction proceeds inductively until a sequence of windows in $\Lambda$  of the following type is obtained
\[ R'_1,R_2,R'_2, \ldots, R_k,R'_k, R_{k+1}, R'_{k+1}, \ldots, R_n, \]
where for each $k\in\{1,\ldots,n-1\}$ we have that $R'_k$ is correctly aligned with $R_{k+1}$ under $S$, and $R_{k+1}$ is correctly aligned with $R'_{k+1}$ under $\Fin ^{K_{k+1}}$ for some $K_{k+1}>0$. Each window $R_k,R'_k$ is contained in an $\eps_k$-neighborhood of the torus $T_k$.  Applying the shadowing lemma-type of result Lemma \ref{lem:shadowing2} provides an orbit that visits an $\eps_k$-neighborhood in the phase space of each window in the sequence, and in particular of each torus in the transition chain.

\subsection{Gluing correctly aligned windows across BZI's with correctly aligned windows along transition chains of tori}\label{sec:windowsglue}

We consider three transition chains of tori $\{T_i\}_{i=i_{k-1}+1,\ldots, i_{{k}}}$, $\{T_i\}_{i=i_k+1,\ldots, i_{k+1}}$,  $\{T_i\}_{i=i_{k+1}+1,\ldots, i_{{k+2}}}$, with the property that each one of the regions between $T_{i_{k}}$ and $T_{i_{{k}}+1}$, and between $T_{i_{k+1}}$ and $T_{i_{k+1}+1}$, is either a BZI as in (A6), or it contains a finite number of invariant tori that separate the region as in (A6$'$). Inside each region there is a prescribed collection of Aubry-Mather sets as in (A7)

The constructions in Subsection \ref{sec:windowsbirkhoff} and in Subsection \ref{sec:windowsbzis} yield correctly aligned windows in $\Lambda$ that cross
the region between $T_{i_{k}}$ and $T_{i_{{k}}+1}$, and  correctly aligned windows in $\Lambda$ that cross the region between $T_{i_{k+1}}$ and $T_{i_{k+1}+1}$. The construction in Subsection \ref{sec:windowstransitionchain} yield sequences of correctly aligned windows along the adjacent transition chains. The choices of the windows constructed along the transition chains depend on the choices of the windows that cross the region between $T_{i_{k}}$ and $T_{i_{{k}}+1}$, and  of the windows that cross the region between $T_{i_{k+1}}$ and $T_{i_{k+1}+1}$. Propagating the construction of windows starting from $T_{i_k+1}$  and moving forward along the transition chain $\{T_i\}_{i=i_k+1,\ldots, i_{k+1}}$, and the construction of windows starting from $T_{i_{k+1}}$ and moving backward along the same transition chain
$\{T_i\}_{i=i_k+1,\ldots, i_{k+1}}$, may result in a pair of windows about some intermediate torus that are not correctly aligned. We would like to glue this sequences of windows in a manner that is correctly aligned, without having to revise the windows constructed to that point.

Assume that $T_j$ is one of the tori $\{T_i\}_{i=i_k+1,\ldots, i_{k+1}}$,  with $j\in\{i_k+2,\ldots, i_{k+1}-1\}$. Assume that one has already constructed a window $R_j$ about $T_j$ by propagating the construction  from $T_{i_k+1}$  and moving forward along the transition chain, and another window $R'_j$ about $T_j$ by propagating the construction  from $T_{i_{k+1}}$ and moving backwards along the transition chain.
The window $R_j$ is of the form
\begin{eqnarray*}
R _{j}=\chi _j ([-\delta _j,\delta _j]\times[0,1]),\\
{R }_{j}^\exit=\chi _j (\partial[-\delta _j,\delta _j]\times[0,1]),
\end{eqnarray*}
where $\chi _{j}:\mathbb{R}^2\to\Lambda$ is a $C^0$-local parametrization with $\chi _{j}(\{0\}\times[0,1])\subseteq T_j$,  and $\chi _j (\{-\delta _j\} \times[0,1])\subseteq T_j^{\rm lower}$ and $\chi _j (\{\delta _j\} \times[0,1])\subseteq T_j^{\rm upper}$, where $T_j^{\rm lower}$ and $T_j^{\rm upper}$ are two primary invariant tori on the opposite sides of $T_j$.
The window $R'_j$ is of the form
\begin{eqnarray*}
R'_{j}=\chi'_j ([0,1]\times[-\delta'_j,\delta'_j]),\\
{R'}_{j}^\exit=\chi'_j (\partial[0,1]\times[-\delta'_j,\delta'_j]),
\end{eqnarray*}
where $\chi'_{j}:\mathbb{R}^2\to\Lambda$ is a $C^0$-local parametrization with $\chi'_{j}([0,1]\times\{0\})\subseteq T_j$,  and $\chi'_j ([0,1]\times\{-\delta'_j\} )\subseteq T_j^{\rm lower}$ and $\chi'_j ([0,1]\times\{\delta'_j\} )\subseteq {T'}_j^{\rm upper}$, where ${T'}_j^{\rm lower}$ and ${T'}_j^{\rm upper}$ are two primary invariant tori on the opposite sides of $T_j$.

Let us assume that the annular region between ${T'}_j^{\rm lower}$ and ${T'}_j^{\rm upper}$ is inside the region between $T_j^{\rm lower}$ and $T_j^{\rm upper}$. We construct a new window $R''_j$ about $T_j$, such that $R_j$  is correctly aligned with $R''_j$ under the identity map, and $R''_j$ is correctly aligned with $R'_j$ under some power of $f$.
We let $R''_j$ is of the form
\begin{eqnarray*}
R''_{j}=\chi''_j ([-\delta''_j,\delta''_j]\times[0,1]),\\
{R''}_{j}^\exit=\chi''_j (\partial[-\delta''_j,\delta''_j]\times[0,1]),
\end{eqnarray*}
where $\chi''_{j}:\mathbb{R}^2\to\Lambda$ is a $C^0$-local parametrization with $\chi''_{j}(\{0\}\times[0,1])\supseteq \chi'_{j}(\{0\}\times[0,1])$,  and  $\chi''_j (\{-\delta''_j\} \times[0,1])\subseteq {T'}_j^{\rm lower}$ and $\chi''_j (\{\delta''_j\} \times[0,1])\subseteq {T'}_j^{\rm upper}$, for some $\delta''_j>0$.
Since the motion on the torus $T_j$ is topological transitive, there exists a power $\Fin ^{K_j}$ of $\Fin$ such that $R''_j$ is correctly aligned with $R'_j$ under $\Fin ^{K_j}$. See Figure \ref{glue.eps}.  We have obtained that  $R_j$  is correctly aligned with $R''_j$ under the identity map, and $R''_j$ is correctly aligned with $R'_j$ under some power $\Fin ^{K_j}$  of $\Fin$.

\begin{figure} \centering
\includegraphics*[width=0.5\textwidth, clip, keepaspectratio]
{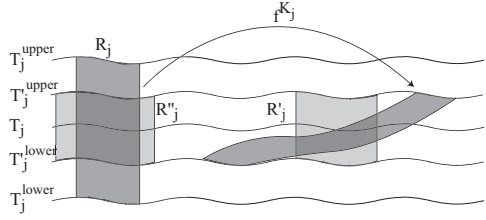}
    \caption[]{Gluing sequences of correctly aligned windows along a transition chain.}
    \label{glue.eps}
\end{figure}

The case when the annular region between $T_j^{\rm lower}$ and $T_j^{\rm upper}$ is inside the region between ${T'}_j^{\rm lower}$ and ${T'}_j^{\rm upper}$ results in the window $R_j$ correctly aligned with $R'_j$ under some power of $f$, without the need of creating an intermediate window $R''_j$. The case when neither annular region is contained in the other annular region can be dealt similarly by constructing an intermediate window $R''_j$.

Through this process, a sequence of correctly windows constructed forward along a transition chain can be glued, in a correctly aligned manner, with a sequence of correctly windows constructed backwards along the same transition chain.

\subsection{Proof of Theorem \ref{thm:main1}}\label{sec:proof}
To summarize, in Subsection \ref{sec:windowsbirkhoff} and in Subsection \ref{sec:windowsbzis}, we described the
construction of correctly aligned windows that cross an annular regions between two consecutive transition chains of invariant tori. These annular regions are BZI's or else they are separated by some finite collection of invariant tori. If each annular region has some prescribed collection od Aubry-Mather sets,  the construction yields  windows that follow these Aubry-Mather sets.
In Subsection  \ref{sec:windowstransitionchain} yield sequences of correctly aligned windows each transition chain.
In Subsection \ref{sec:windowsglue} the construction of a sequence of correctly windows constructed forward along a transition chain can be glued, in a correctly aligned manner, with a the construction of a sequence of correctly windows constructed backwards along the same transition chain.
Thus, starting from some initial annular region,
one can construct  sequences of  correctly
aligned windows, forward and backwards,  along infinitely many topological transition chains
interspersed with annular regions. The Shadowing Lemma-type of result Theorem
\ref{lem:shadowing2} implies the existence of an orbit that gets $\eps_i$-close to each
transition torus $T_i$, and also follows each Aubry-Mather set $\Sigma_i$ for a prescribed time
$n_i$.

\bibliography{diff}
\bibliographystyle{plain}
\end{document}